\documentclass[10pt]{article} 
\usepackage[left=2cm,right=2cm,top=2.54cm,bottom=2.54cm]{geometry}
\usepackage[T1]{fontenc}
\usepackage[latin1]{inputenc}
\usepackage{array}
\usepackage{url}
\usepackage{caption}
\usepackage{subcaption}
\usepackage{bbm}
\usepackage{enumitem}   
\usepackage{mathtools}
\usepackage{stmaryrd}

\usepackage{amsthm, amsmath, amsfonts, mathtools, amssymb} 
 
\DeclareMathOperator*{\argmin}{arg \ min} 
\DeclareMathOperator*{\djv}{div}

\usepackage{sgame, tikz} 
\usetikzlibrary{trees, calc} 
\usepackage[countmax]{subfloat}
\usepackage{hyperref} 
\usepackage{cases}
\usepackage{float}

\usepackage{listings}
\usepackage{xcolor}
\usepackage{csvsimple}

\definecolor{codegray}{rgb}{0.5,0.5,0.5}
\definecolor{codepurple}{rgb}{0.54, 0.17, 0.89}
\definecolor{codegreen}{rgb}{0.55, 0.71, 0.0}
\definecolor{backcolour}{rgb}{1, 1, 1}
\definecolor{blue(pigment)}{rgb}{0.2, 0.2, 0.6}
\definecolor{ao(english)}{rgb}{0.0, 0.5, 0.0}

\usepackage{scalerel}

\lstdefinestyle{mystyle}{
    backgroundcolor=\color{backcolour},   
    commentstyle=\color{codepurple},
    keywordstyle=\color{codegray},
    numberstyle=\tiny\color{codegray},
    stringstyle=\color{codegreen},
    basicstyle=\ttfamily\footnotesize,
    breakatwhitespace=false,         
    breaklines=true,                 
    captionpos=b,                    
    keepspaces=true,                 
    numbers=left,                    
    numbersep=5pt,                  
    showspaces=false,                
    showstringspaces=false,
    showtabs=false,                  
    tabsize=2
}

\def\mc{\ensuremath\mathcal}

\numberwithin{equation}{section}

\newtheorem{theorem}{Theorem}[section]
\newtheorem{as}[theorem]{Assumption}
\newtheorem{proposition}[theorem]{Proposition}

\newtheorem{lemma}[theorem]{Lemma}
\newtheorem{cor}[theorem]{Corollary}
\newtheorem{rem}[theorem]{Remark}

\newcommand\numberthis{\addtocounter{equation}{1}\tag{\theequation}}
\newcommand{\R}{\mathbb{R}}

\newcommand{\gr}{\text{grad} \ \mathcal}

\newtheoremstyle{break}
  {\topsep}{\topsep}%
  {\itshape}{}%
  {\bfseries}{}%
  {\newline}{}%
\theoremstyle{break}

\newtheorem{bdefinition}{Definition}[section]

\allowdisplaybreaks

\DeclareSymbolFont{symbolsC}{U}{txsyc}{m}{n}
\SetSymbolFont{symbolsC}{bold}{U}{txsyc}{bx}{n}
 \DeclareMathSymbol{\df}{\mathrel}{symbolsC}{"42}
\newcommand{\f}[2]{\frac{#1}{#2}}

\newcommand{\cD}{\mathcal{D}}
\newcommand{\cE}{\mathcal{E}}
\newcommand{\cF}{\mathcal{F}}
\newcommand{\cG}{\mathcal{G}}
\newcommand{\cH}{\mathcal{H}}

\newcommand{\cL}{\mathcal{L}}

\newcommand{\cP}{\mathcal{P}}

\newcommand{\cS}{\mathcal{S}}

\newcommand{\cU}{\mathcal{U}}
\newcommand{\cV}{\mathcal{V}}

\newcommand{\LL}{\mathbb{L}}
\newcommand{\NN}{\mathbb{N}}

\newcommand{\RR}{\mathbb{R}}

\newcommand{\di}{\displaystyle}
\newcommand{\iy}{\infty}

\newcommand{\lt}{\left}
\newcommand{\me}{\medskip}
\newcommand{\na}{\nabla}
\newcommand{\pa}{\partial}
\newcommand{\ri}{\rightarrow}
\newcommand{\rt}{\right}
\newcommand{\sm}{\smallskip}

\newcommand{\fo}{\forall\ }

\newcommand{\lan}{\lt\langle}

\newcommand{\lVe}{\lt\Vert}

\newcommand{\ran}{\rt\rangle}

\newcommand{\rVe}{\rt\Vert}

\newcommand{\st}{\,:\,}

\newcommand{\bq}{\begin{eqnarray*}}
\newcommand{\bqn}[1]{\begin{eqnarray}\label{#1}}
\newcommand{\eq}{\end{eqnarray*}}
\newcommand{\eqn}{\end{eqnarray}}
\newcommand{\wwtbp}{\hfill $\blacksquare$\par\me\noindent}
\newcommand{\thistitlepagestyle}{}
\newcommand{\lin}{\llbracket}
\newcommand{\rin}{\rrbracket}

\newcommand{\ttsim}{\raise.17ex\hbox{$\scriptstyle\mathtt{\sim}$}}
\newcommand{\kh}{\kern .08em}
\newcommand{\lojasiewicz}{\L ojasiewicz}

\newcommand{\cPp}{\cP_+(V)}
\newcommand{\cLi}{\cL_{\mathrm{i}}(V)}
\newcommand{\dd}{\mathrm{d}}

\title{Swarm dynamics for global optimisation on finite sets}
\author{Nhat-Thang Le${}^\dagger$ 
and Laurent Miclo${}^\ddagger$
}
 \date{\vbox{\copy0
}
 }

\begin{document}

\setbox0=\vbox{
\large
\begin{center}
${}^\dagger$Toulouse School of Economics\\
Institut de Mathématiques de Toulouse\\
University of Toulouse\\[2mm]
 ${}^\ddagger$Toulouse School of Economics\\
Institut de Mathématiques de Toulouse\\
CNRS and University of Toulouse
\end{center}
} 
\setbox5=\vbox{
\hbox{lenhat.thang@tse-fr.eu\\[1mm]}
\hbox{miclo@math.cnrs.fr\\[1mm]}
\vskip1mm
\hbox{Toulouse School of Economics,\\}
\hbox{1, Esplanade de l'université,\\}
\hbox{31080 Toulouse cedex 6, France.\\}
\hbox{Institut de Mathématiques de Toulouse,\\}
\hbox{Université Paul Sabatier, 118, route de Narbonne,\\}
\hbox{31062 Toulouse cedex 9, France.\\[1mm]}
}
\maketitle
 \thistitlepagestyle
\abstract{
Consider the global optimisation of a function $U$ defined on a finite set $V$ endowed with an irreducible and reversible Markov generator.
By integration, we  extend $U$ to the set $\cP(V)$ of probability distributions on $V$ and we penalise it with a time-dependent generalised entropy functional.
Endowing $\cP(V)$ with a Maas' Wasserstein-type Riemannian structure, enables us to consider an associated time-inhomogeneous gradient descent algorithm.
There are several ways to interpret this $\cP(V)$-valued dynamical system as the time-marginal laws of a time-inhomogeneous non-linear Markov process taking values in $V$, each of them allowing for interacting particle approximations.
This procedure extends to the discrete framework the continuous state space swarm algorithm approach of Bolte, Miclo and Villeneuve \cite{Bolte}, but here we go further by 
considering more general generalised entropy functionals for which functional inequalities can be proven.
Thus in the full generality of the above finite framework, we  give conditions on the underlying time dependence ensuring the convergence of the algorithm toward laws supported by the set of global minima of $U$.
Numerical simulations illustrate that one has to be careful about the choice of the time-inhomogeneous non-linear Markov process  interpretation.
}
\vfill\null
{\small
}\par
\vskip.3cm
{\small
\textbf{Fundings: }
This work  was supported by the grants ANR-17-EURE-0010 and AFOSR-22IOE016.
}

\pagenumbering{arabic}

\section{Introduction}
The global minimization of a function $U$ given on a set $V$ is in general an important but difficult task. When $V$ is a compact and connected manifold and $U$ is smooth function, a time-inhomogeneous swarm algorithm was proposed in \cite{Bolte} to approach the set of global minimizers. Our purpose here is to deal with discrete optimisation problems and second to go beyond some technical restrictions that have appeared in \cite{Bolte}, in particular concerning some functional inequalities. 

Let us begin by recalling the swarm algorithm presented in \cite{Bolte}. We start by up-lifting through integrations the function $U$ on $V$ to the functional $\mathcal U$ defined on the set $\mathcal P(V)$ of probability measures on $V$ via 
\begin{align*}
    \forall \rho \in \mathcal P(V), \quad \mathcal U(\rho) \df  \int_V U(x) \rho(dx).
\end{align*}
Next, we penalise this functional by a $\varphi$-entropy term. Let $\varphi: \mathbb R_+ \rightarrow \mathbb R_+$ ($\R_+\df  [0,+\infty)$) be a convex function satisfying $\varphi(1) = \varphi'(1) = 0$ and consider the functional
    \begin{equation} \label{functional: H}
        \mathcal H: \mathcal P (V) \ni \rho \mapsto \begin{cases}
            \int \varphi(\rho(x)) \ell(dx)&, \quad \text{when} \quad \rho  \texttt{<<} \ell \\ 
            + \infty &, \quad \text{otherwise}\end{cases}
    \end{equation}
where $\ell$ is the Riemannian probability measure on $V$ and where we denoted in the same way a probability measure and its Radon-Nikodym density with respect to the reference measure $\ell$. 

For any $\beta \geq 0$, seen as an inverse temperature, consider the functional 
\bqn{Ub}
        \mathcal U_\beta \df  \beta \mathcal U + \mathcal H.
\eqn
When $\beta$ is large, the global optimisation of $\mathcal U_\beta /\beta$ on $\mathcal P(V)$ is to some degree equivalent to the global optimisation of $U$ on V. We endow $\mathcal P(V)$, or rather its subset $\mathcal D_+(V)$ consisting of probability measures admitting a positive and smooth density, with the Wasserstein structure. Then under the additional assumption that $\varphi'(0) = - \infty$ and starting from an initial probability measure $\rho_0 \in \mathcal D_+(V)$, we can consider the gradient descent associated to $\mathcal U_\beta$ to come close to the unique stationary probability measure, which is almost concentrated on $\mathcal M(U)$, the set of global minima of $U$. To really concentrate on $\mathcal M(U)$, we have to let $\beta$ depends on time in some way, with in particular $\lim_{t\rightarrow+\infty} \beta_t = + \infty$. The resulting evolution $(\rho_t)_{t \geq 0}$ has in the weak sense the non-linear Markov representation 
    \begin{align} \label{Markov representation}
        \forall t \geq 0, \quad \dot \rho_t  = \rho_t L_{\beta_t, \rho_t},
    \end{align}
where for any $\beta \geq 0$ and $\rho \in \mathcal D_+(V)$, $L_{\beta, \rho}$ is the diffusion generator on $V$ defined by 
    \begin{align*}
        L_{\beta,\rho}[\cdot] = \alpha(\rho) \Delta \cdot - \beta \langle \nabla, \nabla \cdot \rangle,
    \end{align*}
where $\Delta, \langle \cdot, \cdot \rangle$ and $\nabla$ are the Laplace-Beltrami operator, the Riemannian scalar product and the gradient operator, and with 
    \begin{align*}
        \forall r >0, \quad \alpha(r) \df  \frac{1}{r}\int_0^r s \varphi''(s)ds,
    \end{align*}
assuming that $\varphi$ is $\mathcal C^2$ on $(0,+\infty)$.

In \cite{Bolte}, convex functions $\varphi$ of the following forms were considered. For any $m \in \mathbb R \setminus \{0,1 \}$ define $\varphi_m$ by 
\begin{align*}
    \forall r \geq 0, \quad \varphi_m(r) \df  \frac{r^m -1 - m(r-1)}{m(m-1)}.
\end{align*}
Observe that $\varphi_m(0) = + \infty$ for $m<0$ and this situation was not taken into account in \cite{Bolte}.

The function $\varphi_0$ and $\varphi_1$ are obtained as limits (respectively for $m \rightarrow 0$ and $m \rightarrow 1$) and are given by 
 \begin{equation}\label{varphi1}
     \forall r \geq 0, \quad \begin{cases}
         \varphi_0(r) \df  -ln(r) + r -1, \\
         \varphi_1(r) \df  r \ln(r) -r +1,
     \end{cases}
 \end{equation}
 (in particular $\varphi_m(0) = +\infty$ iff $m \leq 0$).

 We deduce a family of convex functions parametrized by $m_1, m_2 \in \mathbbm R$ (respectively controlling the behavior at $0$ and $+\infty$) via 
 \begin{align*}
     \forall r \geq 0,\quad \varphi_{m_1,m_2} = \begin{cases}
         \varphi_{m_1} (r) \quad \text{if} \quad r \in (0,1], \\
         \varphi_{m_2} (r) \quad \text{if} \quad r \in (1, + \infty),
     \end{cases}
 \end{align*}
and note that these functions $\varphi_{m_1,m_2}$ are $\mathcal C^2$ on $(0, +\infty)$.

It was proven in \cite{Bolte} that if $V$ is the circle, if $\varphi : = \varphi_{m,2}$ with $m\in (0, 1/2)$ and if the inverse temperature schedule is given by 
\begin{align*}
    \forall t \geq 0, \quad \beta_t \df  k t^{1/\gamma}, \ \text{with} \ k >0 \ \text{and} \ \gamma = \frac{3(2-m)}{1-2m} \in [6, + \infty),
\end{align*}
then the solution of \eqref{Markov representation} concentrates around the set of global minima of $U$ for large times. We expect that a variant of this result holds for any compact connected manifold $V$, but we restricted to the case of the circle to get the underlying functional inequality.

As mentioned previously, here one of our goals is to transpose the above considerations to the situation of a finite set $V$, in particular to get around the difficulty of the underlying functional inequality. As illustrated by the two papers of Holley and Stroock \cite{HS.} and Holley, Kusuoka and Stroock \cite{HSK}, such inequalities can be easier to obtain in the finite context than in the continuous one. 

Let us describe how the previous objects have to be modified. The compact and connected Riemannian manifold is replaced  by a finite set $V$ endowed with a Markov generator $L\df  (L(x,y))_{x,t \in V}$ plays the role of the Beltrami-Laplacian $\Delta$ (which encapsulates the whole Riemannian structure), so we assume that it is irreducible and reversible with a probability distribution still denoted $\ell: = (\ell(x))_{x \in V}$ (which necessarily gives a positive weight to all points of $V$). Let $\mathcal P (V)$ be the set of probability measures on $V$. To any $\mu: = (\mu(x))_{x \in V} \in \mathcal P(V)$, we associate its density $\rho$ with respect to $\ell$: 
\begin{align}
    \forall x \in V, \quad \rho(x) \df  \frac{\mu(x)}{\ell(x)}.
\end{align}

The set of such densities is denoted by $\mathcal D(V)$, we will often move back and forth between $\mathcal P(V)$ and $\mathcal D(V)$, which somewhat respectively corresponds to probabilist and analyst points of view. 

Similar to \eqref{functional: H}, the functional $\mathcal H$ is given by 
\begin{align*}
    \forall \mu \in \mathcal P(V), \quad \mathcal H(\mu) \df  \sum_{x \in V} \varphi(\rho(x)) \ell(x),
\end{align*}
where $\varphi$ is a convex function as above, except that we furthermore allow $\varphi(0) = +\infty$ (in this case we assume that $\lim_{r \rightarrow 0^+} \varphi(r) = +\infty $).

Given a mapping $U: V \rightarrow \mathbbm R$, as above we can then extend it into the functionals $\mathcal U$ and $\mathcal U_\beta$, for any $\beta \geq 0$, defined on $\mathcal P(V)$.

To go further, we have to endow $\mathcal P(V)$ with a Riemannian structure (with boundary), an ersatz of the Wasserstein distance, to be able to consider gradient descent for $\mathcal U_\beta$. To do so, we follow Erbar and Maas \cite{Maas1}. Choosing a particular metric among those they propose, see the next section for details, and starting from a positive probability $\mu_0$, the gradient descent evolution $(\mu_t)_{t\geq 0}$ satisfies the equation 
\begin{align}\label{Markov inter for inhomo}
    \forall t > 0, \quad \dot \mu_t = \mu_t L_{\beta_t, \rho_t},
\end{align}
(recall that for $t\geq 0$, $\mu_t$ is the probability admitting $\rho_t$ as density with respect to $\ell$), with the mapping 
\begin{align*}
    \mathbbm R_+ \times \mathcal D_+(V) \ni (\beta, \rho) \mapsto L_{\beta, \rho}: = ( L_{\beta, \rho}(x,y))_{x,y \in V} \in \mathcal G(V),
\end{align*}
where $\mathcal L(V)$ is the set of Markov generators on $V$, 
\begin{align*}
    \mathcal D_+(V) : = \{ \rho \in \mathcal D(V): \forall x \in V, \rho(x) > 0 \},
\end{align*}
and where
\begin{align*}
     \forall x \neq y, \quad L_{\beta, \rho} \df  \left( \frac{\rho(y) - \rho(x)}{\rho(x)(\varphi'(\rho(y)) - \varphi' (\rho(x)) }\beta(U(y) - U(x)) + \frac{\rho(y)}{\rho(x)} -1 \right)_-L(x,y),
\end{align*}
where $(x)_- = \max(0,-x)$ and with the convention that $\forall x, y \in V$ such that $\rho(y) = \rho(x)$,
\begin{align*}
    \frac{\rho(y) - \rho(x)}{\varphi'(\rho(y)) - \varphi' (\rho(x))} = \frac{1}{\varphi''(\rho(x))}.
\end{align*}
Since $L_{\beta, \rho}$ is a Markov generator, we don't need to specify it's diagonal entries, they are given by 
\begin{align*}
    \forall x \in V, \quad L_{\beta, \rho}(x,x) = - \sum_{y\in V \setminus \{x \} } L_{\beta, \rho}(x,y).
\end{align*}
Inspired by \eqref{Markov representation}, we then consider time-inhomogeneous inverse temperature schemes $(\beta_t)_{t \geq 0}$ and the associated evolution equations
\begin{align} \label{eq: (4)}
    \forall t > 0, \quad \dot \mu_t = \mu_t L_{\beta_t, \rho_t}, \quad \mu_0 \in \mathcal P_+(V),
\end{align}
where $\mathcal P_+(V) \df  \{ \mu \in \mathcal P(V): \rho \in \mathcal D_+(V) \} $.

The main result of this paper is then: 
    \begin{theorem} \label{theo:1.1}
        For any $m <0$, consider the function $\varphi = \varphi_{m,2}$ as well as the time-inhomogeneous inverse temperature scheme
        \begin{align*}
            \forall t \geq 0, \quad \beta_t = (t_0+t)^{\kappa (m)} -1,
        \end{align*}
        where $t_0  \geq 1 $ and
        \begin{equation}
            \kappa(m) = \frac{-m}{2(1-m)} \in (0, \frac{1}{2}).
        \end{equation}
        For the corresponding \eqref{eq: (4)}, we have 
        \begin{align}
            \lim_{t \rightarrow + \infty} \mu_t[\mathcal M(U)] = 1,
        \end{align}
        where $\mathcal M(U)$ is the set of global minimizers of $U$.
    \end{theorem}
This is a discrete analogue to the corresponding result of \cite{Bolte}, with the improvement that there is no more restriction on the ``geometry" of the energy landscape ($L,\ell, U)$. 

In practice it is often difficult to compute the evolution \eqref{eq: (4)}, so one traditionally resorts to interacting particle approximations. An numerical illustration is given at the end of the paper.

The paper is constructed according to the following plan. In the following section, we recall the metric constructions of Erbar and Maas on $\mathcal P_+(V)$ and our particular choice. In Section 3, we present the details of the adaptation to the finite setting of the program described at the beginning of this introduction, in particular we extend the penalized cost \eqref{Ub} for a specific family of convex functions. In Section 4 we consider the convergence to the equilibrium of the time-homogeneous and non-linear Markov evolution \eqref{Markov inter for inhomo}. Since it is a representation of the gradient descent with respect to $\mathcal U_\beta$, we use this functional as a Lyapunov function and are led to a functional inequality which is investigated in Section 5. The proof of Theorem \ref{theo:1.1} is given in Section 5 by adapting the same approach. The last section contains the numerical illustration. A first appendix explains why the traditional Metropolis algorithm is not included into our framework based on Maas' formalism  \cite{Maas1} and how to extend it.
The second
appendix recalls some facts relative to linear and non-linear Markov samplings.

 \bigskip
 \par\hskip5mm\textbf{\large Acknowledgments:}\par\sm\noindent 
We would particularly like to thank Stéphane Villeneuve for the
discussions we had about this paper.

\section{Riemannian structures on \texorpdfstring{$\mathcal D_+(V)$}{Lg} }
In this section, we revisit certain Riemannian structures on $\mathcal D_+(V)$ and the concept of gradient flow for a smooth functional on $(\mathcal D_+(V), \mathcal W_\theta)$ as introduced in Maas \cite{Maas1}, where $\mathcal W_\theta$ denotes a a Riemannian metric. Throughout the paper, we endow the set $V$ with an irreducible Markov generator $L = (L(x,y))_{x,y \in V}$, i.e., 
    \begin{align*}
       \forall x \neq y,\quad L(x,y) \geq 0 \quad \text{and} \quad \forall x \in V, \quad \sum_{y\in V}L(x,y) = 0.
    \end{align*}
Irreducibility means that for any $x \neq y \in V$, there is a path $x = x_0, x_1,...,x_n =y$ such that $L(x_i,x_{i+1})>0$ for all $i = 0,1,...,n-1$. It is well-known that such a generator possesses a unique positive invariant measure $\ell = (\ell(x))_{x\in V}$ by Perron-Frobenius theorem and we assume further that $\ell$ is reversible for $L$, i.e., 
    \begin{align*}
        \ell(x) L(x,y) = \ell(y) L(y,x) \quad \forall x,y \in V.
    \end{align*}
\subsection{Geometric notions}

\begin{bdefinition}[Discrete gradient and divergence]
    For any function $\psi \in \R^V$, the discrete gradient of $\psi$, denoted as $\nabla \psi$, is defined by
        \begin{equation}
            \nabla \psi: V \times V \rightarrow \R, \qquad \nabla \psi(x,y) \df  \psi(y) - \psi(x).
        \end{equation}
    For any function $\Psi \in \R^{V \times V}$, the discrete divergence of $\Psi$, denoted as $\djv \Psi$, is defined by
        \begin{equation}
            \djv \Psi: S \rightarrow \R, \qquad \djv \Psi(x) \df  \frac{1}{2} \sum_{y \in V}L(x,y)(\Psi(x,y) - \Psi(y,x)).
        \end{equation}
    \end{bdefinition}
    \begin{bdefinition}[Inner products]
        For $\phi, \psi \in \R^V$, the inner product with respect to $\ell$ is defined by
            \begin{equation}
                \langle \phi, \psi \rangle_{\mathbb L^2(\ell)} \df  \sum_{x\in V} \ell(x)\phi(x) \psi(x) .
            \end{equation}
        For $\Phi, \Psi \in \R^{V\times V}$, the inner product with respect to $\ell \ltimes L$ is defined by
            \begin{equation}
                \langle \Phi, \Psi \rangle_{\ell \ltimes L}  \df  \frac{1}{2}\sum_{\substack{x,y \in V \\ x \neq y}} \ell(x) L(x,y) \Phi(x,y) \Psi(x,y).
            \end{equation}
    \end{bdefinition}
\noindent From the definitions provided above, it can be readily verified that the ``integration by parts" formula holds.
    \begin{equation}
        \langle \nabla \psi, \Phi\rangle_{\ell \ltimes L} = - \langle\psi, \djv \Phi \rangle_{\mathbb L^2(\ell)}.
    \end{equation}
Another crucial notion is the definition of tangent spaces over $\rho \in \mathcal D_+(V)$, which serves as a fundamental component for the Riemannian structures on on $\mathcal D_+(V)$.
\begin{bdefinition}[Tangent space and inner product]
    Let $\rho \in \mathcal D_+(V)$, the tangent space over $\rho$ is defined by
        \begin{align}
            T_\rho \df  \{ \nabla \psi \in \R^{S\times S}: \psi \in \R^S \}.
        \end{align}
    Note that $T_\rho$ does not depend on $\rho$, but we endow it with a inner product that does:
        \begin{align}
           \forall \nabla \phi, \nabla \psi \in T_\rho, \qquad \langle \nabla\phi, \nabla\psi \rangle_\rho \df  \frac{1}{2}\sum_{x,y \in V} \nabla \phi(x,y) \nabla \psi(x,y) L(x,y)\theta(\rho(x),\rho(y))\ell(x),
        \end{align}
    where $\theta: \R_+ \times \R_+ \rightarrow \R_+$ is a suitable nonnegative function of two variables, chosen carefully in Maas \cite{Maas1} to make $T_\rho$ a Hilbert space. We will give more details on the function $\theta$ in the next section.
\end{bdefinition}

\subsection{Maas' metric}

In \cite{Maas1}, Maas introduced a notion of Wasserstein-like metric on $\mathcal P(V)$, for which he closely followed Brenier-Benamou's interpretation of the 2-Wasserstein metric on $\mathcal P_2(\mathbb R^n)$ in Benamou and Brenier \cite{Brenier}, the space of probability measures on $\mathbb R^n$ with finite second moment. Initially, Maas used a Markov kernel to define the metric, but later in Erbar and Maas \cite{Erbar}, they replaced the Markov kernel with a Markov generator, which is the one presented here. The key aspect is that Maas employed a function of two variables $\theta: \R_+ \times \R_+ \rightarrow \R_+$ satisfying the following collection of assumptions.

\begin{as} \label{as: one theta}
    The function $\theta: \R_+ \times \R_+  \rightarrow \R_+$ satisfies
    \begin{itemize}
        \item (A1): $\theta$ is continuous and is symmetric on $\R_+ \times \R_+$, i.e., $\theta(s,t) = \theta (t,s), \ \forall s,t \geq 0$.
        \item (A2): $\theta$ is $C^\infty$ on $(0,+\infty) \times (0,+\infty)$.
        \item (A3): $\theta (s,t) > 0, \ \forall s,t >0$, and vanishes at the boundary: $\theta(0,t) = 0, \quad \forall t \geq 0$.
        \item (A4): $\theta(r,t) \leq \theta(s,t)$, for all $ 0 \leq r \leq s$ and $t\geq 0$.
        \item (A5): For any $T>0$, there exists a constant $C_T > 0$ such that $ \theta(2s,2t) \leq 2 C_T \theta (s,t)$, whenever $ s,t \leq T$.
    \end{itemize}
\end{as}
    \begin{bdefinition}[Maas metric]
        Let $\theta$ be a function as described in Assumption \ref{as: one theta}. For $\bar \rho_0, \bar \rho_1 \in \mathcal D(V)$, we set 
            \begin{align} \label{eq: def of Maas metric}
                \mathcal W_\theta^2(\bar \rho_0,\bar \rho_1) &\df  \inf_{\rho,\psi}\left\{ \frac{1}{2} \int^1_0 \sum_{x,y\in V} \ell(x)L(x,y) \theta(\rho_t(x),\rho_t(y))(\nabla\psi_t(x,y))^2\, dt\right\}=\inf_{\rho,\psi}\left\{  \int^1_0 \| \nabla \psi_t \|^2_{\rho_t} dt\right\},
            \end{align}
        where the infimum runs over all pairs $(\rho,\psi)$ such that $\rho: [0,1] \rightarrow \mathcal D(V)$ is a piecewise $C^1$ curve in $\mathcal D(V)$ and $\psi: [0,1] \rightarrow \R^S$ is a measurable function, the pair satisfies, for a.e. $t\in[0,1]$,
            \begin{align} \label{eq: 2.9}
                \begin{cases}
                \dot \rho_t(x) + \sum_{y \in V}\nabla \psi_t(x,y)L(x,y)\theta(\rho_t(x), \rho_t(y)) = 0, \ \forall x \in V, \\
                \rho_0 = \bar \rho_0, \ \rho_1 = \bar \rho_1.
                \end{cases}
            \end{align}
    \end{bdefinition}
    
    We have the following summarized result by Maas, the proof of which can be found in the proofs of Theorems 3.12, 3.19, and Lemma 3.30 in Maas \cite{Maas1}. 
    \begin{theorem} \label{maas: metric}
        Suppose that 
        \begin{align} \label{C_theta}
            C_\theta \df  \int^1_0 \frac{1}{\sqrt{\theta(1-r,1+r)}} dr < +\infty
        \end{align}
        then $\mathcal W_\theta$ is a metric on $\mathcal P(V)$. Additionally, if $\theta$ is concave, then for any $\bar \rho_0, \bar \rho_1 \in \mathcal D_+(V)$, we can restrict the set in the infimum in \eqref{eq: def of Maas metric} to curves $\rho = (\rho_t)_{t \in [0,1]} \subset \mathcal D_+(V)$. As a consequence, $(\mathcal D_+(V),\mathcal W_\theta)$ is a Riemannian manifold (i.e., $\mathcal W_\theta$ can be intepreted as a Riemannian distance).
    \end{theorem}
\subsection{Gradient flows of functionals}
\begin{bdefinition}[Tangent vector field of a curve] \label{def: tangent vt field}
Let $\rho = (\rho_t)_{t\geq 0} \subset \mathcal D_+(V)$ be a smooth curve. The tangent vector field along $\rho$ is denoted by $D_t\rho \in T_{\rho_t}$. At any time $t \geq 0$, $D_t\rho$ is the unique element $\nabla g_t$ of $T_{\rho_t}$ such  that 
    \begin{equation}
        \dot \rho_t + \djv(\widehat \rho_t \odot \nabla g_t) = 0,
    \end{equation}
where $\widehat \rho_t(x,y) \df  \theta(\rho_t(x),\rho_t(y))$ and the notation $\odot$ represents the entrywise product, i.e., if $H, K \in \mathbb R^{V\times V}$ then $H \odot K \df  (H(x,y)K(x,y))_{x,y \in V}$.
\end{bdefinition}
In view of Maas \cite{Maas1}, we shall consider two special types of functionals: 
    \begin{itemize}
        \item For a function $R: \mathcal S \rightarrow \R$ we consider the \textit{potential energy functional} $\mathcal V_1:  \mathcal D_+(V) \rightarrow \R$ defined by 
            \begin{equation} \label{functional: potential}
                \mathcal V_1(\rho)\df  \sum_{x\in V}R(x) \rho(x)\ell(x).
            \end{equation}
        \item For a differentiable function $f: (0,+\infty) \rightarrow \R$, we consider the \textit{generalized entropy} $\mathcal V_2:  \mathcal D_+(V) \rightarrow \R$ defined by 
            \begin{equation} \label{functional: general}
                \mathcal V_2(\rho) \df  \sum_{x\in V} f(\rho(x))\ell(x).
            \end{equation}
    \end{itemize}
\begin{bdefinition}[Gradient of a smooth functional]
    The gradient of a smooth functional $\mathcal{V}: \mathcal D_+(V) \rightarrow \R$ at $\rho \in \mathcal D_+(V)$ with respect to the metric $\mathcal W_\theta$, denoted by $\gr V$, is the unique element of $T_{\rho}$ such that, for any smooth curve $(\rho_t)_{t \in(-\epsilon , \epsilon)} \subset \mathcal D_+(V)$ with $\rho_0 = \rho$,
        \begin{equation*}
            \normalfont \frac{d}{dt} \mathcal V(\rho_t)\Big|_{t=0} = \langle \gr V(\rho) ,D_t\rho |_{t=0}\rangle_{\rho \ltimes L}.
        \end{equation*}
\end{bdefinition}

\begin{theorem} \label{grad: of 2 functionals}
    For the functionals $\mathcal V_1$ and $\mathcal V_2$ introduced in   \eqref{functional: potential}, \eqref{functional: general}, their gradients at $\rho \in \mathcal D_+(V)$ are
        \begin{align*}
            \normalfont \gr V_1(\rho) = \nabla R, \qquad 
            \gr V_2(\rho) = \nabla [f' \circ \rho].
        \end{align*}
\end{theorem}
\begin{proof}
    See proofs of Propositions 4.1 and 4.2 in Maas \cite{Maas1}.
\end{proof}

\begin{bdefinition}[Gradient flow] \label{def: Gradient flow}
    Given a metric $\mathcal W_\theta$, a smooth curve $\rho = (\rho_t)_{t\geq 0} \subset \mathcal D_+(V)$ is called a gradient flow of a functional $\mathcal{V}$ if 
    \begin{equation*}
        \normalfont D_t\rho = - \gr V(\rho_t), \quad \forall t \geq 0.
    \end{equation*}
\end{bdefinition}

In particular, given a functional of the form  $\mathcal V \df  \beta \mathcal V_1  + \mathcal V_2, \ \beta \geq 0$, Theorem \ref{grad: of 2 functionals} gives 
    \begin{align*}
        \gr V(\rho)  = \beta \nabla R + \nabla [f'\circ\rho].
    \end{align*}
Such an example of the functional $\mathcal V$ is the penalized cost 
    \begin{align*}
        \mathcal U_\beta(\rho) \df  \beta \sum_{x \in V} U(x)\rho(x)\ell(x) + \sum_{x \in V} \varphi(\rho(x)) \ell(x), \quad \beta \geq 0, 
    \end{align*}
where $\varphi \in C^2(0,+\infty)$ is strictly convex. We will study in detail this functional together with its associated gradient flow in the next section.

\section{Presentation of the problem}
\subsection{The choice of family of relaxations on \texorpdfstring{$\mathcal P(V)$}{Lg}}
Let $V$ be the finite set mentioned in the introduction and $U: V \rightarrow \mathbb R$ be a function on $V$. As previously mentioned in the introduction, our goal is to minimize the function $U$ over $V$. To achieve this, we first up-lift through integration the function $U$ on $V$ to the functional $\mc U$ defined on the set $\mathcal P(V)$  of probability measures on $V$ via
    \begin{align*}
        \forall \mu \in \mathcal P(V), \qquad \mathcal U(\mu) \df  \sum_{x \in V} U(x) \mu(x). 
    \end{align*}
Recall that in the previous section, we endowed $V$ with an irreducible Markov generator $L$. Its positive invariant probability measure $\ell$ allows us to identify each $\mu \in \mathcal P(V)$ with its density with respect to $\ell$:
    \begin{align*}
        \forall x \in V, \qquad \rho(x) \df  \frac{\mu(x)}{\ell(x)}.
    \end{align*}
Therefore, we will often write $\mathcal U(\rho)$ instead of $\mathcal U(\mu)$:
    \begin{align}
        \forall \rho \in \mathcal D(V), \qquad \mathcal U(\rho) \df  \sum_{x \in V} U(x) \rho(x) \ell(x).
    \end{align}
We turn to the choice of the $\mathcal C^2$, convex function $\varphi$ mentioned in the introduction that we are going to use throughout this paper. Let $m<0$ be a negative real number, define the function $\varphi: (0,+\infty) \rightarrow \mathbb R_+$ as follows
    \begin{align} \label{func: varphi}
        \forall r > 0, \qquad \varphi(r) \df  \varphi_{m,2}(r) =\begin{cases}
            \cfrac{r^m-1-m(r-1)}{m(m-1)}&, \quad r\in (0,1) \\ 
            \cfrac{(r-1)^2}{2} &, \quad r \in [1,+\infty)
        \end{cases}
    \end{align}
It can be easily verified that $\varphi$ is  $C^2$ with its first derivative given by 
    \begin{align*}
        \forall r > 0, \qquad \varphi'(r)  =\begin{cases}
            \cfrac{r^{m-1}-1}{m-1}&, \quad r\in (0,1) \\ 
            r-1 &, \quad r \in [1,+\infty)
        \end{cases}
    \end{align*}
and its second derivative is given by
    \begin{align*}
        \forall r > 0, \qquad \varphi''(r)  =\begin{cases}
            r^{m-2}&, \quad r\in (0,1) \\ 
            1 &, \quad r \in [1,+\infty)
        \end{cases}
    \end{align*}
so that $\lim_{r\rightarrow 0^+} \varphi(r) = +\infty$, $\varphi(1) = \varphi'(1)=0$. The second derivative $\varphi''$ is decreasing on $(0,+\infty)$ and $\forall r >0,\ \varphi''(r) \geq 1$, implying that its first derivative $\varphi'$ is strictly increasing and concave. Consequently, $\varphi$ is strictly convex. Additionally, we have $\lim_{r \rightarrow 0^+}\varphi'(r) = -\infty$ and $\varphi'(0,+\infty) = \mathbb R$, thus $\varphi'$ has an inverse $(\varphi')^{-1}: \R \rightarrow (0,+\infty)$ which we denote by $g= (\varphi')^{-1}$. A standard result in real analysis shows that the function $g: \R \rightarrow (0,+\infty)$ is strictly positive and increasing, with the first derivative 
    \begin{align*}
        g'(x) = \frac{1}{\varphi''(g(x))} \in (0,1].
    \end{align*}
We can think of the function $\varphi$ defined in \eqref{func: varphi} as a family of convex functions indexed by $m <0$. The negativity of $m$ forces $\lim_{r \rightarrow 0+}\varphi(r) = +\infty$, which will be crucial in the proofs of existence, uniqueness and convergence theorems of the gradient flows associated with the penalized cost $\mathcal U_\beta$ (given in the next subsection) and its time-dependent version later. We would like to emphasize that from now on, whenever we write $\varphi$, we implicitly refer to the one defined in \eqref{func: varphi} unless explicitly stated otherwise.

\subsection{Penalized cost functional and stationary measure}

Using function $\varphi$ in \eqref{func: varphi}, we define the a $\varphi$-entropy term $\mathcal H$ by 
    \begin{align*}
        \forall \rho \in \mathcal D(V), \qquad \mathcal H(\rho)\df  \sum_{ x \in V}\varphi(\rho(x)) \ell(x),
    \end{align*}
and use this term to penalize $\mathcal U(\rho)$. For $\beta \geq 0$, seen as an inverse temperature, consider the penalized-cost functional
    \begin{align*}
        \forall \rho \in \mathcal D(V), \qquad \mathcal U_\beta(\rho) &\df  \beta \mathcal U(\rho) + \mathcal H (\rho) \\
        &= \beta \sum_{x \in V}U(x)\rho(x)\ell(x) + \sum_{x\in V} \varphi(\rho(x))\ell(x).
    \end{align*}
From the choice of $\varphi$ in \eqref{func: varphi}, we have the following result.

\begin{theorem} \label{theo: the unique minizer rhobeta}
    Let $\beta \geq 0$, and $\varphi$ be given in \eqref{func: varphi}. The functional $\rho \mapsto \mathcal U_\beta (\rho)$ is strictly convex and admits a unique minimizer $\eta_\beta \in \mathcal D_+(V)$.
\end{theorem} 
    \begin{proof}
        The functional $\rho \mapsto \mathcal U_\beta (\rho)$ is strictly convex because it is a sum of a linear functional and a strictly convex functional. Hence, if such a minimizer $\eta_\beta$ exists, it is necessarily unique. To show the existence of $\eta_\beta \in \mathcal D(V)$, we define $\eta_\beta$ to be the (unique) solution of the equation
        \begin{align} \label{rho_beta}
            \rho \in \mathcal D_+(V), \quad \forall x \in V, \qquad \beta U(x) + \varphi'(\rho(x)) = c(\beta),
        \end{align}
        where $c(\beta)$ solves the equation 
        \begin{align} \label{C constant}
            c \in \mathbb R, \qquad \sum_{x \in V} \ell(x) g(c-\beta U(x)) = 1, 
        \end{align}
        with $g = (\varphi')^{-1}$, the inverse function of $\varphi'$. To show $\eta_\beta$ is well-defined, consider
        \begin{align*}
            f: \mathbb R \rightarrow (0,+\infty), \qquad c \mapsto  f(c) = \sum_{x \in V} \ell(x) g(c-\beta U(x)).
        \end{align*}
        The first derivative of $f$ is 
        \begin{align*}
            f'(c) = \sum_{x \in V}  \frac{\ell(x)}{\varphi''(g(c-\beta U(x)))} >0, \qquad (\text{since} \ \varphi'' >0)
        \end{align*}
        and note that $\lim_{c\rightarrow -\infty} f(c) = 0$,  $\lim_{c\rightarrow +\infty} f(c) = +\infty$, so \eqref{C constant} has a unique solution $c(\beta)$. If we let $\forall x \in V, \ \eta_\beta(x) \df  g(c(\beta)-\beta U(x))$ then $\eta_\beta >0$ and satisfies $\sum_{x \in V} \ell(x) \eta_\beta(x) = 1$, thus $\eta_\beta \in \mathcal D_+(V)$. Also, $\eta_\beta$ satisfies \eqref{rho_beta}, i.e., $\beta U(x) + \varphi'(\eta_\beta(x)) = c(\beta)$. Finally, to show $\eta_\beta$ is the unique minimizer, write
        \begin{align*}
            \mathcal U_\beta(\rho) - \mathcal U_\beta(\eta_\beta) &= \sum_{x \in V} \beta U(x)(\rho(x) -\eta_\beta(x))\ell(x) + \sum_{x \in V}\Big(\varphi(\rho(x)) - \varphi(\eta_\beta(x))\Big) \ell(x) \\
            &= \sum_{x \in V} (-\varphi'(\eta_\beta(x)) + c(\beta))(\rho(x) -\eta_\beta(x))\ell(x)  + \sum_{x \in V}\Big(\varphi(\rho(x)) - \varphi(\eta_\beta(x))\Big) \ell(x)\\
            &= -\sum_{x \in V} \varphi'(\eta_\beta(x))(\rho(x) -\eta_\beta(x))\ell(x)  + \sum_{x \in V}\Big(\varphi(\rho(x)) - \varphi(\eta_\beta(x))\Big) \ell(x)
            \\
            &= \sum_{x \in V} \ell(x) \Big(\varphi(\rho(x)) - \varphi(\eta_\beta(x)) - \varphi'(\eta_\beta(x))(\rho(x) -\eta_\beta(x))  \Big) \\
            &\geq 0,
        \end{align*}
    where the last inequality follows from the convexity of $\varphi$. 
    \end{proof}
The next result shows the limiting behavior of $\eta_\beta$ as $\beta \rightarrow +\infty$. Recall that $\mc M(U)$ is the set of global minimizers of $U$.
    \begin{theorem} \label{theo: properties of rho beta}
        Under the hypothesis of Theorem \ref{theo: the unique minizer rhobeta}, the following statements hold
        \begin{enumerate}[label=(\roman*)]
        \item $\forall x,y \in V$ such that $U(x) = U(y)$, $\eta_\beta (x) = \eta_\beta(y)$ (in particular $x,y \in \mathcal M(U)$).
        \item $\forall x \notin \mathcal M(U)$, $\lim_{\beta \rightarrow +\infty }\eta_\beta (x) = 0$. More precisely, we have 
            \begin{align*}
                \forall x \notin \mathcal M(U), \qquad \lim_{\beta \rightarrow +\infty} \beta [\eta_\beta(x)]^{1-m} = \frac{1}{(1-m)(U(x) - \min U)},
            \end{align*}
        where $m <0$ is fixed in \eqref{func: varphi}.
        \item $\forall x \in \mathcal M(U)$, $\eta_\beta(x) \geq 1$ and $\cfrac{\partial \eta_\beta(x)}{\partial\beta} \geq 0$, where the equality holds iff $U$ is constant on $V$.
        \item $\forall x \in \mathcal M(U)$, $\lim_{\beta \rightarrow +\infty}\eta_\beta (x) = \cfrac{1}{\sum_{y \in \mathcal M(U)} \ell(y)}$. In other words, if $\zeta_\beta$ is the probability measure with density $\eta_\beta$ and $ V \ni x \mapsto 1_{\mathcal M(U)}(x)$ is the function on $V$ which equals $1$ if $x \in \mathcal M(U)$ and 0 otherwise, then $\lim_{\beta \rightarrow +\infty} \zeta_\beta (x) = \zeta_\infty(x)$, where
        \begin{align} \label{mu_infty}
            \forall x \in V, \qquad \zeta_\infty(x) \df  \frac{ \ell(x)}{\sum_{y \in \mathcal M(U)} \ell(y)} 1_{\mathcal M(U)}(x).
        \end{align}
    \end{enumerate}
    \end{theorem}
    \begin{proof}
        The first statement $(i)$ is a consequence of \eqref{rho_beta}. For $(ii)$, we let $x \notin \mathcal M(U)$ and $x_0 \in \mathcal M(U)$ then from \eqref{rho_beta} and the fact that $\rho(x_0) \ell(x_0) \leq 1$,
        \begin{align*}
            \varphi'(\eta_\beta(x))  &= -\beta(U(x) - \min U)) + \varphi'(\eta_\beta(x_0)) \\
            &\leq -\beta(U(x) - \min U)) + \varphi'(\ell(x_0)^{-1}) \\
            & \longrightarrow -\infty \qquad \text{as} \quad \beta \rightarrow +\infty,
        \end{align*}
        which in turn gives $\eta_\beta(x) \rightarrow 0$ as $\beta \rightarrow + \infty$ and $\eta_\beta(x_0)$ is bounded away from 0 by a positive constant (in fact the lower bound is 1 by statement $(iii)$, which we will show later). So for $\beta$ big enough, such that $\eta_\beta(x) < 1$, we have 
        \begin{align*}
            -\beta(U(x) - \min U)) + \varphi'(\eta_\beta(x_0)) =  \varphi'(\eta_\beta(x)) = \frac{[\eta_\beta(x)]^{m-1}-1}{m-1},
        \end{align*}
        or
        \begin{align*}
            \beta [\eta_\beta(x)]^{1-m} &= \Big( (1-m)(U(x) - \min U) + \frac{(m-1)\varphi'(\eta_\beta(x_0)) +1}{\beta} \Big)^{-1} \\
            &\longrightarrow \frac{1}{(1-m)(U(x) - \min U)} \quad \text{as} \quad \beta \rightarrow +\infty.
        \end{align*}
        Let us prove $(iii)$. As in the proof of $(ii)$
        \begin{align*}
            \varphi'(\eta_\beta(x)) - \varphi'(\eta_\beta(x_0))   &= -\beta(U(x) - \min U) < 0,
        \end{align*}
        so $\eta_\beta(x) < \eta_\beta(x_0)$ by monotonicity of $\varphi'$. Since $x_0 \in \mathcal M(U)$ and $ x \notin \mathcal M(U)$ are arbitrary, from $(i)$,
        \begin{align*}
            \eta_\beta(x_0) &= \eta_\beta(x_0)\sum_{y \in \mathcal M(U)} \ell(y) + \eta_\beta(x_0)\sum_{y \notin \mathcal M(U)} \ell(y) \\
            &=  \sum_{y \in \mathcal M(U)}\eta_\beta(y) \ell(y)  + \eta_\beta(x_0)\sum_{y \notin \mathcal M(U)} \ell(y) \\
            &= 1 + \sum_{y \notin \mathcal M(U)}( \eta_\beta(x_0)- \eta_\beta(y)) \ell(y)\\
            &\geq 1.
        \end{align*}
        Consider the function of two variables
        \begin{align*}
            \R_+ \times \R \ni (\beta,c) \mapsto Z(\beta,c) \df  \sum_{x \in V} g(c - \beta U(x)) \ell(x), \qquad g = (\varphi')^{-1}.
        \end{align*}
        From the proof of Theorem \ref{theo: the unique minizer rhobeta}, for each $\beta \geq 0$ there exists a constant $c(\beta)$ such that $Z(\beta,c(\beta))=1$. By the implicit function theorem, 
        \begin{align*}
            \frac{\partial c}{\partial \beta} &= - \frac{\partial Z}{\partial \beta} \left( \frac{\partial Z}{\partial c}\right)^{-1} \\
            &= \frac{\sum_{x \in V} g'(c - \beta U(x)) U(x) \ell(x) }{\sum_{x \in V} g'(c - \beta U(x)) \ell(x)}\\
            &\geq \min U.
        \end{align*}
    The last inequality follows from the fact that $g' > 0$. Note that $ \eta_\beta (x_0) 
    =g(c(\beta) - \beta U(x_0))$, hence
        \begin{align*}
            \frac{\partial \eta_\beta (x_0)}{ \partial \beta} = g'(c(\beta) - \beta U(x_0)) \left( \frac{\partial c}{\partial \beta} - \min U \right) \geq 0,
        \end{align*}
        where the last inequality becomes equality if and only if $U$ is constant on $V$. 
    Finally, from $(i)$, $(ii)$ and $(iii)$ we deduce 
        \begin{align*}
            \eta_\beta(x_0) = \frac{1 - \sum_{y \notin \mathcal M(U)}\eta_\beta(y)\ell(y)}{\sum_{y \in \mathcal M(U)}\ell(y)} \overset{\beta \rightarrow +\infty}{\longrightarrow} \frac{1 }{\sum_{y \in \mathcal M(U)}\ell(y)},
        \end{align*}
    which establishes $(iv)$ and finishes the proof.
    \end{proof}
    
    \begin{rem}
        (a) The limit measure \eqref{mu_infty} is independent of the choice of $m <0$ in the definition of $\varphi = \varphi_{m,2}$, it only 
        depends on the invariant measure $\ell= (\ell(x))_{x \in V}$.
        \par\smallskip
        (b) In the theory of simulated algorithms on finite state spaces, see e.g.\ Holley and Stroock \cite{HS.}, one also gets a convergence in law toward a measure such as 
        \eqref{mu_infty}, where $\ell$ is the reversible measure of the underlying exploration kernel. It appears as the small temperature limit $1/\beta\rightarrow 0_+$ of the 
        Gibbs distribution given by
                    \begin{align*}
                        \forall\ x\in V,\qquad       \zeta_\beta(x)\df \frac{\ell(x) e^{-\beta U(x)}}{ \sum_{y\in V} \ell(y) e^{-\beta U(y)}}
                    \end{align*}
        Note that the latter corresponds to our stationary measure when we take $\varphi=\varphi_1$  in \eqref{varphi1}.      
        But the time-inhomogeneous evolution of probability measures $(\rho_t(x)\, \ell(dx))_{t\geq 0}$ we will construct in Section \ref{S5}
        do not correspond to the evolution of the Simulated Annealing   algorithm considered in Holley and Stroock \cite{HS.}.
        This a departure  with   the continuous space situation of
        of \cite{Bolte}, where we recover the Simulated Annealing algorithm of Holley, Kusuoka and Stroock \cite{HSK} by taking $\varphi=\varphi_1$.
        For other classical references to the Simulated Annealing algorithms on finite states space,  see \cite{Laarhoven, Aarts}, where $\ell$ is rather taken to be the uniform distribution on $V$.

    \end{rem}

\section{The time-homogeneous situation} \label{Section: homo}
\subsection{Gradient flow of \texorpdfstring{$\mathcal U_\beta$}{Lg} }
We define the function $\theta: \R_+ \times \R_+ \rightarrow \R_+$ by 
    \begin{align} \label{func: theta our}
        \theta(s,t)\df  \begin{cases}
            \frac{s-t}{\varphi'(s) - \varphi'(t)}, \quad &\forall s,t >0, \ t\neq s \\
            \frac{1}{\varphi''(s)}, \quad &s = t>0
            \\
            0, \quad &t=0 \ \text{or} \ s = 0.
        \end{cases}
    \end{align}
It will be shown below that $\theta$ satisfies Assumption \ref{as: one theta} except (A2) because $\theta$ is not differentiable (e.g., at (1,1)). Even without the smoothness of $\theta$, suppose that $C_\theta < \infty$ in \eqref{C_theta} then $(\mathcal D(V), \mathcal W_\theta)$ remains a complete metric space by a slight modification of the proof of Theorem \ref{maas: metric} in Maas \cite{Maas1} (to the best of our understanding, Maas did not rely on the smoothness of $\theta$ to prove $\mathcal W_\theta$ is a metric on $\mathcal D(V)$). However, $(\mathcal D_+(V), \mathcal W_\theta)$ ceases to be a Riemannian manifold due to the lack of smoothness in $\theta$. Since we will not investigate curvature notions in this paper, smoothness of the Riemannian metric $\langle \cdot, \cdot \rangle_\rho$ is not needed for our purposes. The continuity of $\theta$ will suffice for a notion of $C^1$ gradient flow dynamic of the penalized cost $\mathcal U_\beta$. We give some properties of the function $\theta$ defined in \eqref{func: theta our}.
        \begin{lemma} \label{theta: properties 1}
            The function $\theta$ in \eqref{func: theta our} satisfies Assumption \eqref{as: one theta} except (A2). Moreover, we have the following inequality
                \begin{align} \label{varphi: an ineq crucial for funcineq}
                    \forall s,t > 0, \qquad (t-s)(\varphi'(t) - \varphi'(s)) \geq \varphi(t) - \varphi(s) - \varphi'(s)(t-s).
                \end{align}
        \end{lemma}
        \begin{proof}
            From the definition of $\theta$, (A1) and (A3) follow from the fact that $\varphi \in C^2$, $\lim_{r\rightarrow 0^+} \varphi'(r) = - \infty$ and $\lim_{r\rightarrow 0^+} \varphi''(r) = +\infty$. To prove (A4), we note that $\varphi'$ is concave because $\varphi''$ is positive and nonincreasing. For fixed $r \leq s$, consider the function $k$ given on $t \in [0, r] \cup [s, +\infty)$ by
                \begin{align*} \label{analysis: theta}
                  \forall t \in [0, r] \cup [s, +\infty), \qquad k(t) \df (s-t)(\varphi'(r) - \varphi'(t)) - (r-t)(\varphi'(s) - \varphi'(t)), \qquad  .
                \end{align*}
            Its first derivative is
                \begin{align*}
                    k'(t) = \varphi'(s) - \varphi'(r) - \varphi''(t) (s-r).
                \end{align*}
            If $t \in  [0, r]$, then $\varphi''(t) \geq \varphi''(r)$ (recall that $\varphi''$ is decreasing) and
                \begin{align*}
                    k'(t) \leq \varphi'(s) - \varphi'(r) - \varphi''(r) (s-r) \leq 0
                \end{align*}
            by the concavity of $\varphi'$. Similarly, if $t \geq s$, then $\varphi''(s) \geq \varphi''(t)$ and
                \begin{align*}
                    k'(t) &\geq \varphi'(s) - \varphi'(r) - \varphi''(s) (s-r) \\
                    & = - (\varphi'(r) - \varphi'(s) - \varphi''(s) (r-s)) \\
                    &\geq 0
                \end{align*}
            again by the concavity of $\varphi'$. We obtain $t \mapsto k(t)$ is decreasing on $[0,r]$ and increasing on $[s,\infty)$. Hence,  $k(t) \geq k(r) = 0$ on $[0, r]$ and $ k(t) \geq k(s) = 0 $ on $[s,+\infty)$. Thus, for $ t\in [0,r) \cup (s,+\infty)$,
                \begin{align}
                    \theta(s,t) - \theta (r,t) = \frac{(s-t)(\varphi'(r) - \varphi'(t)) - (r-t)(\varphi'(s) - \varphi'(t))}{(\varphi'(s) - \varphi'(t))(\varphi'(r) - \varphi'(t))} \geq 0
                \end{align}
            because the numerator is just $k(t) \geq 0$ and the denominator is always positive due to the fact that $\varphi'$ is strictly increasing on $(0,+\infty)$. If $ t \in (r,s)$,
                \begin{align*}
                    \theta(s,t) - \theta (t,t) = \frac{\varphi'(s) - \varphi'(t) - \varphi''(t)(s-t)}{(\varphi'(t) - \varphi'(s))\varphi''(t)} \geq 0
                \end{align*}
            since the numerator and the denominator are both negative because $\varphi'$ is concave and strictly increasing. In the same manner, $\theta(t,t) \geq \theta(r,t)$ and hence $ \theta(s,t) \geq \theta(r,t)$.
            Since $r,s,t$ are arbitrary, thus we have established (A4). For the last inequality \eqref{varphi: an ineq crucial for funcineq}, we rewrite it as 
                \begin{align*}
                      \varphi(s) - \varphi(t) -\varphi'(t)(s-t) \geq 0,
                \end{align*}
            but this is obviously true from the convexity of $\varphi$.
        \end{proof}

        \begin{lemma} \label{func: theta our is locally Lipchizt}
            The function $\theta$ given in \eqref{func: theta our} is locally Lipschitz on $(0,+\infty) \times (0,+\infty)$, i.e., $\forall (s_0,t_0) \in (0,+\infty) \times (0,+\infty) $, there exist $\delta > 0$ and $K = K(s_0,t_0, \delta)$ such that 
            \begin{align*}
            \forall (s',t'), (s'',t'') \in B((s_0,t_0), \delta), \qquad 
                |\theta(s',t') - \theta(s'',t'')| \leq K \| (s',t')  -(s'',t'')\|_\infty,
            \end{align*}
            where $ \| (s',t')  -(s'',t'')\|_\infty\df  \max \{ |s'-s''|, |t'-t''| \}$ and $ B((s_0,t_0), \delta)\df  \{(s,t): \| (s,t)  -(s_0,t_0)\|_\infty < \delta \}$.
        \end{lemma}
        \begin{proof}
            We first prove that the function $ (0,+\infty) \ni r \mapsto \varphi''(r)$ is Lipschitz on any compact intervals. It is clear that $\varphi''$ is continuously differentiable on any compact interval that does not contains 1 and thus Lipschitz continuous there. Suppose $[a,b]$ is an interval that contains 1, and let $s,t \in [a,b] $ be such that $ a \leq s \leq 1 \leq t \leq b$ then,
            \begin{align*}
                \left|\varphi''(s) - \varphi''(t) \right| &= | s^{m-2} - 1 | =  \frac{1-s^{2-m}}{s^{2-m}} \\
                &\leq  \frac{2-m}{a^{2-m}} (1-s) \\
                & \leq \frac{2-m}{a^{2-m}} (t-s),
            \end{align*}
            and so $\varphi''$ is Lipschitz continuous on compacts that contain 1 too. Now we fixed a point $(s_0,t_0) \in (0,+\infty) \times (0,+\infty) $ and let $\delta >0$ be such that $\min \{ s_0 - \delta, t_0 -\delta \} > 0$. Consider two points $(s',t'), (s'',t'') \in  B((s_0,t_0), \delta)$. By symmetry of $\theta$, we assume without loss of generality $s_0 \leq t_0$, $ s' < t'$, $s'' < t''$ and $t' \leq t''$. Note that $s',t',s'',t'' \in [s_0-\delta, t_0 + \delta]$. If $ t' \leq s''$, then by the monotonicity in each variable of $\theta$ in Lemma \ref{theta: properties 1},  $\theta(s'',t'') \geq \theta(t',t'') \geq \theta(t',s') $. So
            \begin{align*}
                0 \leq \theta(s'',t'')  - \theta (s',t') &\leq \theta(t'',t'')  - \theta(s',s') \\
                &=  \left|\frac{1}{\varphi''(t'')} - \frac{1}{\varphi''(s')} \right| \\
                &\leq \left|\varphi''(t'') - \varphi''(s') \right| \quad (\text{since} \ \varphi'' \geq 1) \\
                &\leq  K(s_0,t_0,\delta)(t'' - s') \\ 
                &=  K(s_0,t_0,\delta)( t'' - t' +  s''-s' +t' -s'' ) \\
                &\leq  2K(s_0,t_0,\delta)  \|(s',t') - (s'',t'')\|_\infty,
            \end{align*}
            where $ K(s_0,t_0,\delta) $ is the Lipschitz constant of $ \varphi''$ on $[s_0-\delta, t_0 + \delta]$. If $t' > s''$, let $x \in  [s_0-\delta, t_0 + \delta]$ and consider the  function  $ [x,t_0 + \delta] \ni r \mapsto p_x(r) \df  \theta (r,x)$. It's derivative is given by 
            \begin{align*}
               \forall r  \in (x,t_0 + \delta), \qquad p'_x(r) = \frac{-(\varphi'(x) - \varphi'(r) - \varphi''(r)(x-r))}{(\varphi'(x) - \varphi'(r))^2} \geq 0.
            \end{align*}
            We will prove that $p'_x$ is bounded by some constant depending on $s_0,t_0$ and $\delta$. Indeed, by the mean value theorem, there is $r^* \in (x,r)$ such that $ \varphi'(x) - \varphi'(r) = (x-r) \varphi''(r^*) $, hence 
            \begin{align*}
                p_x'(r) &= \frac{(r-x)(\varphi''(r^*) - \varphi''(r))}{(r-x)^2 \varphi''(r^*)^2} \\
                &\leq \frac{\varphi''(r^*) - \varphi''(r)}{r-x} \qquad (\text{since} \ \varphi'' \geq 1) \\
                &\leq  \frac{\varphi''(x) - \varphi''(r)}{r-x} \qquad (r^* >x, \ \varphi'' \ \text{is decreasing})
                \\
                &\leq K(s_0,t_0, \delta) \quad \ \qquad (x,r \in [s_0-\delta, t_0 + \delta]). \numberthis \label{442}
            \end{align*}
            Similarly, we have the same estimate for the derivative of the function $ [s_0-\delta, y] \ni r \mapsto h_y(r)\df  \theta(y,r)$, i.e., 
            \begin{align} \label{443}
                h_y'(r) \leq K(s_0,t_0, \delta)
            \end{align}
            Hence, we have 
            \begin{align*}
                |\theta(s',t') - \theta(s'',t'')| &\leq |\theta(s',t') - \theta(s'',t')| + |\theta(s'',t') - \theta(s'',t'')| \\
                &\leq  \left| \int_{s'}^{s''} h'_{t'}(r)dr \right| +  \left| \int_{t'}^{t''} p'_{s''}(r)dr \right| \quad (\text{set} \ x = s'', y = t' \ \text{in } \eqref{442}, \eqref{443}) \\
                &\leq K(s_0,t_0, \delta)(|s'-s''| + |t'-t''|) \\
                &\leq 2 K(s_0,t_0, \delta) \| (s',t') - (s'',t'') \|_\infty.
            \end{align*}
            which finishes the proof.
        \end{proof} 
    
Following Definition \ref{def: Gradient flow}, but restricting to $C^1$ curves, we define the gradient flow of $\mathcal U_\beta(\rho)$ is any $C^1$ curve $(\rho_t)_{t\geq 0} \subset \mathcal D_+(V)$ satisfying
    \begin{align} \label{gradient flow on tangent}
        \forall t>0, \qquad D_t\rho = - \gr U_\beta(\rho) = - (\beta\nabla U + \nabla [\varphi'\circ \rho]),
    \end{align}
where the second equation follows from Theorem \ref{grad: of 2 functionals}. We can rewrite \eqref{gradient flow on tangent} in terms of each coordinate as 
    \begin{align} \label{ODE: beta fixed}
       \forall x \in V, \ \forall t>0, \qquad \dot \rho_t(x) = \sum_{y \in V} L(x,y) \theta(\rho_t(x),\rho_t(y))\Big(\beta\nabla U(x,y) + \nabla[\varphi' \circ \rho_t](x,y)\Big),
    \end{align}
with the notation $\nabla[\varphi' \circ \rho_t](x,y) = \varphi'(\rho_t(y)) - \varphi'(\rho_t(x))$. We shall often write the last quantity in \eqref{ODE: beta fixed} compactly as $\nabla[\beta U + \varphi' \circ \rho] (x,y)$. The existence and uniqueness of \eqref{ODE: beta fixed} is guaranteed by the following proposition.
    \begin{proposition} \label{prop: e&u of GF of Ub}
        For any initial condition $\rho_0 \in \mathcal D_+(V) $, there is a unique solution $ (\rho_t)_{t\geq 0} \subset \mathcal D_+(V)$, which exists for all time $t \geq 0$ and satisfies \eqref{ODE: beta fixed}.
    \end{proposition}
    \begin{proof}
    For notational convenience, we define the functional 
    \begin{align*}
        \forall \rho \in \mathcal D_+(V), \quad \forall x \in V, \qquad \mathcal F(\rho)(x) \df  \sum_{y \in V} L(x,y) \theta(\rho(x),\rho(y))\Big(\nabla[\beta U + \varphi' \circ \rho](x,y)\Big),
    \end{align*}
        then the map $\rho \mapsto \mathcal F(\rho)$ is clearly continuous and is also locally Lipschitz. Indeed, fix $\rho^* \in \mathcal D_+(V)$, and consider $\delta >0$ such that $B(\rho^*, \delta)\df  \{ \rho \in \mathcal D(V): \| \rho - \rho^* \|_\infty \df  \max_{x\in V} |\rho(x) - \rho^*(x)|< \delta \}$ lies in $\mathcal D_+(V)$. Let $\rho_1,\rho_2 \in B(\rho^*, \delta)$, and fix one $x \in V$, denote $\| L \|_\infty \df  \max_{y \in V} |L(y,y)|$, $\text{osc }U := \max U - \min U$. Due to the particular choice of $\theta$, for $\rho \in \mc D_+(V)$,
        \begin{align*}
            \sum_{y \in V} L(x,y)\theta(\rho(x),\rho(y)) \Big(\nabla [\varphi'\circ \rho](x,y)\Big) = \sum_{y \in V} L(x,y) (\rho(y) - \rho(x)) =  \sum_{y \in V} L(x,y) \rho(y),
        \end{align*}
        hence
        \begin{align*}
            \frac{1}{\| L \|_\infty }|\mathcal F(\rho_1)(x) - &\mathcal F(\rho_2)(x)| \leq \beta \text{osc }U\sum_{y \in V} \Big|\theta(\rho_1(x), \rho_1(y)) - \theta(\rho_2(x), \rho_2(y)) \Big| + \sum_{y \in V} |\rho_1(y) - \rho_2(y) | \\
            &\leq K\beta \text{osc }U \sum_{y \in V} \max \{ |\rho_1(x)- \rho_2(x)|, |\rho_1(y)- \rho_2(y)| \}  + \sum_{y \in V} |\rho_1(y) - \rho_2(y)| \\
            & \leq (K\beta \text{osc }U+1) |V| \| \rho_1 - \rho_2 \|_\infty,
        \end{align*}
        where $|V|$ is the cardinality of $V$, $K$ is the local Lipschitz constant of $\theta$ in Lemma \ref{func: theta our is locally Lipchizt}. Therefore, there is a constant $K'$ such that 
        \begin{align*}
            \forall \rho_1, \rho_2 \in B(\rho^*,\delta), \qquad |\mathcal F(\rho_1)- \mathcal F(\rho_2)|_\infty \leq K' \| \rho_1 - \rho_2\|_\infty.
        \end{align*}
        By Cauchy-Picard and extensibility theorems, we have the existence and uniqueness of a solution $(\rho_t)_{t\in [0,T)}$ on some maximal interval $[0,T)$. By differentiating in time the function $t\mapsto \mathcal U_\beta(\rho_t)$,  
        \begin{align*}
            \frac{d}{dt}\mc U_\beta(\rho_t) &= \sum_{x\in V}  \Big(\beta U(x) + \varphi'(\rho_t(x))\Big) \dot \rho_t(x)  \ell(x) \\
            &=  \sum_{x\in V} \Big(\beta U(x) + \varphi'(\rho_t(x))\Big) \ell(x) \sum_{y \in V}L(x,y) \theta (\rho_t(x), \rho_t(y)) \Big(\nabla[\beta U + \varphi' \circ \rho_t](x,y) \Big) \\
            &= \sum_{x\in V}  \sum_{y \in V} \Big( \beta U(x) + \varphi'(\rho_t(x))\Big) \ell(x) L(x,y) \theta (\rho_t(x), \rho_t(y)) \Big( \nabla[\beta U + \varphi' \circ \rho_t](x,y) \Big) \\
            &= \sum_{y\in V}  \sum_{x \in V} \Big( \beta U(y) + \varphi'(\rho_t(y))\Big) \ell(y) L(y,x) \theta (\rho_t(y), \rho_t(x)) \Big(\nabla[\beta U + \varphi' \circ \rho_t](y,x) \Big) \\
            &= \sum_{y\in V}  \sum_{x \in V} \Big( \beta U(y) + \varphi'(\rho_t(y))\Big) \ell(x) L(x,y) \theta (\rho_t(x), \rho_t(y)) \Big(\nabla[\beta U + \varphi' \circ \rho_t](y,x) \Big) \\
            &= \frac{1}{2} \sum_{y\in V}  \sum_{x \in V} \Big(\nabla[\beta U + \varphi' \circ \rho_t](x,y))\Big) \ell(x) L(x,y) \theta (\rho_t(x), \rho_t(y)) \Big(\nabla[\beta U + \varphi' \circ \rho_t](y,x)) \Big) \\
            &=- \frac{1}{2} \sum_{y\in V}  \sum_{x \in V}\ell(x) L(x,y) \theta (\rho_t(x), \rho_t(y)) \Big(\nabla[\beta U + \varphi' \circ \rho_t](x,y) \Big)^2 \label{computation for G(rho)} \numberthis \\
            &\leq 0,
        \end{align*}
        so $\mathcal U_\beta$ is decreasing along $(\rho_t)_{t\in [0,T)}$. If $T< +\infty$ then $\rho_t$ must go to the boundary of $\mathcal D_+(V)$, as $t \rightarrow T^-$. However, this is not possible since $ \rho \mapsto \mathcal U_\beta (\rho)$ explodes at the boundary by the fact that $\lim_{r \rightarrow 0+}\varphi(r) = +\infty$. Therefore, $T = +\infty$ and the solution exists for all time $t \geq 0$.
    \end{proof}

\subsection{A functional inequality for \texorpdfstring{$\beta \geq 0$}{Lg} fixed}
    \begin{proposition}\label{FuncIneq: beta fixed}
        Let $\beta \geq 0$, we consider the functionals
        \begin{align*}
            \mathcal I(\beta, \rho)&\df  \mathcal U_{\beta}(\rho) -  \mathcal U_{\beta}(\eta_\beta)\\
            &= \sum_{x \in V} \ell(x)\Big( \varphi(\rho(x)) - \varphi(\eta_\beta(x)) - \varphi'(\eta_\beta(x))(\rho(x)-\eta_\beta(x)) \Big), \numberthis \label{functional: I} \\ 
            \mathcal G(\beta,\rho) &\df  \frac{1}{2}\sum_{x \in V}\sum_{y\in V} \ell(x) L(x,y) \theta(\rho(x),\rho(y)) \Big( \beta\nabla U(x,y) + \nabla [\varphi' \circ \rho](x,y)\Big)^2. \numberthis \label{functional: G}
        \end{align*}
        Then
        \begin{align*}
            \chi (\beta) \df  \inf_{\rho \in \mathcal D_+(V)\setminus \{ \eta_\beta \}} \frac{\mc G (\beta,\rho)}{\mc I (\beta,\rho)} > 0.
        \end{align*}
    \end{proposition}
    
    Before giving the proof, we need the following lemmas.
    
    \begin{lemma}\label{lem: FuncIneq for beta fixed}
        Fix $\beta \geq 0$ For any sequence $( \rho_n)_{n\geq 1} \subset \mathcal D_+(V)$ such that $\|  \rho_n -\eta_\beta \|_\infty \rightarrow 0$, we have 
            \begin{align*}
                \liminf_{n \rightarrow +\infty} \frac{\mc G(\beta,  \rho_n)}{\mc I(\beta,  \rho_n)} > 0.
            \end{align*}
    \end{lemma}
    
    \begin{proof}
        Assume the statement in Lemma \ref{lem: FuncIneq for beta fixed} does not hold, then there exists a subsequence, still denoted by $( \rho_n)_{n\geq 1}$ that 
            \begin{equation*}
                \lim_{n \rightarrow +\infty} \frac{\mc G(\beta, \rho_n)}{\mc I(\beta,  \rho_n)} = 0.
            \end{equation*}
        Denote $a_n = \|  \rho_n - \eta_\beta \|_\infty$, and $h_n = \frac{ \rho_n - \eta_\beta}{a_n}$, so $\| h_n\|_\infty = 1$. Recall from \eqref{rho_beta} that $\varphi'(\eta_\beta) + \beta U$ is constant, which gives $\beta\nabla U(x,y) = - \nabla [\varphi'\circ \eta_\beta](x,y)$. We have, 
            \begin{align*}
                \frac{\mc G(\beta, \rho_n)}{a_n^2} &= \frac{1}{2a_n^2}\sum_{x, y\in V}\ell(x)L(x,y)\theta( \rho_n(x),  \rho_n(y))\Big(\varphi'( \rho_n(x)) -\varphi'(\eta_\beta(x)) - [\varphi'( \rho_n(y)) - \varphi'(\eta_\beta(y)) ]\Big)^2
            \end{align*}
        and by mean value theorem, for some vector $\lambda_n \in [0,1]^V$, we have
            \begin{align*}
                \frac{1}{a_n^2}\Big(\varphi'&( \rho_n(x)) -\varphi'(\eta_\beta(x)) - (\varphi'( \rho_n(y)) - \varphi'(\eta_\beta(y)) )\Big)^2 \\
                &= \frac{1}{a_n^2}\Big(\varphi'(\eta_\beta(x) +a_nh_n(x)) -\varphi'(\eta_\beta(x)) - (\varphi'(\eta_\beta(y) + a_nh_n(y)) - \varphi'(\eta_\beta(y)) )\Big)^2 \\
                &= \Big( \varphi''\Big(\eta_\beta(x) + a_n \lambda_n(x)h_n(x)\Big)h_n(x)  - \varphi''\Big(\eta_\beta(y) + a_n \lambda_n(y)h_n(y)\Big)h_n(y) \Big)^2.
            \end{align*}
        Also, by Taylor's theorem for second derivative (see e.g. Wolfe \cite{Wolfe}), for some vector $\lambda'_n \in [0,1]^V$, we have
            \begin{align}
                \mc I(\beta,  \rho_n) = \frac{1}{2}\sum_{x\in V}\ell(x) \varphi''\Big(\eta_\beta(x) + a_n \lambda'_n(x) h_n(x)\Big) a_n^2 h_n^2(x),
            \end{align}
        or  
            \begin{align*}
                \frac{\mc I(\beta, \rho_n)}{a_n^2} = \frac{1}{2}\sum_{x\in V}\ell(x) \varphi''\Big(\eta_\beta(x) + a_n \lambda'_n(x) h_n(x)\Big) h_n^2(x).
            \end{align*}
        Observe that the set $\mc H_{0,1}: = \{ h: \ell[h] = 0, \| h\|_\infty = 1\}$ is compact so there is a subsequence $(h_{n_k})_{k \geq 1}$ converging to $h_0 \in \mc H_{0,1}$. Furthermore, because $a_{n_k} \rightarrow 0$, we have
            \begin{align}
                \lim_{k\rightarrow +\infty}\frac{\mc I(\beta,  \rho_{n_k})}{a_{n_k}^2} &= \frac{1}{2}\sum_{x\in V}\ell(x) \varphi''(\eta_\beta(x))  h_0^2(x) > 0, \\
                \lim_{k\rightarrow +\infty} \frac{\mc G(\beta, \rho_{n_k})}{a_{n_k}^2} &= \frac{1}{2}\sum_{x,y}\ell(x)L(x,y)\theta(\eta_\beta(x),\eta_\beta(y)) \Big( \varphi''(\eta_\beta(x))h_0(x) - \varphi''(\eta_\beta(y))h_0(y)\Big)^2.
            \end{align}
        Since $ \lim_{n\rightarrow +\infty} \frac{\mc G(\beta, \rho_n)}{\mc I(\beta, \rho_n)} = 0$, the subsequence $\Big(\frac{G(\beta,\rho_{n_k})}{\mc I(\beta, \rho_{n_k})}\Big)_{k \geq 1}$ converges to 0 and hence
            \begin{equation*}
                \sum_{x,y \in V} \ell(x) L(x,y) \theta( \eta_\beta(x),\eta_\beta(y) ) \Big( \varphi''(\eta_\beta(x)) h_0(x) - \varphi''(\eta_\beta(y)) h_0(y) \Big)^2 =0.
            \end{equation*}
        This, together with irreducibility of $L$ and $\eta_\beta \in \mathcal D_+(V)$ imply $ \forall x \in V, \ \varphi''(\eta_\beta(x))h_0(x) = C$ for some constant $C$. If $C = 0$, then $h_0 = 0 \notin \mc H_{0,1}$ (because $\varphi'' >0$), which is impossible. If $C \neq 0 $ then $\ell[h_0] \neq 0$ since the sign of $h$ is equal to that of $C$, a contradiction.
    \end{proof}
    
    \begin{lemma} \label{lemma:funcineq difficult lemma}
        Denote
        \begin{align*}
            \partial \mc D_+(V)\df  \mc D(V) \setminus \mc D_+(V) = \{ \rho \in \mc D(V): \rho(x) = 0 \ \text{for some} \ x \in V \}.
        \end{align*}
        Let $\rho_* \in \partial \mc D_+(V)$ and $( \rho_n)_{n\geq 0} \subset \mc D_+(V)$ be a sequence such that $ \rho_n \rightarrow \rho_* $, i.e. $\|  \rho_n -\rho \|_\infty \rightarrow 0$, then
        \begin{align*}
            \lim_{n \rightarrow 0} \frac{\mc G(\beta, \rho_n)}{\mc I(\beta, \rho_n)} = + \infty.
        \end{align*}
    \end{lemma}
    \begin{proof}
        For any $f \in \R^V$ we will use the notations $f_\wedge \df  \min_{x\in V} f(x)$, $f_\vee \df  \max_{x\in V}f(x)$, $\{ \rho_* = 0 \} \df  \{ x\in V: \rho_*(x) = 0 \}$ and we call a state $x \in V$ a minimizer of $\rho \in \mc D(V)$ if $\rho(x)= \rho_\wedge$. Since the set $\{ \rho_* = 0 \}$ is finite and $ \rho_n \rightarrow \rho_*$,  there is $N>0$ such that for all $n\geq N$, the minimizers of $ \rho_n$ must lie in $\{ \rho_* = 0 \}$. Up to taking a subsequence, we can assume that for all $n\geq 0$, $ \rho_n$ admits some $x_0 \in \{ \rho_* = 0 \}$ as a minimizer, i.e., $ \rho_n(x_0) = \rho_{n,\wedge}$ .
        
        For every $\rho \in \mc D_+(V)$ set $M_\rho\df  \{ x \in V: \rho(x) \leq 1 \}$. We have
            \begin{align*}
                \forall \rho \in \mc D_+(V),\quad \mc I (\beta, \rho) &= \sum_{x \in V}\ell(x)\rho(x) \beta U(x) + \sum_{x \in V} \varphi(\rho(x))\ell(x) - \mc U_\beta(\eta_\beta) \label{I(beta,rho)} \numberthis \\
                &\leq \beta \max U - \mc U_\beta(\eta_\beta) + \ell_\vee \sum_{x \in V} \varphi(\rho(x)) \\
                &=\beta \max U - \mc U_\beta(\eta_\beta) + \ell_\vee \Big( \sum_{x \in M_\rho} \varphi(\rho(x)) + \sum_{x \notin M_\rho} \varphi(\rho(x)) \Big)\\ 
                &\leq \beta \max U - \mc U_\beta(\eta_\beta) + \ell_\vee \Big( \sum_{x \in M_\rho} \varphi(\rho_\wedge) + \sum_{x \notin M_\rho} \varphi(1/\ell_\wedge) \Big) \\
                &\leq \beta \max U - \mc U_\beta(\eta_\beta) + \ell_\vee |V| \Big(\varphi(\rho_\wedge)+ \varphi(1/\ell_\wedge)\Big),
            \end{align*}
    where $|V|$ is the cardinality of $V$ and we have used the fact that $\ell(x) \leq \ell_\vee$, $\rho(x) \leq 1/\ell(x) \leq 1/\ell_\wedge$, $\varphi$ is decreasing on $(0,1)$ and increasing on $[1,\infty)$. In short, there is a constant $K >0$ such that $\mc I(\beta,\rho) \leq K(\varphi(\rho_\wedge) + \frac{1}{m(m-1)}), \ \forall \rho \in \mc D_+(V)$. Also, from $\eqref{I(beta,rho)}$, we have
        \begin{align*}
            \forall \rho \in \mc D_+(V), \quad \mc I (\beta, \rho) &\geq \beta \min U - \mc U_\beta(\eta_\beta) + \ell_\wedge \sum_{x \in V} \varphi(\rho(x)) \\
            &\geq \beta\min U - \mc U_\beta(\eta_\beta) + \ell_\wedge\varphi(\rho_\wedge),
        \end{align*}
    thus if $ \rho_n \rightarrow \rho^* \in \partial \mc D_+(V)$ then $ \mc I(\beta,  \rho_n) \rightarrow +\infty$ (recall $\lim_{r\rightarrow 0^+} \varphi(r)= +\infty$). Next, we decompose $\mc G(\beta, \rho) = \mathcal G_1 (\beta,\rho)  + \mc G_2(\rho)$, where
        \begin{align*}
            \mathcal G_1 (\beta,\rho)&\df  \frac{1}{2}\sum_{x,y\in V} \ell(x)L(x,y) \theta(\rho(x),\rho(y)) \Big( [\beta\nabla U(x,y)]^2 + 2\beta \nabla U(x,y) \nabla [\varphi' \circ \rho](x,y) \Big) \\
            \mathcal G_2 (\rho)&\df  \frac{1}{2}\sum_{x,y\in V} \ell(x)L(x,y) \theta(\rho(x),\rho(y))  \Big(\nabla [\varphi' \circ \rho](x,y)\Big)^2.
        \end{align*}
    Observe that $\forall t,s >0$, $\theta(s,t) (\varphi'(t) - \varphi'(s)) = t-s$ and $\forall x \neq y \in V$, $|\rho(y)-\rho(x)| \leq \frac{2}{\ell_\wedge}$, $|U(y) - U(x)| \leq \text{osc } U \df  \max U - \min U $. We define
        \begin{align*}
            \mathcal G_1 (\beta,\rho) &= \frac{1}{2}\sum_{x \neq y\in V} \ell(x)L(x,y) \theta(\rho(x),\rho(y)) \Big( [\beta\nabla U(x,y)]^2 + 2\beta \nabla U(x,y) \nabla [\varphi' \circ \rho](x,y) \Big) \\
            &\leq \frac{1}{2} \ell_\vee \| L \|_\infty \sum_{x\neq y \in V} \theta (1/\ell_\wedge, 1/\ell_\wedge) \beta^2 (\text{osc } U)^2 +  \beta\sum_{x \neq y\in V} \ell(x)L(x,y) (\rho(y) -\rho(x))(U(y)-U(x)) \\
            &\leq \frac{1}{2}\ell_\vee \| L \|_\infty |V|^2 \beta^2 (\text{osc } U)^2 + \beta \text{osc } U \sum_{x\neq y \in V} \ell_\vee \| L\|_\infty \frac{2}{\ell_\wedge} \\
            &\leq  \frac{1}{2}\ell_\vee \| L \|_\infty |V|^2 \beta^2 (\text{osc } U)^2 + \beta \text{osc } U |V|^2 \ell_\vee \| L\|_\infty \frac{2}{\ell_\wedge}, 
        \end{align*}
    where $\| L \|_\infty = \max_{x\in V} |L(x,x)|$. Hence, $\mc G_1$ is bounded above for all $\rho \in \mc D_+(V)$ and since $ \rho_n \rightarrow \rho_* \in \partial \mc D_+(V)$, we have $\cfrac{\mc G_1(\beta, \rho_n)}{\mc I(\beta,  \rho_n)} \rightarrow 0$. Observe that $\rho_{n,\wedge} \leq 1$, so (recall that $m <0$)
        \begin{align*}
             \varphi(\rho_{n,\wedge}) = \varphi(  \rho_n(x_0)) = \frac{ \rho_n(x_0)^{m} - 1 -m( \rho_n(x_0) -1)}{m(m-1)} \leq \frac{ \rho_n(x_0)^{m} - 1}{m(m-1)}
        \end{align*}
    Therefore, the conclusion of the lemma will follow if we can show $\frac{\mc G_2(\beta,  \rho_n)}{\rho_{n,\wedge}^m} \rightarrow +\infty$ because 
        \begin{align*}
            \frac{\mc G_2(  \rho_n)}{\mc I(\beta,  \rho_n)} \geq \frac{\mc G_2(  \rho_n)}{K(\varphi (  \rho_{n,\wedge}) +\frac{1}{m(m-1)})} \geq \frac{m(m-1)}{K}  \frac{\mc G_2( \rho_n)}{( \rho_{n,\wedge})^m}.
        \end{align*}
    By the irreducibility of $L$, we can find $y_0 \notin \{ \rho_* = 0\}$ and a path $x_0,x_1,...,x_q = y_0$ in a way such that $x_i \in \{\rho_* = 0\}$ for all $i =0,1,...,q-1$ and $L(x_0,x_1)L(x_1,x_2)...L(x_{q-1},x_q) >0$. Denote 
        \begin{align*}
            \Upsilon\df  \min \{\ell(x)L(x,y): x,y \in V \ \text{such that} \ \ell(x)L(x,y)>0 \}, 
        \end{align*}
    then
        \begin{align*}
            \forall \rho \in \mc D_+(V), \qquad  \mc G_2(\rho) &= \frac{1}{2}\sum_{x,y \in V} \ell(x) L(x,y) \theta (\rho(x),\rho(y)) \Big( \varphi'(\rho(y)) - \varphi'(\rho(x)) \Big)^2 \\ 
            &= \frac{1}{2}\sum_{x,y \in V} \ell(x) L(x,y) (\rho(y)-\rho(x))\Big( \varphi'(\rho(y)) - \varphi'(\rho(x)) \Big)   \\
            &\geq \Upsilon \sum_{i=0}^{q-1} (\rho(x_{i+1})-\rho(x_i))\Big( \varphi'(\rho(x_{i+1})) - \varphi'(\rho(x_i)) \Big). \numberthis \label{expression 0}
        \end{align*}
    Suppose for now that $\rho(x_i) \leq 1 \ \forall i=0,1,...,q-1$ and denote for $i = 1,...,q$, $t_i \df  \frac{\rho(x_{i})}{\rho(x_{i-1})}, r_i= t_1...t_i $,  $r_0 \df 1$  then from \eqref{expression 0},
        \begin{align*}
            \frac{\mc G_2(\rho)}{\rho(x_0)^m} &\geq \Upsilon \sum_{i=0}^{q-1} \frac{\rho(x_i)^{m}}{\rho(x_0)^{m}}\frac{(\rho(x_{i+1})-\rho(x_i))}{\rho(x_i)}\frac{\Big( \varphi'(\rho(x_{i+1})) - \varphi'(\rho(x_i)) \Big)  }{\rho(x_i)^{m-1}}  \\
            &= \Upsilon  \sum_{i=0}^{q-1} \frac{\rho(x_i)^{m}}{\rho(x_0)^{m}} \frac{(\rho(x_{i+1})-\rho(x_i))}{\rho(x_i)} \frac{\Big( \rho(x_{i+1})^{m-1} -\rho(x_{i})^{m-1} \Big)  }{(m-1)\rho(x_i)^{m-1}} \label{ex for cite} \numberthis \\
            &= \frac{\Upsilon}{1-m} \sum_{i=0}^{q-1} r_i^m(t_{i+1} -1) ( 1 - t_{i+1}^{m-1}). \label{expression} \numberthis
        \end{align*}
    Consider the following family of functions indexed by $k \in [0, +\infty)$
        \begin{align*}
            \forall t >0, \qquad  p(k,t) \df  (t -1)(1-t^{m-1}) + kt^m.
        \end{align*}
    Observe that $p(0,t) \geq 0$, $p(k,t) = p(0,t) + kt^m$ and $\lim_{t\rightarrow + \infty} p(0,t)= + \infty$. For $t \in (0, k^{\frac{1}{1-m}}) $, we have $p(k,t) \geq kt^m \geq k^{\frac{1}{1-m}}$. Let $k \geq 1$, for $t \geq k^{\frac{1}{1-m}} $ we have $1- t^{m-1} \geq 1- 1/k$ and $t-1 \geq k^{\frac{1}{m-1}} - 1 \geq 0$. Thus, we have
        \begin{align*}
            w(k)\df  \min_{t>0} p(k,t) \geq \min \left\{ k^{\frac{1}{1-m}}, (k^{\frac{1}{1-m}} - 1)(1- \frac{1}{k}) \right\} \rightarrow + \infty, \ \text{as} \ k \rightarrow +\infty.
        \end{align*}
    We regard $w$ as a function of $k \in [0,+\infty)$ and introduce the notation $w_j: = w\circ...\circ w$ as the convolution of $j$ times the function $w$ for $j \geq 1$. Note that $\lim_{k \rightarrow +\infty} w_j(k) = +\infty$ for all $j \geq 1$. We deduce that the sum in \eqref{expression}
        \begin{align}
            \sum_{i=0}^{q-1} r_i^m(t_{i+1} -1) ( 1 - t_{i+1}^{m-1}) \geq w_{q-1}(p(0,t_q)). \label{important estimate}
        \end{align}
    For example, if $q = 2$ then the sum in the expression \eqref{expression} is 
        \begin{align*}
            p(0,t_1) +t_1^m p(0,t_2) + t_1^m t_2^m p(0,t_3) &= p(0,t_1) +t_1^m (p(0,t_2) + t_2^m p(0,t_3)) \\
            &= p(0,t_1) +t_1^m p( p(0,t_3), t_2) \\
            &\geq p(0,t_1) +t_1^m w( p(0,t_3)) \\
            &= p(w( p(0,t_3)), t_1) \\
            &\geq w(w( p(0,t_3)))\\
            &= w_2( p(0,t_3)).
        \end{align*}
    Now, we consider three cases: $\rho_*(y_0) < 1$ and $\rho_*(y_0) > 1$ and $\rho_*(y_0) = 1$. If $\rho_*(y_0) < 1$, then for $n$ big enough $ \rho_n(y_0) <1$. Using the expression \eqref{ex for cite} and the estimate \eqref{important estimate} evaluated at $\rho =  \rho_n$, we get
        \begin{align}
             \frac{ \mc G_2(\rho)}{ \rho_n(x_0)^{m}} \geq \frac{\Upsilon}{1-m} w_{q-1} \Big(p(0,\frac{ \rho_n(y_0)}{ \rho_n(x_{q-1})}) \Big) \rightarrow + \infty \label{final estimate....}
        \end{align}
    because $ \rho_n(x_{q-1}) \rightarrow 0$, $ \rho_n(y_0) \rightarrow \rho_*(y_0) > 0$ (recall that $x_{q-1} \in \{ \rho_*=0\}$, $y_0 \notin \{ \rho_* = 0 \}$). If $ \rho_*(y_0) > 1$, for $n$ big enough $ \rho_n(y_0) \geq 1$ then the equality in \eqref{ex for cite} does not hold when evaluated at $\rho =  \rho_n$ because $\varphi'( \rho_n(y_0)) =  \rho_n(y_0) -1$. Instead, the last term $i = q-1$ in the sum in \eqref{ex for cite}, ignoring the multiplicative term $ \left( \frac{ \rho_n(x_{q-1})}{ \rho_n(x_0)} \right)^{m}$, is replaced by
        \begin{align*}
            \left(\frac{ \rho_n(x_{q})}{ \rho_n(x_{q-1})} -1\right) \left(  \rho_n(x_q)  \rho_n(x_{q-1}) ^{1-m}  - \frac{1}{m-1} +  \rho_n(x_{q-1})^{1-m} (\frac{1}{m-1}-1)  \right) =: k_n,
        \end{align*}
    which tends to infinity because the first term $ \frac{ \rho_n(x_{q})}{ \rho_n(x_{q-1})} -1$ tends to infinity while the second term converges to $\frac{1}{1-m} > 0$. Therefore, we can use the inequality \eqref{final estimate....} with $ p(0,\frac{ \rho_n(y_0)}{ \rho_n(x_{q-1})})$ replaced by $k_n \rightarrow \infty$. For the last case $ \rho_*(y_0) = 1$, if $\rho_n(x_0)$ is either eventually less than 1 ($<1$) or eventually no less than 1 ($\geq 1$) then we can use the same arguments just above. If this is not the case, 
    $( \rho_n)_{n\geq 0}$ can be split into two subsequences $(\rho^{(1)}_n)_{n\geq 0}, (\rho^{(2)}_n)_{n\geq 0}$ such that $\ \forall n \geq 0$, $\rho^{(1)}_n(y_0) \geq 1 $ and $ \rho^{(2)}_n(y_0) < 1$ then the same arguments apply for these subsequences. We end our proof of this lemma here.
    \end{proof} 
    \begin{proof}
        (Of Proposition \ref{FuncIneq: beta fixed}) Let $( \rho_n)_{\geq 0}$ be a minimizing sequence. Since $\mc D(V)$ is compact, there exists a subsequence, still denoted by $( \rho_n)_{n\geq 0}$, converging to some $\rho_* \in \mc D(V)$. If $\rho_* = \eta_\beta$ then from Lemma \ref{lem: FuncIneq for beta fixed},
            \begin{align*}
                \chi(\beta) = \lim_{n\rightarrow +\infty} \frac{\mc G (\beta, \rho_n)}{\mc I (\beta, \rho_n)} = \liminf_{n\rightarrow +\infty} \frac{\mc G (\beta, \rho_n)}{\mc I (\beta, \rho_n)} > 0.
            \end{align*}
        If $\rho_* \in \partial \mc D_+(V)$, then Lemma \ref{lemma:funcineq difficult lemma} shows 
            \begin{align*}
                \chi(\beta) = \lim_{n\rightarrow +\infty} \frac{\mc G (\beta, \rho_n)}{\mc I (\beta, \rho_n)} = +\infty,
            \end{align*}
        which is impossible. If $\rho_* \in \mc D_+(V) \setminus \{ \eta_\beta \}$ then clearly $\chi(\beta) = \frac{\mc G(\beta, \rho_*)}{\mc I(\beta, \rho_*) } >0$. Hence, $\chi(\beta) >0$.
    \end{proof}

    \begin{cor} \label{corollary: 1}
        Let $\beta \geq 0$ fixed and $(\rho_t)_{t\geq 0}$ be the associated gradient flow of $\mathcal U_\beta$ with the initial condition $\rho_0 \in \mathcal D_+(V)$, then we have the following inequalities,
            \begin{align} \label{cor: functional ineq for beta fixed}
               \forall t >0, \qquad \frac{1}{2} \| \rho_t - \eta_\beta \|^2_{\mathbb L^2(\ell)} \leq \ \mathcal I(\beta, \rho_t) \leq e^{-\chi(\beta)t} \mathcal I(\beta, \rho_0).
            \end{align}
            As a consequence, $\lim_{t\rightarrow +\infty} \mathcal U_\beta(\rho_t) = \mathcal U_\beta(\eta_\beta)$ and $\lim_{t\rightarrow +\infty}\rho_t =  \eta_\beta$.
    \end{cor}
    \begin{proof}
        By differentiating with respect to time and using the identities in \eqref{functional: I}, \eqref{functional: G},
        \begin{align*}
             \forall t>0, \qquad \partial_t{\mathcal{I}}(\beta, \rho_t)  \df  \frac{\partial}{\partial t} \mathcal I(\beta, \rho_t) &= \frac{\partial}{\partial t} \mathcal U_\beta(\rho_t) \\
             &= -\mathcal G(\beta,\rho_t) \\
             &\leq -\chi(\beta) \mathcal I(\beta, \rho_t),
        \end{align*}
    and by a Gronwall Inequality (see e.g. Pachpatte \cite{Pachpatte}), we have the second inequality in \eqref{cor: functional ineq for beta fixed}. For the first inequality in \eqref{cor: functional ineq for beta fixed}, we use the identity \eqref{functional: I}, together with Taylor's theorem (or mean value theorem) for second derivatives,
        \begin{align*}
            \forall t >0, \qquad \mathcal I(\beta, \rho_t) &= \sum_{x \in V} \ell(x)\Big( \varphi(\rho_t(x)) - \varphi(\eta_\beta(x)) - \varphi'(\eta_\beta(x))(\rho_t(x)-\eta_\beta(x)) \Big) \\
            &= \frac{1}{2}\sum_{x \in V}\ell(x) \varphi''(\lambda_t(x))(\rho_t(x) - \eta_\beta(x))^2\\
            &\geq \frac{1}{2}  \sum_{x \in V}\ell(x) (\rho_t(x) - \eta_\beta(x))^2 \\
            &=  \frac{1}{2}\| \rho_t - \eta_\beta \|_{\mathbb L^2(\ell)}^2,
        \end{align*}
        where in the second equality, $\lambda_t: V \rightarrow \mathbb R^+$ is some vector in the Taylor's theorem satisfying $\lambda_t(x) \in (\rho_t(x)\wedge \eta_\beta(x), \rho_t(x)\vee \eta_\beta(x) ), \ \forall x \in V$, where $x \wedge y \df  \min\{x,y\}$, $x\vee y\df \max\{x,y\}$. 
    \end{proof}
 
    \begin{proposition}[An upper bound for $\chi(\beta)$]
            Define the Markov generator $Q_\beta$ as follows
                \begin{align*}
                    \forall x \neq y \in V, \qquad Q_\beta(x,y) \df  \varphi''(\eta_\beta(x)) L(x,y) \theta(\eta_\beta(x), \eta_\beta(y)).
                \end{align*}
        We can easily check that $Q_\beta$ is irreducible and the probability measure $\ell_\beta = (\ell_\beta(x))_{x\in V}$, defined by
            \begin{align*}
                \forall x \in V, \qquad \ell_\beta(x) \df  \frac{\ell(x) [\varphi''(\eta_\beta(x))]^{-1}}{\sum_{y \in V} \ell(y) [\varphi''(\eta_\beta(y))]^{-1} },
            \end{align*}
        is reversible for $Q_\beta$. Then we have
            \begin{align*}
                \chi(\beta) \leq \lambda(Q_\beta),
            \end{align*}
        where $\lambda(Q_\beta)$ is the spectral gap of $(Q_\beta, \ell_\beta)$.
    \end{proposition}
    \begin{proof}
        Let $h \in \R^V$ be such that $\ell[h] = 0$. Then from the proof of Lemma \ref{lem: FuncIneq for beta fixed}, we have
            \begin{align*}
                \lim_{\epsilon \rightarrow 0} \frac{\mc G(\beta, \eta_\beta + \epsilon h)}{\mc I(\beta, \eta_\beta + \epsilon h)} &= \frac{\sum_{x,y \in V} \ell(x)L(x,y) \theta(\eta_\beta(x), \eta_\beta(y)) \Big( \varphi''(\eta_\beta(x)) h(x) - \varphi''(\eta_\beta(y)) h(y) \Big)^2  }{\sum_{x\in V} \ell(x) \varphi''(\eta_\beta(x)) h^2(x)} \\
                &= \frac{\sum_{x,y \in V} \ell_\beta(x) Q_\beta(x,y) \Big( \varphi''(\eta_\beta(x)) h(x) - \varphi''(\eta_\beta(y)) h(y) \Big)^2 }{ \sum_{x\in V} \ell_\beta(x) [\varphi''(\eta_\beta(x)) h(x)]^2}.
            \end{align*}
        Set $f(x) \df  \varphi''(\eta_\beta(x)) h(x), \ \forall x\in V $. Then 
            \begin{align*}
                \ell_\beta[f] &= \frac{\sum_{x \in V} \ell(x)[\varphi''(\eta_\beta(x))]^{-1} \varphi''(\eta_\beta(x)) h(x) }{ \sum_{x \in V} \ell(x) [\varphi''(\eta_\beta(x))]^{-1}} \\
                &= \frac{\ell[h]}{ \sum_{x \in V} \ell(x) [\varphi''(\eta_\beta(x))]^{-1}} \\
                &= 0.
            \end{align*}
        Hence, 
            \begin{align*}
                \lim_{\epsilon \rightarrow 0} \frac{\mc G(\beta, \eta_\beta + \epsilon h)}{\mc I(\beta, \eta_\beta + \epsilon h)} &= \frac{\sum_{x \in V} \ell_\beta (x) Q_\beta(x,y)(f(y)-f(x))^2}{ \ell_\beta[f^2]} \\
                &= \frac{ -\langle f, Q_\beta[f] \rangle_{\mathbb L^2(\ell_\beta)}}{2 \text{Var}_{\ell_\beta}[f]} = -\frac{\ell_\beta[f Q_\beta[f]] }{ 2\text{Var}_{\ell_\beta}[f] } \\
                &\geq  \lambda(Q_\beta).
            \end{align*}
        As $h$ varies over the set $\{ h \in \R^V: \ell[h] =0 \}$, $f$ takes all values in $\{f \in \R^V: \ell_\beta[f]= 0 \}$: take $f' \in  \{f \in \R^V: \ell_\beta[f]= 0 \}$, define $h'(x)\df  f'(x)[\varphi''(\eta_\beta(x)]^{-1}$ then $\ell[h'] = \ell_\beta[f'] = 0$. By the variational  characterization of $\lambda(Q_\beta)$, 
            \begin{align*}
                \lambda(Q_\beta) &= \inf_{f:f \notin \text{Vect}(1), \ell_\beta[f]= 0} -\frac{\ell_\beta[f Q_\beta[f]] }{ 2\text{Var}_{\ell_\beta}[f] }\\
                &= \inf_{h: \ell[h] = 0}  \lim_{\epsilon \rightarrow 0} \frac{\mc G(\beta, \eta_\beta + \epsilon h)}{\mc I(\beta, \eta_\beta + \epsilon h)} \\
                &\geq \chi(\beta)
            \end{align*}
        by the definition of $\chi(\beta)$, which shows the announced result.
    \end{proof}

\subsection{Nonlinear Markov representation} \label{sub: homo inter}
    For $\beta \geq 0$ fixed, the dynamic \eqref{ODE: beta fixed} can be reinterpreted as a nonlinear Markov dynamic satisfying
        \begin{align}\label{Nonlinear: betafixed}
            \forall t >0, \qquad \dot \mu_t = \mu_t L_{\beta, \rho_t},
        \end{align}
    where $\mu_t \in \mathcal P_+(V)$ is the probability measure on $V$ admitting $\rho_t$ as its density with respect to $\ell$. The generator in \eqref{Nonlinear: betafixed} is given by
        \begin{align} \label{Nonlinear: generator}
            \forall \rho \in \mathcal D_+(V), \ \forall x \neq y, \quad L_{\beta, \rho}(x,y) = L(x,y) \frac{\theta(\rho(x),\rho(y))}{\rho(x)}\Big( \nabla [\beta U + \varphi'\circ \rho] (x,y) \Big)_-,
        \end{align}
    where $\rho$ is the density with respect to $\ell$ of $\mu$ and $x_- = \max\{0,-x\}$. We can write \eqref{Nonlinear: generator} more explicitly as follows
        \begin{align} \label{Nonlinear: generator explicit}
            \forall \rho \in \mathcal D_+(V), \ \forall x \neq y, \ L_{\beta, \rho}(x,y) = L(x,y) \left( \frac{\rho(y) - \rho(x)}{\rho(x)[\varphi'(\rho(y)) - \varphi' (\rho(x))] }\beta(U(y) - U(x)) + \frac{\rho(y)}{\rho(x)} -1  \right)_-.
        \end{align}
    Using the formula \eqref{Nonlinear: generator explicit}, we can therefore simulate a Markov process that has its law satisfying \eqref{Nonlinear: betafixed}. Our algorithm is then an approximation of this Markov process through particle systems. The details can be found in the appendix. We now show that the nonlinear interpretation \ref{Nonlinear: betafixed} holds.
    \begin{theorem} \label{generator L_h-}
    Let $G$ be an irreducible Markov generator on $V$ with reversible measure $\pi >0$ and $H= (H(x,y))_{x,y \in V}$ be a function with the property $\forall x \neq y$, $H(x,y) = - H(y,x)$. Define the divergence operator $ \normalfont \text{div}_G: \R^{V\times V} \rightarrow \R^V$  by
        \begin{align*}
            \normalfont \forall F \in \R^{V\times V}, \forall x \in V \qquad  \text{div}_G[F](x)\df  \frac{1}{2}\sum_{y\in V} G(x,y)(F(x,y) - F(y,x))
        \end{align*}
    Denote $H_{-}(x,y) \df  - \min\{H(x,y),0\}$ and similarly $H_{+}(x,y) \df  \max\{H(x,y),0\}$, then for any test function $f \in \R^V$ we have
            \begin{align}
               \normalfont \pi [f \text{div}_G H] = \pi [G_{H_-}[f]],
            \end{align}
    where $G_{H_-}$ is the Markov generator defined by
            \begin{align}
                \forall x\neq y, \qquad G_{H_-}(x,y) \df  G(x,y)H_-(x,y).
            \end{align}
    \end{theorem}
    \begin{proof}
        We have
            \begin{align*}
                 \pi [f \text{div}_G H] &= \sum_{x \in V} \pi(x) f(x) \text{div}_G H(x) \\
                &=  \sum_{x \in V} \pi(x) f(x) \sum_{y\in V} \frac{1}{2}G(x,y)(H(x,y) - H(y,x)) \\
                &= \sum_{x \neq y \in V} \pi(x) G(x,y) H(x,y) f(x) \\
                &= \sum_{x \neq y \in V} \pi(x) G(x,y) H_+(x,y) f(x) - \sum_{x \neq y \in V} \pi(x) G(x,y) H_-(x,y) f(x).
            \end{align*}
        Recall that $\pi$ is reversible for $G$ and from the assumptions on $H$, it holds that $H_-(y,x) = H_+(x,y)$. We then have
            \begin{align*}
                \pi[f G_{H_-}[f]] &= \sum_{x \in V} \pi(x) G_{H_-}f(x) \\
                &= \sum_{x \in V} \sum_{y \in V}\pi(x) G_{H_-}(x,y)f(y) \\
                &= \sum_{x \neq y \in V}  \pi(x) G_{H_-}(x,y)f(y) + \sum_{x \in V} \pi(x)  G_{H_-}(x,x)f(x) \\
                &= \sum_{x \neq y \in V}  \pi(x) G(x,y)H_{-}(x,y)f(y) - \sum_{x \in V} \pi(x) f(x) \sum_{y\in V\setminus \{x \} } G(x,y)H_{-}(x,y) \\
                &= \sum_{x \neq y \in V}  \pi(y) G(y,x)H_{-}(y,x)f(x) - \sum_{x \neq y \in V} \pi(x)G(x,y)H_{-}(x,y) f(x) \\
                &= \sum_{x \neq y \in V}  \pi(x) G(x,y)H_{+}(x,y)f(x) - \sum_{x \neq y \in V} \pi(x)G(x,y)H_{-}(x,y) f(x),
            \end{align*}
        and hence $  \pi [f \text{div}_G H] = \pi[f G_{H_-}[f]]$.
    \end{proof}
    Now we can apply Theorem \ref{generator L_h-} to the curve of generator 
        \begin{align*}
            \forall x\neq y \in V, \ \forall t \geq 0, \qquad G_{\rho_t}(x,y) \df  L(x,y) \frac{\theta(\rho_t(x), \rho_t(y))}{\rho_t(x)}
        \end{align*}
    and the curve  
        \begin{align*}
           \forall x, y \in V, \ \forall t \geq 0, \qquad H_{\rho_t}(x,y) &\df  \nabla [\beta U + \varphi'\circ \rho_t] (x,y) \\
            &= \beta (U(y) - U(x)) + \varphi'(\rho_t(y)) -  \varphi'(\rho_t(x)).
        \end{align*}
    Observe that $\mu_t$ is reversible for $G_{\rho_t}$:
        \begin{align*}
            \mu_t(x) G_{\rho_t}(x,y) &= \ell(x) \rho_t(x) \times  L(x,y) \frac{\theta(\rho_t(x), \rho_t(y))}{\rho_t(x)} \\
            &= \ell(x) L(x,y) \theta(\rho_t(x), \rho_t(y)) \\
            &= \mu_t(y) G_{\rho_t}(y,x)
        \end{align*}
    by the symmetry of $\theta$ and the fact that $\ell$ is reversible for $L$. We have for any test function $f \in \R^V$, 
        \begin{align*}
            \dot \mu_t[f] &= \sum_{x \in V} \pi(x) \dot \rho_t(x) f(x) \\
            &= \sum_{x \in V} \pi(x) f(x) \sum_{y \in V} L(x,y)\theta(\rho_t(x),\rho_t(y)) \nabla [\beta U + \varphi'\circ \rho_t] (x,y) \\
            &= \sum_{x,y \in V} \pi(x) \rho_t(x) f(x) L(x,y) \frac{\theta(\rho_t(x),\rho_t(y))}{\rho_t(x)} \nabla [\beta U + \varphi'\circ \rho_t] (x,y) \\
            &= \sum_{x,y \in V} \mu_t f(x) G_{\rho_t}(x,y) H_{\rho_t}(x,y) \\
            &= \mu_t[ f \text{div}_{G_{\rho_t}} H_{\rho_t}] \\
            &= \mu_t[L_{\beta,\rho_t}[f]], \label{L; Instant invariantmeasure} \numberthis
        \end{align*}
    where we used Theorem \ref{generator L_h-} with the triple $(\pi, G, H) = (\mu_t, G_{\rho_t}, H_{\rho_t})$ and $L_{\beta,\rho_t}$ in \eqref{Nonlinear: generator} plays the role of the generator $G_{H_-}$. Hence, we have transformed the dynamic \eqref{ODE: beta fixed} to a nonlinear Markov representation as announced in \eqref{Nonlinear: betafixed}, which is useful for a simulation.
    
        Observe that the generator given in \eqref{Nonlinear: generator explicit} is not irreducible for all $\rho \in \mc D_+(V)$. For example, it may happen that for a fixed $x$, $\beta\nabla U(x,y) > - \nabla [\varphi'\circ \rho](x,y)$ for all $y \neq x$, and hence $L_{\beta,\rho}(x,y) = 0$. In fact, there is another Markov interpretation such that the nonlinear generator is irreducible for all $\rho \in \mc D_+(V)$. We can replace the generator $L_{\beta,\rho}$ by $Q_{\beta, \rho}$ defined by
            \begin{align} \label{generator: new interpretation}
                \forall \rho \in \mc D_+(V), \ \forall x \neq y, \quad Q_{\beta, \rho} (x,y) \df L(x,y)\Big(1 + \frac{\theta(\rho(x),\rho(y))}{\rho(x)}\beta (U(y)-U(x))_-\Big),
            \end{align}
        or more explicitly
            \begin{align} \label{generator: new interpretation, explicit}
                \forall \rho \in \mc D_+(V), \ \forall x \neq y, \quad Q_{\beta, \rho} (x,y) \df L(x,y)\Big(1 + \frac{\rho(y) - \rho(x)}{\rho(x)(\varphi'(\rho(y)) - \varphi'(\rho(x) )}\beta (U(y)-U(x))_-\Big).
            \end{align}
        Clearly $\forall x \neq y$, $ Q_{\beta, \rho} (x,y) \geq L(x,y)$ hence $Q_{\beta, \rho} (x,y)$ is irreducible. The generator $Q_{\beta, \rho}$ is obtained by the following computations. We have
            \begin{align*}
                \dot \mu_t[f] &= \sum_{x \in V} \ell(x) \dot \rho_t(x) f(x) \\ 
                &= \sum_{x \in V}\ell(x) f(x) \sum_{y \in V}L(x,y) \theta (\rho_t(x),\rho_t(y)) \nabla[ \beta U + \varphi' \circ \rho_t](x,y) \\
                &= \sum_{x \in V}\ell(x) f(x) \sum_{y \in V}L(x,y) \theta (\rho_t(x),\rho_t(y)) \beta\nabla U(x,y) + \sum_{x \in V}\ell(x) f(x) \sum_{y \in V}L(x,y) \theta (\rho_t(x),\rho_t(y)) \nabla[\varphi' \circ \rho_t](x,y) \\
                &= \sum_{x \in V}\ell(x) f(x) \sum_{y \in V}L(x,y) \theta (\rho_t(x),\rho_t(y)) \beta\nabla U(x,y) + \sum_{x \in V}\ell(x) f(x) \sum_{y \in V}L(x,y) (\rho_t(y) - \rho_t(x)) \\
                &= \sum_{x \in V}\mu_t(x) f(x) \sum_{y \in V}L(x,y) \frac{\theta (\rho_t(x),\rho_t(y))}{\rho_t(x)} \beta\nabla U(x,y) + \sum_{x \in V}\ell(x) f(x) \sum_{y \in V}L(x,y) \rho_t(y) \\
                &= \sum_{x \in V}\mu_t(x) f(x) \sum_{y \in V}L(x,y) \frac{\theta (\rho_t(x),\rho_t(y))}{\rho_t(x)} \beta\nabla U(x,y) + \sum_{x \in V}\ell(x) f(x) \sum_{y \in V}L(x,y) \rho_t(y) \\
                &= \mu_t[f \text{div}_{G_{\rho_t}}[\beta\nabla U]] + \ell[f L[\rho_t]],
            \end{align*}
        where the generator $G_\rho$ is defined by
            \begin{align*}
                \forall \rho \in \mc D_+(V), \ \forall x \neq y, \qquad G_{\rho}(x,y) \df L(x,y) \frac{\theta(\rho(x),\rho(y))}{\rho(x)}
            \end{align*}
        Since $\ell$ is reversible for $L$, the last expression is equal to $ \ell[f L[\rho_t]] =  \ell[\rho_t L[f]] = \mu_t[L[f]]$. Applying Theorem \ref{generator L_h-}, we get 
            \begin{align*}
                \mu_t[f \text{div}_{G_{\rho_t}}[\beta\nabla U]] = \mu_t[G_{\rho_t, \beta\nabla U_-}[f]],
            \end{align*}
        where $G_{\rho, \beta\nabla U_-}$ is the Markov generator defined by 
            \begin{align*}
                \forall \rho \in \mc D_+(V), \ \forall x \neq y, \qquad G_{\rho_t, \beta\nabla U_-}(x,y) \df L(x,y) \frac{\theta(\rho(x),\rho(y))}{\rho(x)} \beta (U(y) - U(x))_-.
            \end{align*}
        Observe that $Q_{\beta, \rho_t} = L + G_{\rho_t, \beta\nabla U_-}$, putting together these computations we get 
            \begin{align*}
                \dot \mu_t[f] = \mu_t [ (L + G_{\rho_t, \beta\nabla U_-} ) [f] ]  = \mu_t [Q_{\beta, \rho_t}[f]],
            \end{align*}
        which leads to the non-linear interpretation in \eqref{generator: new interpretation}, \eqref{generator: new interpretation, explicit}.

        {So far, we have two generators that give rise to two different nonlinear Markov processes, both admitting $(\mu_t)_{t\geq 0}$ as their time-marginal laws. Clearly, we can also replace these generators with any convex combination of these two generators and still get a new Markov process with the same time-marginal laws.} For the density $\eta_\beta$ in Theorem \ref{theo: the unique minizer rhobeta}, the probability measure associated with this density, denoted $\zeta_\beta$, is the unique invariant measure of $Q_{\beta, \eta_\beta}$, in the sense that $ \zeta_\beta Q_{\beta, \eta_\beta} = 0$, while $L_{\beta, \eta_\beta} = 0$. Indeed, recall that $\varphi'(\eta_\beta(x)) + \beta U(x) = \varphi'(\eta_\beta(y)) + \beta U(y)$ for any $x,y \in V$, putting this back in \eqref{Nonlinear: generator explicit}, we see $L_{\beta, \eta_\beta}(x,y) = 0$. Fix one $x \in V$, we have
            \begin{align*}
                \zeta_\beta Q_{\beta, \eta_\beta} (x) &= \sum_{y \in V}\zeta_\beta(y) Q_{\beta, \eta_\beta}(y,x) \\\
                &= \sum_{y\in V \setminus \{ x\}}\ell(y) \eta_\beta(y) L(y,x)\Big(1+ \frac{\theta(\eta_\beta(x), \eta_\beta(y))}{\eta_\beta(y)} \beta(U(x) - U(y))_-\Big) + \ell(x) \eta_\beta(x) Q_{\beta, \eta_\beta}(x,x) \\
                &=   \ell(x) \sum_{y\in V \setminus \{ x\}} L(x,y)\Big(\eta_\beta(y) + (\eta_\beta(y) - \eta_\beta(x))_- \Big)  \\ 
                 &\qquad    - \ell(x) \eta_\beta(x) \sum_{y\in V \setminus \{ x\}} L(x,y)\Big( 1+ \frac{\theta(\eta_\beta(x), \eta_\beta(y))}{\eta_\beta(x)} \beta(U(y) - U(x))_- \Big)  \\
                &= \ell(x) \sum_{y\in V \setminus \{ x\}} L(x,y) \Big( \eta_\beta(y) - \eta_\beta(x) + (\eta_\beta(y) - \eta_\beta(x))_- - (\eta_\beta(x) - \eta_\beta(y))_-  \Big) \\
                &= 0,
            \end{align*}
    where we have used the fact that 
            \begin{align*}
                \theta(\eta_\beta(x), \eta_\beta(y)) \beta(U(x) - U(y))_- &= \Big( \frac{\eta_\beta(y) -\eta_\beta(x) }{\varphi'(\eta_\beta(y)) - \varphi'(\eta_\beta(x))} \times (\varphi'(\eta_\beta(y)) - \varphi'(\eta_\beta(x))) \Big)_- \\
                &= (\eta_\beta(y) -\eta_\beta(x))_-.
            \end{align*}
    We will compare the simulations of particle systems using these two different generators \eqref{Nonlinear: generator} and \eqref{generator: new interpretation} in Section \ref{Simulation}.

\section{The time-inhomogeneous situation}\label{S5}
    In this section, we will consider the case when $\beta$ depends on time, i.e., $\beta = (\beta_t)_{t\geq0}$ is an inverse temperature schedule. Inspired by the gradient flow dynamic of $\mathcal U_\beta$ \eqref{ODE: beta fixed} and its homogeneous nonlinear Markov representation \eqref{Nonlinear: betafixed}, we consider the time-inhomogeneous dynamic
        \begin{align} \label{Nonlinear: inhomo dynamics}
            \forall x \in V, \ \forall t >0, \quad \dot \rho_t(x) = \sum_{y\in V} L(x,y) \theta(\rho_t(x),\rho_t(y)) \Big(\nabla [ \beta_tU + \varphi' \circ \rho_t](x,y) \Big), \quad \rho_0 \in \mathcal D_+(V).
        \end{align}
    As we shall see, for some appropriate choices of the temperature schedule $\beta = (\beta_t)_{t\geq0} $ such that $t \mapsto \beta_t$ increases slowly enough to infinity, a unique solution to \eqref{Nonlinear: inhomo dynamics} will exist. Moreover, $\mu_t$ (the measure with density $\rho_t$) will converge to a measure concentrated on $\mathcal M(U)$, the set of global minimizers of $U$. In the meantime, we will always assume that the temperature schedule $\beta = (\beta_t)_{t\geq 0}$ satisfies
    \begin{as} \label{as: beta_t meantime}
        $\lim_{t\rightarrow +\infty} \beta_t = +\infty$ and $t \mapsto \beta_t $ is continuously differentiable with $\dot \beta_t \df  \partial\beta /\partial t > 0$, $ \forall t >0$ and $\beta_0 \geq 0$.
    \end{as}
    The crucial aspect of this section is a functional inequality that provides a lower bound of $\rho_{t,\wedge}\df  \min_{x\in V} \rho_t(x)$ for a solution $(\rho_t)_{t \in[0,T)}$ of \eqref{Nonlinear: inhomo dynamics} on some interval $[0,T)$ (the existence and uniqueness on a small interval is guaranteed by Cauchy-Picard theorem) in terms of $\beta_t$. Consequently, it facilitates proving the existence and uniqueness of such a solution on the entire half-real line $[0,+\infty)$ as well as the convergence of $(\mu_t)_{t\geq 0}$ ($\mu_t$ being the measure with density $\rho_t$) to the measure $\mu_{\infty}$ in \eqref{mu_infty}.
    \subsection{A functional inequality}
    For any $t\geq 0$, we let $\nu_t$ be the global mimimizer of the function $\mathcal U_{\beta_t}$, i.e., $\nu_t$ is the density satisfying
        \begin{align*}
            \forall x \in V, \qquad \varphi'(\nu_t(x)) + \beta_t U(x) = c(\beta_t),
        \end{align*}
        where $c(\beta_t)$ is the unique number that solves the equation (recall that $g = (\varphi')^{-1}$ in \eqref{C constant})
        \begin{align*}
            c \in \mathbb R, \qquad \sum_{x \in V} \ell(x)g(c - \beta_tU(x)) = 1.
        \end{align*}
        In other words, $\nu_t = \eta_{\beta_t}$, where $\beta \mapsto \eta_\beta$ is given in Theorem \ref{theo: the unique minizer rhobeta}. The following properties of $\nu = (\nu_t)_{t\geq 0}$ are consequences of Theorem \ref{theo: properties of rho beta}.
    \begin{theorem} \label{theo: properties of nu_t}
        Let $\beta = (\beta_t)_{t\geq 0}$ satisfy Assumption \ref{as: beta_t meantime}
        \begin{enumerate}[label=(\roman*)]
        \item $\forall x,y \in V$ such that $U(x) = U(y)$, $\nu_t (x) = \nu_t(y)$ (in particular $x,y \in \mathcal M(U)$, the set of global minimizers of $U$ defined in Theorem \ref{theo:1.1}).
        \item $\forall x \notin \mathcal M(U)$, $\lim_{t \rightarrow +\infty }\nu_t (x) = 0$. More precisely, we have 
            \begin{align*}
                \forall x \notin \mathcal M(U), \qquad \lim_{t \rightarrow +\infty} \beta_t [\nu_t(x)]^{1-m} = \frac{1}{(1-m)(U(x) - \min U)},
            \end{align*}
        where $m <0$ is fixed in \eqref{func: varphi}.
        \item $\forall x \notin \mathcal M(U)$, we have the inequalities
            \begin{align*}
                \normalfont \forall t\geq0, \qquad  \nu_t(x)^{1-m} \geq \frac{1}{ \beta_t(1-m)\text{osc } U + 1}, \quad \text{where} \quad  \text{osc } U = \max U - \min U.
            \end{align*}
        \item $\forall x \in \mathcal M(U)$, $\nu_t(x) \geq 1$ and $ \dot \nu_t \df  \cfrac{\partial \nu_t(x)}{\partial t} \geq 0$, where the equality holds iff $U$ is constant on $V$.
        \item $\forall x \in \mathcal M(U)$, $\lim_{t \rightarrow +\infty}\nu_t (x) = \cfrac{1}{\sum_{y \in \mathcal M(U)} \ell(y)}$ and so $\lim_{t \rightarrow +\infty}\ell(x)\nu_t (x) = \zeta_\infty (x) $, where $\zeta_\infty$ is given in \eqref{mu_infty}. 
    \end{enumerate}
    \end{theorem}
    \begin{proof}
        We prove $(iii)$ because the rest are just restatements of Theorem \ref{theo: properties of rho beta}. Indeed, from the proof of $(ii)$ of Theorem \ref{theo: properties of rho beta}, we have
            \begin{align} \label{paas}
                \beta_t[\nu_t(x)]^{1-m} = \frac{1}{(1-m)(U(x) - \min U) + \frac{(m-1)\varphi'(\nu_t(x_0)) + 1}{\beta_t}},
            \end{align}
        where $x \notin \mathcal M(U) $ and $x_0 \in \mathcal M(U)$ and observe that $ \nu_t(x_0) \geq 0 $ so $\varphi'(\nu_t(x_0)) \geq 0$. Therefore the denominator of \eqref{paas} is smaller than $(1-m)(U(x) - \min U) + \beta_t^{-1}$, so after dividing both sides with $\beta_t$, the denominator of the left hand side is $ \beta_t(1-m)\text{osc }U + 1$, which is the desired result.
    \end{proof}
    Having defined $\nu = (\nu_t)_{t\geq 0}$, we can interpret its as a curve of ``instantaneous invariant density" in the sense that $\forall t\geq0, \ v_tQ_{\beta_t,\nu_t} = 0$, where $v_t$ is the probability measure admitting $\nu_t$ as its density and $Q_{\beta, \rho}$ is the generator given in \eqref{generator: new interpretation}, \eqref{generator: new interpretation, explicit}. Also, note that $\forall t\geq0, \ L_{\beta_t, \nu_t} = 0$, where $L_{\beta, \rho}$ is given in \eqref{Nonlinear: generator}, \eqref{Nonlinear: generator explicit}, because $\nabla \beta_t U(x,y) = - \nabla \varphi' \circ \nu_t(x,y)$, $\forall x \neq y$. We move on to the definition of the functionals that are concerned in this section. Given $\varphi$ defined in \eqref{func: varphi}, consider
        \begin{align} \label{functional: I, G}
            \forall t\geq0, \ \forall \rho \in \mathcal D_+(V), \ \mathcal I(t,\rho)&\df  \sum_{x \in V} \ell(x)\Big( \varphi(\rho(x)) - \varphi(\nu_t(x)) - \varphi'(\nu_t(x))(\rho(x)-\nu_t(x)) \Big), \\
             \forall t\geq0, \ \forall \rho \in \mathcal D_+(V), \ \mathcal G(t,\rho) &\df  \frac{1}{2}\sum_{x\in V}\sum_{y\in V}\ell(x)L(x,y) \theta(\rho(x),\rho(y))\Big(\nabla[\beta_t U + \varphi'\circ \rho](x,y) \Big)^2,
        \end{align}
        where we recall the notation $ \nabla[\beta_t U + \varphi'\circ \rho](x,y) = \beta_t (U(y) - U(x)) + \varphi'(\rho(y)) - \varphi'(\rho(x)) $. Observe that $\mc I(t,\rho) $ and $\mc G(t,\rho)$ are exactly $ \mc I(\beta_t,\rho)$ and $\mc G(\beta_t,\rho)$ in \eqref{functional: I},\eqref{functional: G} evaluated at $\beta = \beta_t$, respectively. Below is a functional inequality that will be crucial in the proofs of convergence theorems of the dynamic \eqref{Nonlinear: inhomo dynamics}.
    \begin{theorem} \label{FuncIneq: inhomo case}
        Let $\varphi$ and $\theta$ be given in \eqref{func: varphi} and \eqref{func: theta our}, and fix a $\rho_* \in \mathcal D_+(V)$ we consider the following functionals
            \begin{align*}
                \forall \rho \in \mathcal D_+(V), \quad \mathcal I_*(\rho) &\df  \sum_{x \in V}  \ell(x)\Big( \varphi(\rho(x)) - \varphi(\rho_*(x)) - \varphi'(\rho_*(x))(\rho(x)-\rho_*(x)) \Big),\\
                \mathcal G_*(\rho) &\df  \frac{1}{2}\sum_{x,y \in V} \ell(x) L(x,y) \theta (\rho(x),\rho(y)) \Big( \varphi'(\rho(x)) - \varphi'(\rho_*(x)) - \varphi'(\rho(y)) + \varphi'(\rho_*(y))\Big)^2.
            \end{align*}
        Then it holds that 
            \begin{align*}
                \mathcal G_*(\rho) \geq \Lambda(\rho) \mathcal I_*(\rho),
            \end{align*}
        where $\Lambda(\rho)$ is the spectral gap of the generator $K_\rho$ given by 
            \begin{align*}
                \forall x \neq y, \qquad K_\rho(x,y) \df  L(x,y) \frac{\theta(\rho(x),\rho(y))}{\theta(\rho(x),\rho_*(y))},
            \end{align*}
        which has the reversible invariant measure $\mu_\rho$ defined by
            \begin{align*}
                \forall x \in V, \qquad \mu_\rho(x) \df  \ell(x)\theta(\rho(x), \rho_*(x)).
            \end{align*}
    \end{theorem}
    \begin{proof}
        We will prove that for all $\rho \in \mathcal D_+(V)$, we have 
            \begin{align*}
                \mathcal G_*(\rho) \geq \Lambda (\rho) \sum_{x\in V} \Big( \varphi'(\rho(x)) - \varphi'(\rho_*(x))\Big) (\rho(x) - \rho_*(x))\ell(x).
            \end{align*}
        Let $f$ defined on $V$ by 
            \begin{align*}
                \forall x \in V, \qquad f(x) \df  \varphi'(\rho(x)) - \varphi'(\rho_*(x)),
            \end{align*}
        so that 
            \begin{align*}
                \mathcal G_*(\rho) &= \frac{1}{2}\sum_{x,y \in V} \mu_\rho(x) K_\rho(x,y)\Big( f(y) - f(x) \Big)^2 \\
                &= - \mu_\rho[f K_\rho[f]].
            \end{align*}
        Note that 
            \begin{align*}
                \mu_\rho[f] &= \sum_{x\in V} \ell(x) \theta(\rho(x), \rho_*(x)) \Big(  \varphi'(\rho(x)) - \varphi'(\rho_*(x))\Big) \\
                &= \sum_{x\in V: \varphi'(\rho(x)) \neq \varphi (\rho_*(x))} \ell(x) \theta(\rho(x), \rho_*(x)) \Big(  \varphi'(\rho(x)) - \varphi'(\rho_*(x))\Big) \\
                &= \sum_{x\in V: \varphi'(\rho(x)) \neq \varphi (\rho_*(x))} \ell(x) \frac{\rho(x) - \rho_*(x)}{\varphi'(\rho(x)) - \varphi'(\rho_*(x))} \Big(  \varphi'(\rho(x)) - \varphi'(\rho_*(x))\Big) \\
                &= \sum_{x\in V: \varphi'(\rho(x)) \neq \varphi (\rho_*(x))} \ell(x) (\rho(x) - \rho_*(x)) \\
                &= \sum_{x\in V} \ell(x) (\rho(x) - \rho_*(x)) \\
                &= 0.
            \end{align*}
        It follows that 
            \begin{align*}
                - \mu_\rho[fK_\rho[f]] &\geq \Lambda(\rho) \mu_\rho[f^2] \\
                &= \Lambda(\rho) \sum_{x \in V} \ell (x) \theta(\rho(x), \rho_*(x)) \Big( \varphi'(\rho(x)) - \varphi'(\rho_*(x))\Big)^2 \\
                &= \Lambda(\rho) \sum_{x \in V} \ell (x) (\rho(x) -\rho_*(x)) \Big( \varphi'(\rho(x)) - \varphi'(\rho_*(x))\Big),
            \end{align*}
        which is the announced result. Now using Lemma \ref{theta: properties 1}, we have
            \begin{align*}
                \mathcal G_*(\rho) &\geq \Lambda(\rho) \sum_{x \in V} \ell (x) (\rho(x) -\rho_*(x)) \Big( \varphi'(\rho(x)) - \varphi'(\rho_*(x))\Big),\\
                &\geq  \Lambda(\rho) \sum_{x \in V} \ell (x) \Big(\varphi(\rho(x)) - \varphi(\rho_*(x)) - \varphi'(\rho_*(x))(\rho(x)- \rho_*(x)) \Big) \\
                &= \Lambda(\rho) \mathcal I_*(\rho).
            \end{align*}
    \end{proof}
    The dependence of $\Lambda(\rho)$ on $\rho$ can be problematic if we have little infomation on $\rho$. The following observation will be useful in this respect.
    \begin{proposition} \label{prop: lower bound on Lambda(rho)}
        With the settings in Theorem \ref{FuncIneq: inhomo case}, we have for all $\rho \in \mathcal D_+(V)$,
            \begin{align*}
                \Lambda(\rho) \geq \lambda \frac{\varphi''(1/\ell_\wedge)}{\varphi''(\rho_\wedge)},
            \end{align*}
        where $\lambda$ is the spectral gap of $L$, $\rho_\wedge\df  \min_{x \in V} \rho(x)$, $\ell_\wedge \df  \min_{x\in V} \ell(x)$.
    \end{proposition}
    \begin{proof}
        Due to the variational principle and monotonicity in each argument of $\theta$, we have for any $\rho \in \mathcal D_+(V)$, 
        \begin{align*}
            \Lambda(\rho) &= \frac{1}{2} \inf_{f\in \R^V \setminus \text{Vect}(1)} \frac{\sum_{x,y \in V} \ell(x)L(x,y) \theta(\rho(x),\rho(y))(f(y)-f(x))^2} {\inf_{c \in \R}\sum_{x \in V} \ell(x) \theta(\rho(x),\rho_*(y)) (f(x) -c)^2 } \\
            &\geq \frac{\theta(\rho_{\wedge}, \rho_\wedge)}{\theta(\rho_{\vee}, \rho_{*,\vee})} \times \frac{1}{2} \inf_{f\in \R^V \setminus \text{Vect}(1)} \frac{\sum_{x,y \in V} \ell(x)L(x,y))(f(y)-f(x))^2} {\inf_{c \in \R}\sum_{x \in V} \ell(x) (f(x) -c)^2 } \\
            &= \lambda  \frac{\theta(\rho_{\wedge}, \rho_\wedge)}{\theta(\rho_{\vee}, \rho_{*,\vee})}
        \end{align*}
    where we denote $\rho_\wedge \df  \min_{x\in V}  \rho(x) $ and  $\rho_\vee \df  \max_{x\in V}  \rho(x) $. We also note that $\rho_\vee \leq 1/\ell_{\wedge}$, thus $\theta(\rho_{\vee}, \rho_{*,\vee}) \leq \theta(1/\ell_\wedge, 1/\ell_\wedge) $ and so
        \begin{align*}
            \Lambda(\rho) \geq \lambda \frac{\theta(\rho_{\wedge}, \rho_\wedge)}{\theta(1/\ell_\wedge, 1/\ell_\wedge)} = \frac{\varphi''(1/\ell_\wedge)}{\varphi''(\rho_\wedge)}.
        \end{align*}
    \end{proof}
    \subsection{Existence, uniqueness and convergence of \texorpdfstring{$(\rho_t)_{t\geq 0}$}{Lg} }
       Apply Theorem \ref{FuncIneq: inhomo case} and Proposition \ref{prop: lower bound on Lambda(rho)}, we have the following estimate
        \begin{align}
            \mc G(t,\rho) \geq \lambda \frac{\varphi''(1/\ell_\wedge)}{\varphi''(\rho_\wedge)} \mc I(t,\rho),
        \end{align}
    We now come the the main result of this paper 
        \begin{theorem} \label{theo: main of ourpaper}
            For any $ m <0$, consider the function $\varphi = \varphi_{m,2}$ in \eqref{func: varphi}, as well as the time-inhomogeneous inverse temperature scheme
                \begin{align*}
                    \forall t \geq 0, \qquad \beta_t = (t_0 + t)^{\alpha} - 1, 
                \end{align*}
            where $t_0 \geq 1$ and 
                \begin{align*}
                   0< \alpha \leq \kappa(m) \df  \frac{-m}{2(1-m)} \in \left(0, \frac{1}{2} \right).
                \end{align*}
            Then there exists a unique solution to the equation \eqref{Nonlinear: inhomo dynamics}, which exists for all time $t\geq0$ and satisfies 
                \begin{align*}
                    \lim_{t \rightarrow +\infty} \mu_t[\mathcal M(U)] = 1.
                \end{align*}
        \end{theorem}
    The proof is decomposed into several intermediate results.
            \begin{lemma} \label{lemma:cite1}
                There  exists a unique solution $(\rho_t)_{t \in [0,T)}$ on a maximal interval $[0,T)$ to the equation \eqref{Nonlinear: inhomo dynamics} and a constant $K >0$ such that 
                    \begin{align*}
                        \forall t \in [0,T), \qquad \rho_{t,\wedge}^{-m} \geq \frac{1}{K(\beta_t+1)}.
                    \end{align*}
            \end{lemma}
            \begin{proof}
                Similarly to Proposition \ref{prop: e&u of GF of Ub}, the existence and uniqueness on a maximal interval $[0,T)$ of the dynamic \eqref{Nonlinear: inhomo dynamics} is guaranteed by Cauchy-Picard Theorem. Indeed, it can be easily checked that the mapping
                \begin{align*}
                    [0,+\infty) \times \mathcal D_+(V) \ni (t,\rho) \mapsto \mathcal F(t,\rho)(x): =  \sum_{y\in V} L(x,y) \theta(\rho(x),\rho(y)) \Big(\nabla \beta_t U(x,y) + \nabla [\varphi' \circ \rho](x,y) \Big)
                \end{align*}
             satisfies a locally Lipschitz condition in $\rho$, uniformly in $t$ on any compact interval $t \in [a,b]$ like in the proof of Proposition \ref{prop: e&u of GF of Ub}. Let us differentiate with respect to time (recall that $\varphi' (\nu_t(x)) + \beta_tU(x) = c(\beta_t)$)
                \begin{align*}
                    \partial_t \mc I (t, \rho_t) &=  \sum_x \dot \rho_t(x) \Big(\beta_t U(x) + \varphi'(\rho_t(x)) \Big)\ell(x) - \sum_x \dot \nu_t(x) \Big(\beta_t U(x) + \varphi'(\nu_t(x) \Big) \ell(x)  \\
                    & \qquad \qquad + \dot \beta_t \sum_{x\in V} U(x)(\rho_t(x) - \nu_t(x))\ell(x) \notag  \\
                    &=  \sum_x \dot \rho_t(x) \Big(\beta_t U(x) + \varphi'(\rho_t(x)) \Big) \ell(x) - c(\beta_t) \sum_x \dot \nu_t(x)  \ell(x)  + \dot \beta_t \sum_{x\in V} U(x)(\rho_t(x) - \nu_t(x))\ell(x) \notag \\
                    &= -\mc G (t,\rho_t) + \dot \beta_t \sum_{x\in V}U(x)(\rho_t(x) - \nu_t(x))\ell(x) \label{ineq: where we can improve rate of It} \numberthis \\
                    &\leq -\lambda \frac{\varphi''(1/\ell_\wedge)}{\varphi''(\rho_{t,\wedge})} \mc I (t, \rho_t) +\dot \beta_t \text{osc } U \label{ineq:improve !}, \numberthis
                \end{align*}
                where we have used the computation \eqref{computation for G(rho)} in \eqref{ineq: where we can improve rate of It} for $\mc G (t,\rho_t)$ and the fact that $\sum_{x\in V} \dot \nu_t (x) \ell(x) = \frac{\partial}{\partial t} (\sum_{x \in V} \nu_t(x) \ell(x)) = \frac{\partial}{\partial t} 1 = 0$. By the time-homogeneous version of Gronwall inequality, we obtain $\forall t\in [0,T)$,
                \begin{align*}
                     \mc I (t, \rho_t) &\leq  \text{osc} U \int_0^t \dot\beta_s \exp\left( - \int^t_s \lambda \frac{\varphi''(1/\ell_\wedge)}{\varphi''(\rho_{u,\wedge})}du \right)ds + \exp \left( - \int^t_0 \lambda \frac{\varphi''(1/\ell_\wedge)}{\varphi''(\rho_{s,\wedge})}ds  \right) \mc I (0,\rho_0) \numberthis \label{Ineq: Gronwal for I} \\ 
                      &\leq \text{osc } U (\beta_t - \beta_0)  + \mc I(0,\rho_0) \\
                      &\leq C_0 (\beta_t +1),
                \end{align*}
            where $C_0 = \max\{ \text{osc} U; \ \mc I(0,\rho_0) - \text{osc} U\beta_0 \}$.
            Using this, we express $\mathcal I(t,\rho_t)$ as 
                \begin{align*}
                    \sum_{x\in V} \beta_t U(x) \Big(\rho_t(x) - \nu_t(x)\Big) \ell(x) + \sum_{x\in V} \Big(\varphi(\rho_t(x)) - \varphi(\nu_t(x))\Big)\ell(x) \leq C_0 (\beta_t +1),
                \end{align*}
            so  
                \begin{align*}
                    \sum_{x\in V} \varphi(\rho_t(x))\ell(x)  &\leq C_0 (\beta_t +1) +  \sum_{x\in V} \varphi(\nu_t(x))\ell(x) - \sum_{x\in V} \beta_t U(x) \Big(\rho_t(x) - \nu_t(x)\Big) \ell(x) \\
                    &\leq C_1(\beta_t +1) +  \sum_{x\in V} \varphi(\nu_t(x))\ell(x),
                \end{align*}
            where $C_1= C_0 + \text{osc} U$. Observe that from Theorem \ref{theo: properties of nu_t}, $\nu_t(x) \geq 1$ if $x \in \mathcal M(U)$ and recall that $m<0$, we have
                \begin{align*}
                     \sum_{x\in V} &\varphi(\nu_t(x))\ell(x) \leq \sum_{x\in V} \Big(\varphi_m(\nu_t(x)) +  \varphi_2(\nu_t(x))\Big) \ell(x) \\
                     &= \sum_{x\in V} \frac{(\nu_t(x) -1)^2}{2}\ell(x) + \sum_{x\in V} \Big(  \frac{\nu_t(x)^{m} - 1 - m(\nu_t(x) - 1) }{m(m-1)}  \Big)\ell(x) \\ 
                     &= \frac{1}{2}\sum_{x\in V} \nu_t(x)^2\ell(x) - \frac{1}{2} + \sum_{x\in \mathcal M(U)}  \frac{\nu_t(x)^{m} }{m(m-1)}  \ell(x) +  \sum_{x\notin \mathcal M(U)}  \frac{\nu_t(x)^{m}}{m(m-1)}  \ell(x) -   \frac{1}{m(m-1)}  \\
                     & \leq \frac{1}{2} \nu_{t,\vee} \sum_{x\in V} \nu_t(x)\ell(x) +  \sum_{x\in \mathcal M(U)}  \frac{ \nu_t(x) \ell(x)}{m(m-1)}   +  \sum_{x\notin \mathcal M(U)}  \frac{\nu_t(x) \ell(x)}{m(m-1)\nu_t(x)^{1-m}}     - \left(\frac{1}{2} + \frac{1}{m(m-1)} \right) \\
                     & \leq \frac{1}{2} \nu_{t,\vee} +  \sum_{x\in \mathcal M(U)}  \frac{ \nu_t(x) \ell(x) }{m(m-1)}   +  \sum_{x\notin \mathcal M(U)}  \frac{\nu_t(x) (\beta_t (1-m) \text{osc } U + 1)}{m(m-1)}  \ell(x)   - \left(\frac{1}{2} + \frac{1}{m(m-1)} \right) \\
                     &\leq \frac{1}{2} ( \nu_{t,\vee} - 1)  +  \beta_t (1-m) \text{osc } U \sum_{x\notin \mathcal M(U)}  \frac{\nu_t(x) }{m(m-1)}  \ell(x) \\
                     &\leq \frac{1}{2} ( \frac{1}{\ell_\wedge} - 1)  + \frac{\beta_t \text{osc } U}{-m} \\
                     &\leq C_2(\beta_t+1),
                \end{align*}
            where $\nu_{t,\vee}\df  \max_{ x\in V}\nu_t(x)$, $\ell_{\wedge} \df  \min_{ x\in V} \ell(x)$, $C_2\df  \max\{ \frac{1}{2}(\frac{1}{\ell_\wedge} -1), \frac{\text{osc}U}{-m}\}$. Also write $\rho_{t,\wedge} \df  \min_{x \in V}\rho_t(x)$, and note that $\rho_{t,\wedge} \leq 1$, we have for all $t \in [0,T)$
                \begin{align*}
                    \ell_\wedge \varphi(\rho_{t,\wedge}) &\leq \sum_{x \in V}  \varphi(\nu_t(x))\ell(x) + C_1(\beta_t +1) \\
                    &\leq C_3(\beta_t+1),
                \end{align*}
            where $C_3 = C_1 +C_2$. Since $ \varphi(\rho_{t,\wedge}) = \frac{\rho_{t,\wedge}^{m} - 1 - m(\rho_{t,\wedge} - 1) }{m(m-1)} $, we get for all $t\in [0,T)$
                \begin{align*}
                    \rho_{t,\wedge}^{m} &\leq \frac{1}{\ell_\wedge}m(m-1)C_3(\beta_t+1) +1 + m(\rho_{t,\wedge} -1)  \\
                     &\leq  \frac{1}{\ell_\wedge}m(m-1)C_3(\beta_t+1) +1 -m \qquad (\text{since} \ m\rho_{t,\wedge} <0) \\
                     &\leq C_4(\beta_t +1),
                \end{align*}
                where $C_4 = \frac{1}{\ell_\wedge}m(m-1)C_3 + 1-m$, and we arrive at 
                \begin{align*}
                    \rho_{t,\wedge}^{-m} \geq \frac{1}{C_4(\beta_t+1)}.
                \end{align*}
                Hence, choosing $K = C_4$ finishes the proof of the lemma.
            \end{proof}
            \begin{lemma}
                There is a constant $I_0$ such that, for all $t \in [0,T)$, $\mc I(t,\rho_t) \leq I_0$.  
            \end{lemma}
            \begin{proof}
            Using the previous lower bound on $\rho_{t,\wedge}$ in the proof of Lemma \ref{lemma:cite1}, we deduce for $s \leq t <T$ (recall that $\varphi''(r) = r^{m-2}$ if $r\in(0,1)$)
                \begin{align*}
                    \exp \left(- \lambda \varphi''(1/\ell_\wedge) \int^t_s \frac{1}{\varphi''(\rho_{u,\wedge})}du \right) &= \exp \left(- \lambda \varphi''(1/\ell_\wedge) \int^t_s \rho_{u,\wedge}^{2-m}du \right) \\
                    &\leq \exp \left(- \lambda \varphi''(1/\ell_\wedge) \int^t_s [C_4(\beta_u+1) ]^{-\frac{2-m}{-m}}du \right) \\
                    &= \exp \left(- \lambda \varphi''(1/\ell_\wedge)C_4^{\frac{2-m}{m}} \int^t_s (t_0+u)^{\alpha \times \frac{2-m}{m}}du \right) \\
                    &= \exp \left(- C_5 \Big[(t_0+t)^{1+\frac{(2-m)\alpha}{m}} - (t_0+s)^{1+\frac{(2-m)\alpha}{m}}\Big] \right) \\
                    &= \frac{\exp(C_5(t_0+s)^{1+\frac{(2-m)\alpha}{m}})}{\exp(C_5(t_0+t)^{1+\frac{(2-m)\alpha}{m}})},
                \end{align*}
            where $ C_5 \df \lambda \varphi''(1/\ell_\wedge)C_4^{\frac{2-m}{m}}$. Putting this back to the inequality \eqref{Ineq: Gronwal for I}, we have $\forall t \in [0, T),$
                \begin{align*}
                    \quad  \mc I (t, \rho_t) &\leq  \text{osc } U \int^t_0  \alpha (t_0+s)^{\alpha-1} \frac{\exp(C_5(t_0+s)^{1+\frac{(2-m)\alpha}{m}})}{\exp(C_5(t_0+t)^{1+\frac{(2-m)\alpha}{m}})}ds + \frac{\exp(C_5t_0^{1+\frac{(2-m)\alpha}{m}})}{\exp(C_5(t_0+t)^{1+\frac{(2-m)\alpha}{m}})} \mc I(0,\rho_0).
                \end{align*}
            Let 
                \begin{align*}
                    p(t) &\df  \frac{1}{C_5 \left(1+\frac{(2-m)\alpha}{m} \right)}(t_0 +t)^{\alpha -1 - \frac{(2-m)\alpha}{m}} = \frac{1}{C_5 \left(1+\frac{(2-m)\alpha}{m} \right)} (t_0 +t)^{\frac{\alpha}{\kappa(m)} -1}, \\ 
                    q(t) &\df   \exp(C_5(t_0+t)^{1+\frac{(2-m)\alpha}{m}}), \\ 
                   \implies p'(t) &=  \frac{1}{C_5 \left(1+\frac{(2-m)\alpha}{m} \right)} \left(\frac{\alpha}{\kappa(m)} -1 \right) (t_0 +t)^{\frac{\alpha}{\kappa(m)} -2},  \\
                    q'(t) &=  C_5 \left(1+\frac{(2-m)\alpha}{m} \right) (t_0+t)^{\frac{(2-m)\alpha}{m}} \exp(C_5(t_0+t)^{1+\frac{(2-m)\alpha}{m}}).
                \end{align*}
            We have as $t\rightarrow + \infty$
                \begin{align*}
                    \frac{p'(t)}{p(t)} \frac{q(t)}{q'(t)} &=  \frac{1}{C_5 \left(1+\frac{(2-m)\alpha}{m} \right)}\left(\frac{\alpha}{\kappa(m)} -1 \right)(t_0+t)^{-1 -\frac{(2-m)\alpha}{m}} \rightarrow 0
                \end{align*}
            because $ -1 -\frac{(2-m)\alpha}{m} \leq -1 + \frac{2-m}{-m} \times \frac{-m}{2(1-m)} = \frac{m}{2(1-m)} <0$ (recall $m<0$). Note that $p' <0$ and
            \begin{align*}
                \lim_{t\rightarrow +\infty}p(t)q(t) &= +\infty \\
                \lim_{t\rightarrow +\infty} \int^t_0 p(s)q'(s)ds &=  \lim_{t\rightarrow +\infty} \Big( p(t)q(t) - p(0)q(0)  + \int^t_0 [-p'(s)]q(s)ds\Big)  = + \infty.
            \end{align*}
        By L'Hopital's rule, we have
        \begin{align*}
            \lim_{t\rightarrow +\infty} \frac{\int^t_0 p(s)q'(s)ds}{p(t)q(t)} =  \lim_{t\rightarrow +\infty} \frac{p(t)q'(t)}{p'(t)q(t) + p(t)q'(t)} = \lim_{t\rightarrow +\infty} \frac{1}{1+ \frac{p'(t)q(t)}{p(t)q'(t)}} = 1.
        \end{align*}
        Hence, there exists a constant $K_{\mc I}$ such that 
        \begin{align} \label{forcitehere}
             \frac{\int^t_0 (t_0+s)^{\alpha-1}\exp(C_5(t_0+s)^{1+\frac{(2-m)\alpha}{m}})ds}{\exp(C_5(t_0+t)^{1+\frac{(2-m)\alpha}{m}})} =  \frac{\int^t_0 p(s)q'(s)ds}{q(t)} \leq K_{\mc I} p(t).
        \end{align}
        The condition $0< \alpha \leq \kappa(m)$ forces $p(t)$ to be constant or goes to $0$. In either case, we have (recall that $(t_0 +t)^{\frac{\alpha}{\kappa(m)} -1} \leq 1$ since $t_0 \geq 1$)
            \begin{align*}
                \forall t \in [0,T), \qquad \mc I(t,\rho_t) \leq \frac{ K_{\mc I}\text{osc } U}{C_5 (1+\frac{(2-m)\alpha}{m}) } + \mathcal I (0,\rho_0) =: I_0,
            \end{align*}
        which ends the proof of the lemma. 
        \end{proof}
            Now we prove Theorem \ref{theo: main of ourpaper}.
        \begin{proof}
        (of Theorem \ref{theo: main of ourpaper}.) Suppose $T<+\infty$ then $\rho_t$ is going to the boundary of $\mathcal D_+(V)$ as $t \rightarrow T^-$, but then $\mathcal I(t,\rho_t)$ will go to infinity since
            \begin{align*}
                \mathcal I(t,\rho_t) = \mathcal U_{\beta_t}(\rho_t) - \mathcal U_{\beta_t}(\nu_t),
            \end{align*}
            and
            \begin{align*}
                [0,T] \ni t \mapsto \mathcal U_{\beta_t}(\nu_t)= \sum_{x \in V} \beta_t U(x)\nu_t(x)\ell(x) + \sum_{x\in V} \varphi(\nu_t(x))\ell(x)
            \end{align*}
            is continuous and bounded, while $\mathcal U_{\beta_t}(\rho_t)$ is going to infinity because $\lim_{r \rightarrow 0^+} \varphi(r) = +\infty$. Therefore, $T = +\infty$ and we have established the existence and uniqueness of a solution $(\rho_t)_{t\geq 0}$ of \eqref{Nonlinear: inhomo dynamics}. Now using \eqref{functional: I, G} together with $ \varphi = \varphi_{m,2} \geq \varphi_m$, for large $t >0$ such that $\nu_t(x) \leq 1$ for $x \notin \mathcal M(U)$, we have
            \begin{align*}
                 \forall x \notin \mathcal M(U), \qquad \frac{I_0}{\ell(x)} &\geq \varphi(\rho_t(x)) - \varphi(\nu_t(x)) - \varphi'(\nu_t(x))(\rho_t(x) - \nu_t(x)) \\
                 &\geq \varphi_m(\rho_t(x)) - \varphi_m(\nu_t(x)) - \varphi_m'(\nu_t(x))(\rho_t(x) - \nu_t(x)) \\
                 &= \nu_t(x)^m \varphi_m \Big(\frac{\rho_t(x)}{\nu_t(x)} \Big)
            \end{align*}
            because
                \begin{align*}
                    \varphi_m(t) - \varphi_m(s) - \varphi_m'(s)(t - s) &= \frac{t^m -1 - m(t-1)}{m(m-1)} - \frac{s^m -1 - m(s-1)}{m(m-1)} - \frac{s^{m-1} - 1}{m-1}(t-s) \\
                    &= \frac{t^m -s^m -m (t-s) -m(s^{m-1}-1)(t-s)}{m(m-1)} \\
                    &= \frac{t^m -s^m  -ms^{m-1}(t-s)}{m(m-1)} \\
                    &= \frac{s^m \Big((t/s)^m -1  -m(t/s-1)\Big)}{m(m-1)} \\
                    &= s^m \varphi_m(t/s).
                \end{align*}
            Thus $\varphi_m \Big(\frac{\rho_t(x)}{\nu_t(x)} \Big) \leq \frac{I_0 \nu_t(x)^{-m}}{\ell(x)} \rightarrow 0  $, which implies $ \frac{\rho_t(x)}{\nu_t(x)} \rightarrow 1$ for $x \notin \mathcal M(U)$ as $t \rightarrow +\infty$. Since $\nu_t(x) \rightarrow 0$ for $x\notin \mathcal M(U)$ it follows that $\mu_t[V\setminus \mathcal M(U)] \rightarrow 0$ or $ \mu_t[\mathcal M(U)] \rightarrow 1$ as $t\rightarrow +\infty$, which finishes the proof of Theorem \ref{theo: main of ourpaper}.
        \end{proof}
        \begin{cor}
            Given then settings in Theorem \ref{theo: main of ourpaper}, let $(\rho_t)_{t\geq 0}$ be the unique solution to $\eqref{Nonlinear: inhomo dynamics}$, then we have 
                \begin{align} \label{4444}
                    |\mu_t[\mathcal M(U)] - 1| = O(t^{-\frac{\alpha}{1-m}}).
                \end{align}
            Moreover, the choice $m = -1$ and $\alpha = \frac{1}{4}$ give the ``best" rate for the upperbound
                \begin{align} \label{55555}
                    |\mu_t[\mathcal M(U)] - 1| = O(t^{-\frac{1}{8}}).
                \end{align}
        \end{cor}
        \begin{proof}
            From the proof of Theorem \ref{theo: main of ourpaper}, we know that for any $x \notin \mathcal M(U)$ we have $ \lim \frac{\rho_t(x)}{\nu_t(x)} = 1$. From Theorem \ref{theo: properties of nu_t} (ii), we have $\nu_t(x) = O(\beta_t^{\frac{1}{m-1}})$, so our choice of $\beta_t = (t_0+t)^{\alpha} -1 \approx t^{\alpha}$ leads to $\rho_t(x) = O(t^{\frac{\alpha}{m-1}})$, and therefore
            \begin{align*}
                |\mu_t[\mathcal M(U)] - 1| &= \mu_t[V\setminus \mathcal M(U)] = \sum_{ x \notin \mathcal M(U)} \ell(x) \rho_t(x)  = O(t^{-\frac{\alpha}{1-m}}),
            \end{align*}
            which shows \eqref{4444}. Since $\alpha \leq \kappa(m) = \frac{-m}{2(1-m)}$, we have $\frac{\alpha}{1-m} \leq \frac{-m}{2(1-m)^2} \leq \frac{1}{8}$ (use $(1+x)^2 \geq 4x$ for $x = -m >0$). The bound is obtained with the choice $m=-1$ and $\alpha = \kappa(-1) = \frac{1}{4}$ (corresponds to the choice $\beta_t = (t_0+t)^{1/4} -1 \approx t^{1/4}$), then we have \eqref{55555}.
        \end{proof}
        \begin{cor}
            Given then settings in Theorem \ref{theo: main of ourpaper}, let $(\rho_t)_{t\geq 0}$ be the unique solution to $\eqref{Nonlinear: inhomo dynamics}$. Suppose in addition $\alpha < \kappa(m)$ then it holds as $t\rightarrow +\infty$ that
                \begin{align} \label{endddd}
                    \frac{1}{2}\| \rho_t - \nu_t \|^2_{\mathbb L^2(\ell)} \leq \mc I(t,\rho_t) &\leq \normalfont \frac{\text{osc }U (t_0+t)^{\frac{\alpha}{\kappa(m)} -1} }{C_5(1+\frac{(2-m)\alpha}{m})} + \frac{\exp(C_5t_0^{1+\frac{(2-m)\alpha}{m}})}{\exp(C_5(t_0+t)^{1+\frac{(2-m)\alpha}{m}})} \mc I(0,\rho_0) \longrightarrow 0.
                \end{align}
            From \eqref{endddd}, we have $\mc I(t,\rho_t) = O(t^{\frac{\alpha}{\kappa(m)} -1})$ but in fact, we have a better rate, that is $\mc I(t,\rho_t) = O(t^{\frac{2\alpha}{\kappa(m)} -2})$.
        \end{cor}
        \begin{proof}
            The first inequality in \eqref{endddd} follows from Corollary \ref{corollary: 1}, while the second is already shown in \eqref{forcitehere}. We prove the last statement. Since the second term on the right handside of the second inequality in \eqref{endddd} is going to 0 at exponential rate, the first term governs the convergence speed to 0 of $\mc I(t,\rho_t)$, which is proportional to $t^{\frac{\alpha}{\kappa(m)} -1}$. Denote $\mathcal I_t \df  \sqrt{\mathcal I(t,\rho_t)}$. From \eqref{ineq: where we can improve rate of It}, we have
            \begin{align*}
                \partial_t \mc I(t,\rho_t)  &= -\mc G (t,\rho_t) + \dot \beta_t \sum_{x\in V}U(x)(\rho_t(x) - \nu_t(x))\ell(x) \\ 
                &=  -\mc G (t,\rho_t) + \dot \beta_t \sum_{x\in V}(U(x) - \ell[U])(\rho_t(x) - \nu_t(x))\ell(x)  \\
                &\leq -\mc G (t,\rho_t) + \dot \beta_t \sqrt{\sum_{x \in V} \ell(x) (U(x)-\ell[U])^2 \times \sum_{x\in V} \ell(x)(\rho_t(x) -\nu_t(x))^2} \\
                &= -\mc G (t,\rho_t) + \dot \beta_t \sqrt{\text{Var}[U]} \| \rho_t - \nu_t\|_{\mathbb L^2(\ell)} \\ 
                &\leq -\lambda \frac{\varphi''(1/\ell_\wedge)}{\varphi''(\rho_{t,\wedge})}\mathcal I(t,\rho_t) + \dot \beta_t \sqrt{2\text{Var}[U] \mc I(t,\rho_t)}.
            \end{align*}
            Dividing the last quantity by $\sqrt{ \mc I(t,\rho_t)}$, we arrive at 
            \begin{align*}
                2\partial_t \mc I_t \leq -\lambda \frac{\varphi''(1/\ell_\wedge)}{\varphi''(\rho_{t,\wedge})}\mathcal I_t + \dot \beta_t  \sqrt{2 \text{Var}[U]},
            \end{align*}
            which looks exactly like \eqref{ineq:improve !} except $\text{osc }U$ is replaced by  $\sqrt{2 \text{Var}[U]}$ and note that $\sqrt{2 \text{Var}[U]} \leq \sqrt{2 } \text{osc } U$. Hence, following the same lines of proof in Theorem \ref{theo: main of ourpaper}, we have that $\mc I_t = O(t^{\frac{\alpha}{\kappa(m)} -1})$ which implies $\mc I(t,\rho_t) = O(t^{\frac{2\alpha}{\kappa(m)} -2})$.
        \end{proof}

    \subsection{Nonlinear Markov interpretation} \label{sub: inhomo inter}
    Similar to Section 4.3, we can reinterpret the dynamic \eqref{Nonlinear: inhomo dynamics} in terms of a time-inhomogeneous nonlinear Markov dynamic satisfying 
        \begin{align} \label{Nonlinear: Markov inhomo inter}
            \forall t > 0, \qquad \dot \mu_t = \mu_t L_{t,\rho_t} = \mu_t Q_{t,\rho_t} ,
        \end{align}
    where the generator $ L_{t,\rho}$ is defined by
        \begin{align}\label{Nonlinear: inhomo generator}
            \forall t >0, \ \forall \rho \in \mathcal D_+(V), \ \forall x \neq y, \quad L_{t, \rho}(x,y) \df  L(x,y) \frac{\theta(\rho(x),\rho(y))}{\rho(x)}\Big( \nabla [\beta_t U + \varphi'\circ \rho] (x,y) \Big)_-, 
        \end{align}
    or more explicitly 
        \begin{align} \label{old gen: inhomo + reducible}
            \forall t>0, \ \forall \rho \in \mathcal D_+(V), \ \forall x \neq y, \ L_{t, \rho}(x,y) = L(x,y) \left( \frac{\rho(y) - \rho(x)}{\rho(x)[\varphi'(\rho(y)) - \varphi' (\rho(x))] }\beta_t(U(y) - U(x)) + \frac{\rho(y)}{\rho(x)} -1  \right)_-,
        \end{align}
    and the generator $ Q_{t,\rho}$ is defined by 
        \begin{align} \label{new gen: inhomo + irre}
            \forall t >0, \ \forall \rho \in \mc D_+(V), \ \forall x \neq y, \quad Q_{t, \rho} (x,y) \df L(x,y)\Big(1 + \frac{\rho(y) - \rho(x)}{\rho(x)[\varphi'(\rho(y)) - \varphi'(\rho(x)) ]}\beta_t (U(y)-U(x))_-\Big)
        \end{align}
    Observe that $L_{t,\rho}= L_{\beta_t,\rho}$ and $ Q_{t,\rho} = Q_{\beta_t,\rho}$ , where $L_{\beta,\rho}$ and $Q_{\beta,\rho}$ are given in \eqref{Nonlinear: generator} and \eqref{generator: new interpretation, explicit}, and therefore the dependence on $t$ of $L_{t,\rho}$ and $Q_{t,\rho}$ is through $\beta_t$ namely on the choice of the temperature schedule.
    
    With access to these Markov interpretations, we can effectively utilize an interacting particle system to approximate a Markov process denoted as $X= (X_t)_{t\geq 0}$, taking values in $V$ and satisfying $\forall t \geq 0, \ \text{Law}(X_t) \equiv \mu_t$, where $\mu_t$ is the unique solution to \eqref{Nonlinear: Markov inhomo inter} (recall that the existence and uniqueness of $(\mu_t)_{t\geq 0}$ has been established in Section 5.2). We call it Swarm algorithm since the process $X= (X_t)_{t\geq 0}$ evolves and interacts with its law $\mu_t$ at all times $t\geq 0$ in \eqref{Nonlinear: Markov inhomo inter} through the generator $L_{t,\rho}$ or $Q_{t,\rho}$. An additional rationale behind the chosen name lies in the essence of our algorithm itself. As an approximation method outlined previously, our algorithm relies on the interaction of particles to approximate the behavior of the Markov process $X$. This interaction goes through the closest neighbors as it is observed in biological systems of birds or fish. This inherent similarity provides further validation for the designation ``Swarm algorithm". The details of sampling techniques, including the homogeneous case, can be found in the appendix.

    We now delve into investigating the behavior of the Markov generators $t \mapsto Q_{t,\rho_t}$ and $ t\mapsto L_{t,\rho_t}$ as $t$ becomes large. Specifically, let us fix two states $x \neq y$ in $V$ and suppose we are at the state $x$. The following observations will shed some lights on the behavior of the values $Q_{t,\rho_t}(x,y)$ and $L_{t,\rho_t}(x,y)$ for large time $t$. If $\min U < U(y) \leq U(x)$ then $(U(y)-U(x))_- = U(x)-U(y)$ and recall from the end of the proof of Theorem \ref{theo: main of ourpaper} that $\lim_{t \rightarrow \infty} \rho_t(z)/\nu_t(z) = 1$ if $z \notin \mathcal M(U)$. Since for $t$ large enough $\rho_t(x), \rho_t(y) < 1$, we have from Theorem \ref{theo: properties of nu_t}
    \begin{align*}
        \lim_{t \rightarrow \infty} \frac{\rho_t(y)^{m-1} }{\rho_t(x)^{m-1}} = \lim_{t \rightarrow \infty} \frac{\nu_t(x)^{1-m} \beta_t }{\nu_t(y)^{1-m} \beta_t} = \frac{U(y) - \min U}{U(x) - \min U},
    \end{align*}
    thus,
    \begin{align*}
        \frac{\rho_t(y) - \rho_t(x)}{\rho_t(x)[\varphi'(\rho_t(y)) - \varphi'(\rho_t(x)) ]}\beta_t (U(y)-U(x))_- &= (m-1) \frac{ \frac{\rho_t(y)}{\rho_t(x)} -1}{ \frac{\rho_t(y)^{m-1}}{\rho_t(x)^{m-1}} - 1} \rho_t(x)^{1-m}\beta_t (U(x) - U(y)) \\
        &\longrightarrow \left(\frac{U(y) - \min U}{U(x) - \min U}\right)^{\frac{1}{m-1}} -1 \label{eq:comment} \numberthis
    \end{align*}
    by Theorem \ref{theo: properties of nu_t} too and so $Q_{t,\rho_t}(x,y) \rightarrow L(x,y) \left(\frac{U(y) - \min U}{U(x) - \min U}\right)^{\frac{1}{m-1}} $ as $t \rightarrow \infty$. The same analysis shows   $L_{t,\rho_t}(x,y) \rightarrow 0 $ as $t \rightarrow \infty$. If $U(x)>U(y) = \min U$ then the above limit in \eqref{eq:comment} is $+\infty$ because $\rho_t(y)/\rho_t(x) \rightarrow + \infty$ and $(m-1)\rho(x)^{1-m}\beta_t (U(x)-U(y)) \rightarrow -1$ while the denominator goes to $-1$. Thus if $L(x,y) >0$, $Q_{t,\rho_t} (x,y) \rightarrow + \infty$, which indicates that for large time $t$, the probability of transitioning from $x$ to $y$ in the next jump approaches $1$ when $y$ is the unique state in $\mathcal M(U)$ that is a neighbor of $x$. If among the neighbors of $x$ there are more than one state in $\mathcal M(U)$ then the process jumps very fast to one of them. However, in this case, we cannot ascertain whether $ L_{t,\rho_t}(x,y) \rightarrow 0 $ or $ L_{t,\rho_t}(x,y) \rightarrow +\infty$ based solely on the asymptotic behavior of $(\rho_t)_{t\geq 0}$ from the proof of Theorem \ref{theo: main of ourpaper}. In contrast, if we replace $\rho_t$ by $\nu_t$ then it holds that $L_{t,\nu_t} =0$ for all times $t \geq0$.
    Now if $\min U = U(x) < U(y)$, for large time $t$, we have
        \begin{align*}
            \frac{\rho_t(y) - \rho_t(x)}{\rho_t(x)[\varphi'(\rho_t(y)) - \varphi'(\rho_t(x))]}\beta_t (U(y) - U(x)) &= (1-m) \frac{\left(\frac{\rho_t(y)}{\rho_t(x)} -1\right) \rho_t(y)^{1-m} \beta_t (U(y) - \min U )}{\rho_t(y)^{1-m} -1 - (1-m)\rho_t(y)^{1-m}(\rho_t(x)-1) } \\
            &\longrightarrow 1
        \end{align*}
    because $\rho_t(y) \rightarrow 0$ (so the denominator goes to 0), $\rho_t(x) \rightarrow \frac{\zeta_\infty(x)}{\ell(x)} = (\sum_{z \in \mathcal M(U)} \ell(z))^{-1}$ by Theorem \ref{theo: properties of nu_t} $i)$ and $ii)$ and $\lim_{t \rightarrow + \infty} \frac{\rho_t(y)}{\nu_t(y)}  =1$ so that 
        \begin{align*}
            \lim_{t\rightarrow + \infty} \rho_t(y)^{1-m} \beta_t = \lim_{t\rightarrow + \infty} \nu_t(y)^{1-m} \beta_t = \frac{1}{(1-m)(U(y) - \min U)}.
        \end{align*}
    Thus, $L_{t,\rho_t}(x,y) \rightarrow 0$ while $Q_{t,\rho_t}(x,y) \rightarrow L(x,y)$ as $t\rightarrow \infty$ in this case. In conclusion, the two generators behave differently for large time $t$. Driven by $Q_{t,\rho_t}$, the process can escape outside $\mathcal M(U)$ but have tendency to return immediately to $\mathcal M(U)$ while it is uncertain if the same phenomena happens for $L_{t,\rho_t}$. {Lastly, if the process under the generator $Q_{t,\rho_t}$ is currently outside $\mathcal M(U)$, it can move more freely around than the process under the reducible generator $L_{t,\rho_t}$ because $Q_{t,\rho_t}$ is irreducible.

    \begin{rem}
        In practice, we can employ the following hybrid generator 
        \begin{align*}
            \forall t\geq0, \ \forall \rho \in \mc D_+(V), \qquad A_{t,\rho} := (1-a_t)L_{t,\rho} + a_t Q_{t,\rho},
        \end{align*}
    where $a: \R_+ \mapsto (0,1)$ is a continuous function properly chosen. For instance, $a$ can be a constant in $(0,1)$ or be such that $a_t Q_{t,\rho_t}(x,y) \rightarrow + \infty$ if $U(x) > U(y) = \min U$ as $t \rightarrow +\infty$. In this manner, it still holds that $\dot \mu_t = \mu_t A_{t,\rho_t}$. The first advantage of using this hybrid generator is that it gives us the option of tuning the parameter $a$ in computer simulations, which could enhance performance. Secondly, it preserves a crucial property of $Q_{t,\rho_t}$: if $x \in \mathcal M(U), \ y\notin \mathcal M(U)$ with $L(x,y) >0$ then $A_{t,\rho_t}(x,y) \rightarrow 0$ and $A_{t,\rho_t}(y,x) \rightarrow + \infty$. Thus the probability that the process jumps from $y$ to a global minimizer from its neighborhood is converging to $1$ just like under $Q_{t,\rho_t}$ for large time $t$. Lastly, as we expect $L_{t,\rho_t} \rightarrow 0$, the time spent at $x \in \mathcal M(U)$ is longer because $|A_{t,\rho_t} (x,x)| \approx a_t |Q_{t,\rho_t}(x,x)| \leq |Q_{t,\rho_t}(x,x)| = |L(x,x)|$, which is helpful for computer simulations. However, it is uncertain if the irreduciblity of $A_{t,\rho_t}$ would bring about better performance in practice at all.
    \end{rem}}
    
\section{Simulation} \label{Simulation}
    In this section we present some simulations of our algorithm in both homogeneous and inhomogeneous cases and compare the generators given in Sections \ref{sub: homo inter} and \ref{sub: inhomo inter}. Consider the state space $V:= \{ 0,1,...,19 \}$ and endow $V$ with the irreducible generator $L$ given by
        \begin{align*}
           \forall i \in V, \qquad L(i,i+1) = L(i,i-1) = 1, \qquad \text{and} \qquad \forall j \notin V \setminus \{i-1,i,i+1 \}, \qquad L(i,j) = 0,
        \end{align*}
    with the convention that $19+1 \equiv 0$ and $0-1 \equiv 19$. $L$ then admits the uniform distribution on $V$ as its unique invariant distribution $\ell$. Let 
        \begin{align*}
            \forall x \in \R, \qquad u(x) \df \frac{x^2}{10} + 2(\cos(3x) +\sin(7x)),
        \end{align*}
    we choose the function $U: V \mapsto \R$ to be minimized as follows:
        \begin{align*}
            \forall i \in V, \qquad U(i) \df u(-0.6+i/5.5).
        \end{align*}
    The graph of $U$ is given in Figure \ref{fig: U}. Observe that $U$ has four local minima and only one global minimum with $i = 7$ is the unique global minimizer. Finally, in both homogeneous and inhomogeneous cases, we use 50 particles in the system.

    The details of the codes can be found in this link: https://github.com/nhatthangle/Swarm-Algorithm.git.
    The strange picks in the following pictures are possibly explained by the fact that the evaluations are made after random jump times and not deterministic fixed times.
    \begin{figure}[H] 
        \centering
        \includegraphics[width=8cm, height=7cm]{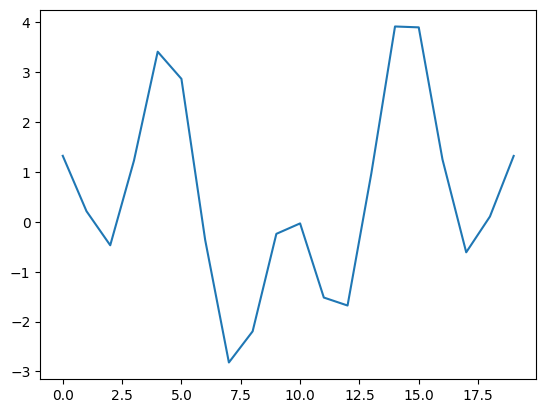}
        \caption{Function $U$}
        \label{fig: U}
    \end{figure}

    \subsection{Homogenenous case}
        We will refer to the generators given in \eqref{Nonlinear: generator explicit} and \eqref{generator: new interpretation, explicit} as the first generator and the second generator, respectively. Recall that the first generator is not irreducible while the second generator is. To facilitate a comparison of the generators, we select a specific time interval, which we shall choose to be 5 minutes, and allow the particle system to evolve until this interval elapses. Initializing with a uniform distribution, we opt for $\beta =5$ as our chosen parameter. With this value of $\beta$, the invariant measure $\eta_\beta$ has $\eta_\beta(\{ 6,7,8\}) \approx 0.65$ (recall that $7 = \argmin_{i \in V} U$). 
    
        Figure \ref{fig: l2 generator 1} depicts the graph illustrating the $\mathbb L^2(\ell)$-distance between the empirical measure of the particle system and the invariant measure $\eta_\beta$ at each transition (with the x-axis representing the number of transitions), utilizing the first generator. Note that each jump time is random, but we care more about the transitions than time. Meanwhile, Figure \ref{Fig: l2 generator 2} presents the analogous graph employing the second generator. Notably, we observe that the distance over time using the second generator exhibits greater fluctuation, attributed to the irreducible nature of this generator.

            \begin{figure}[H]
               \begin{minipage}{0.48\textwidth}
                 \centering
                 \includegraphics[width=.9\linewidth]{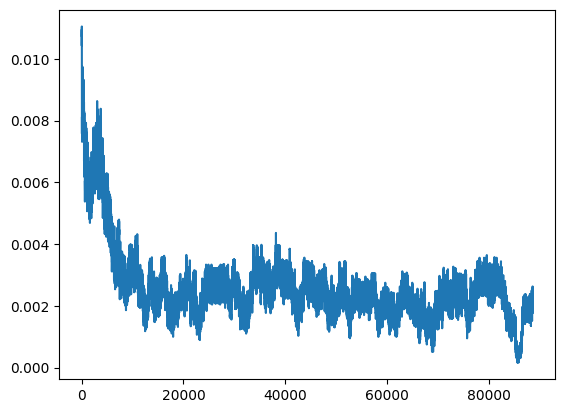}
                 \caption{$L^2(\ell)$-distance using first generator} 
                 \label{fig: l2 generator 1}
               \end{minipage}\hfill
               \begin{minipage}{0.48\textwidth}
                 \centering
                 \includegraphics[width=.9\linewidth]{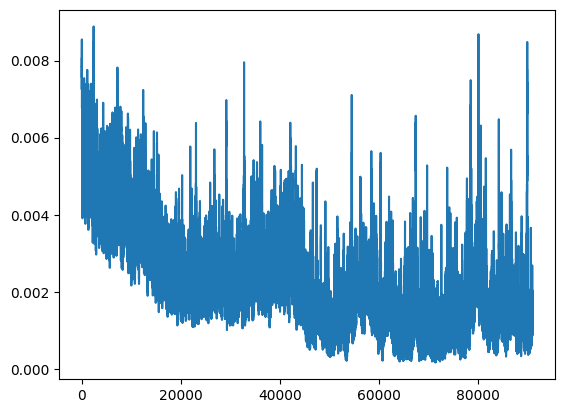}
                 \caption{$L^2(\ell)$-distance using second generator}
                 \label{Fig: l2 generator 2}
               \end{minipage}
            \end{figure}
             Figures \ref{fig: hist old 16}, \ref{Fig: hist new 16} depict the evolution of histograms representing the empirical measures (progressing from left to right), with each picture generated after 1/16 of the chosen time interval (which amounts to 5 minutes); the initial and final pictures correspond to the start and ending of the simulation, respectively. The two figures look relatively similar, contrary to Figures \ref{fig: l2 generator 1} and \ref{Fig: l2 generator 2}.

            \begin{figure}[H] 
               \begin{minipage}{0.48\textwidth}
                 \centering
                 \includegraphics[width=\linewidth]{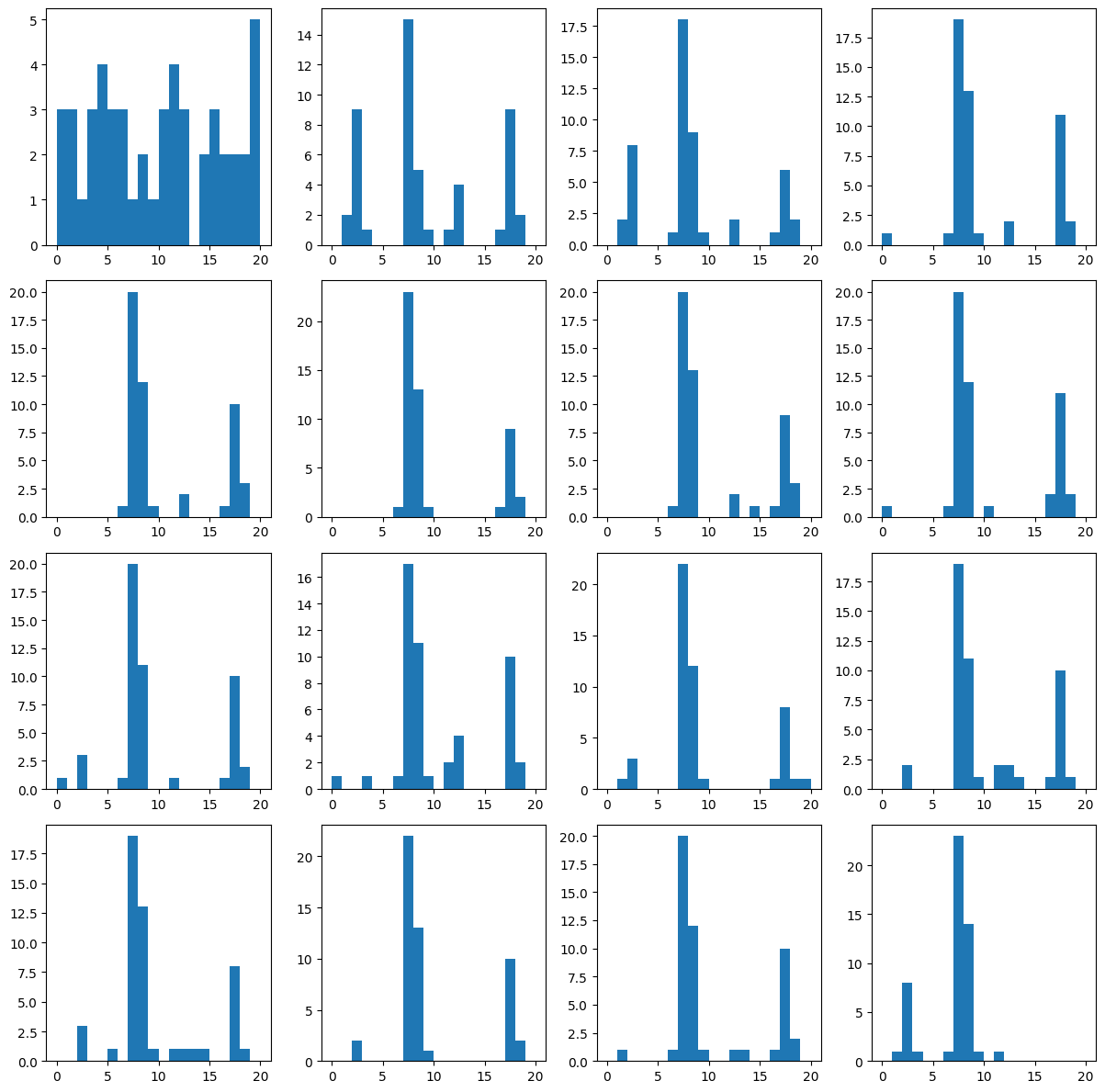}
                 \caption{Empirical measure over time (1st generator)} 
                 \label{fig: hist old 16}
               \end{minipage}\hfill
               \begin{minipage}{0.48\textwidth}
                 \centering
                 \includegraphics[width=\linewidth]{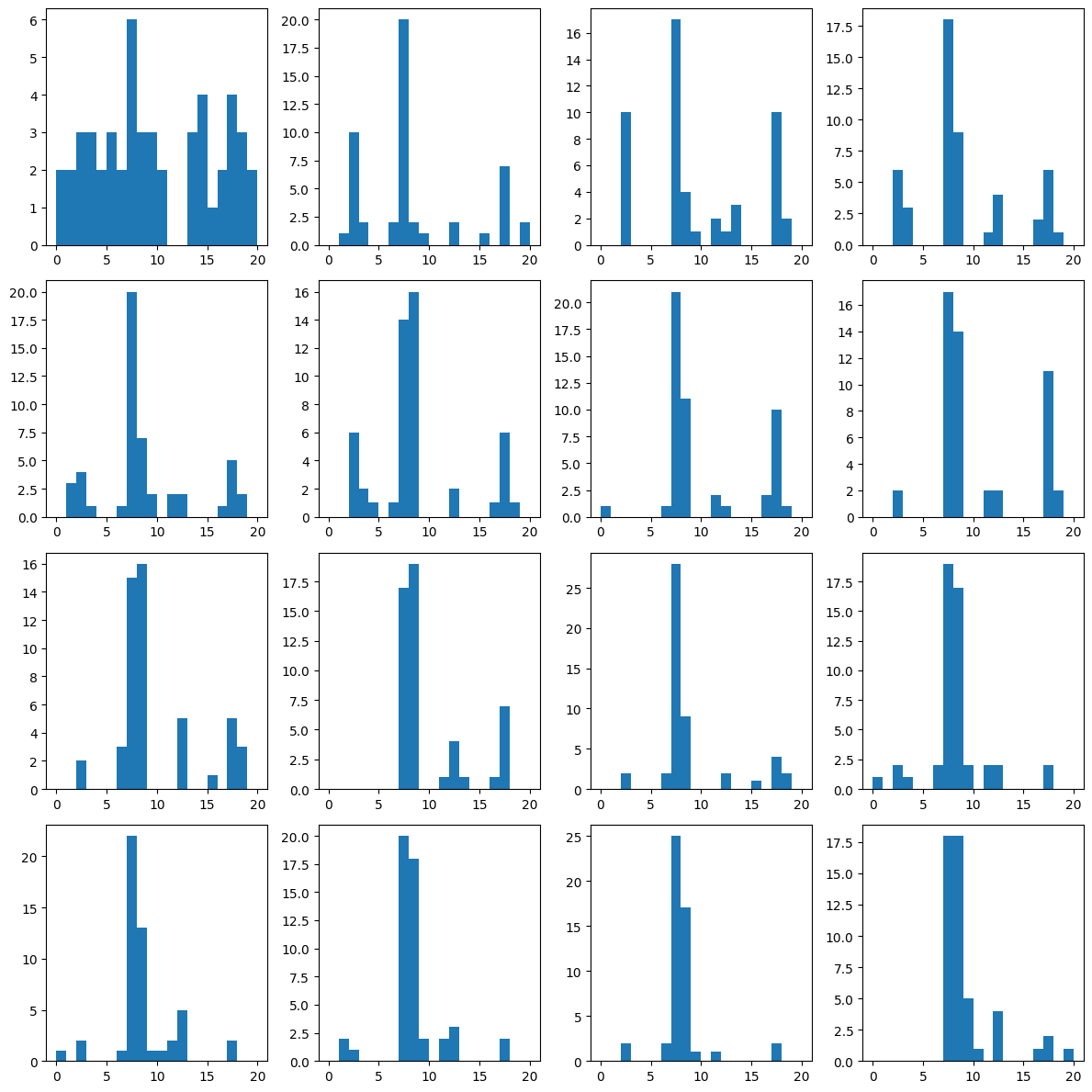}
                 \caption{Empirical measure over time (2nd generator)}
                 \label{Fig: hist new 16}
               \end{minipage}
            \end{figure}
      
    \subsection{Inhomogeneous case}
        In this subsection, we perform the Swarm algorithm introduced in Section \ref{sub: inhomo inter}. Once more, we undertake a comparative analysis of the generators \eqref{old gen: inhomo + reducible} and \eqref{new gen: inhomo + irre}, which we will refer to as the first and the second inhomogeneous families of generators, respectively. We choose $\alpha = 1/4$ and $t_0 = 1$, so that $\beta_t = (1+t)^{\frac{1}{4}} -1$. We denote the empirical measure by $\hat \mu_t$, which is given by
            \begin{align*}
                \hat \mu_t (\cdot) = \sum_{i = 1}^{50} \mathbbm 1_{ \{X_i(t) \} }(\cdot),
            \end{align*}
        where $X(t) = (X_1(t),...,X_{50}(t)$) is the particle system we are using and for all $i \in V$ and $A \subset V$, $\mathbbm 1_{ \{ i \} } (A) = 1$ if $i \in A$ and $0$ otherwise. We denote $\hat \rho_t$ as the empirical density of $\hat \mu_t$ with respect to $\ell$. In what follows, we fix a 2-hour period for our simulations.

        Figures \ref{fig: old l2-inv-emp} and \ref{Fig: new l2-inv-emp} illustrate the $\mathbb L^2(\ell)$ distance between the empirical density $\hat \rho_t$ and the ``instantaneous" density $\nu_t$ at each transition (with the x-axis representing the number of transitions) over a 2-hour duration. It is worth noting that within an equivalent time frame, the system governed by the second generator experiences more than eight times the number of transitions compared to that governed by the first generator. This is due to the asymptotic behavior of $L_{t,\rho_t}$ and $Q_{t,\rho_t}$ discussed in Section \ref{sub: inhomo inter}. Also, due to the random nature of the particle system, at some point, a particle will move outside $\mathcal M(U)$ and causes a few upward spikes as shown in the pictures.

            \begin{figure}[H]
               \begin{minipage}{0.48\textwidth}
                 \centering
                 \includegraphics[width=\linewidth]{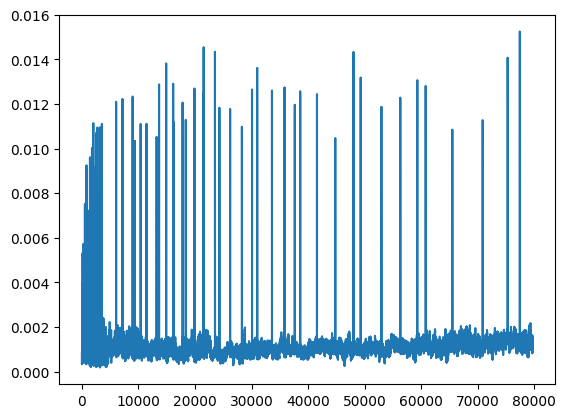}
                 \caption{$\| \hat \rho_t - \nu_t \|_{\mathbb L^2(\ell)} $ (1st generator)} 
                 \label{fig: old l2-inv-emp}
               \end{minipage}\hfill
               \begin{minipage}{0.48\textwidth}
                 \centering
                 \includegraphics[width=\linewidth]{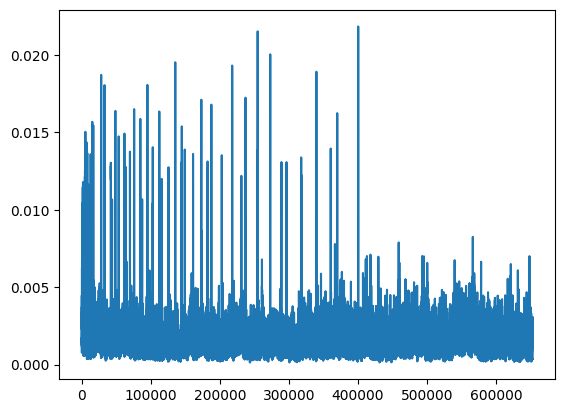}
                 \caption{$\| \hat \rho_t - \nu_t \|_{\mathbb L^2(\ell)} $ (2nd generator)}
                 \label{Fig: new l2-inv-emp}
               \end{minipage}
            \end{figure}
        
        Figures \ref{fig: old l2-emp-dirac} and \ref{Fig: new l2-emp-dirac} show the $\mathbb L^2(\ell)$-distance between the empirical density and the density of Dirac measure at state $ 7 = \argmin_{i \in V} U$. With a slight abuse of notation, we employ $\mathbbm 1_{\{ 7\}}$ to represent both the Dirac probability measure and its density with respect to $\ell$. It is evident from these figures that the particle system stemming from the second generator displays more pronounced fluctuations, indicative of its irreducible nature. Conversely, the particle system originating from the first generator appears more stable, with comparatively smaller variance.

        Figures \ref{fig: old 16} and \ref{Fig: new 16}  display the histograms illustrating the evolution over time of the empirical distribution $\hat \mu_t$. Each histogram, progressing from left to right, represents a snapshot taken after $1/16$ of the total time, which equates to 2 hours. The initial frames in both figures portray a sample of 50 particles drawn from the uniform distribution over $V$. The final frames in both figures depict the terminal positions of the particle systems emerging from the two generators.
            \begin{figure}[H]
               \begin{minipage}{0.48\textwidth}
                 \centering
                 \includegraphics[width=\linewidth]{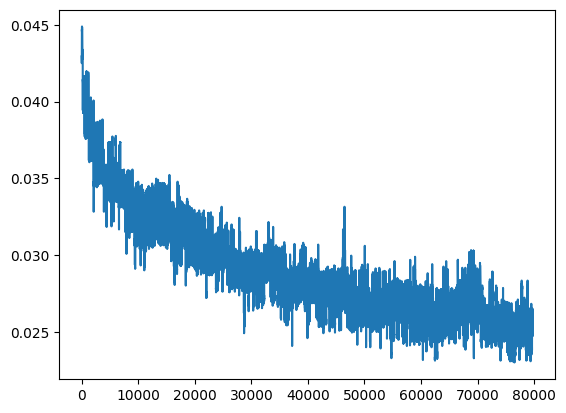}
                 \caption{$\| \hat \rho_t - \mathbbm 1_{\{7\}} \|_{\mathbb L^2(\ell)} $ (1st generator)} 
                 \label{fig: old l2-emp-dirac}
               \end{minipage}\hfill
               \begin{minipage}{0.48\textwidth}
                 \centering
                 \includegraphics[width=\linewidth]{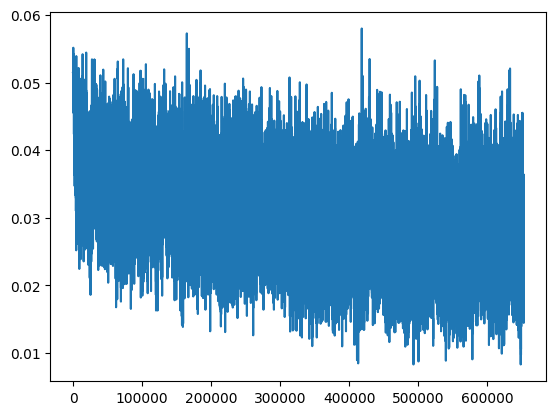}
                 \caption{$\| \hat \rho_t - \mathbbm 1_{\{7\}} \|_{\mathbb L^2(\ell)} $ (2nd generator)}
                 \label{Fig: new l2-emp-dirac}
               \end{minipage}
            \end{figure}

            \begin{figure}[H]
               \begin{minipage}{0.48\textwidth}
                 \centering
                 \includegraphics[width=\linewidth]{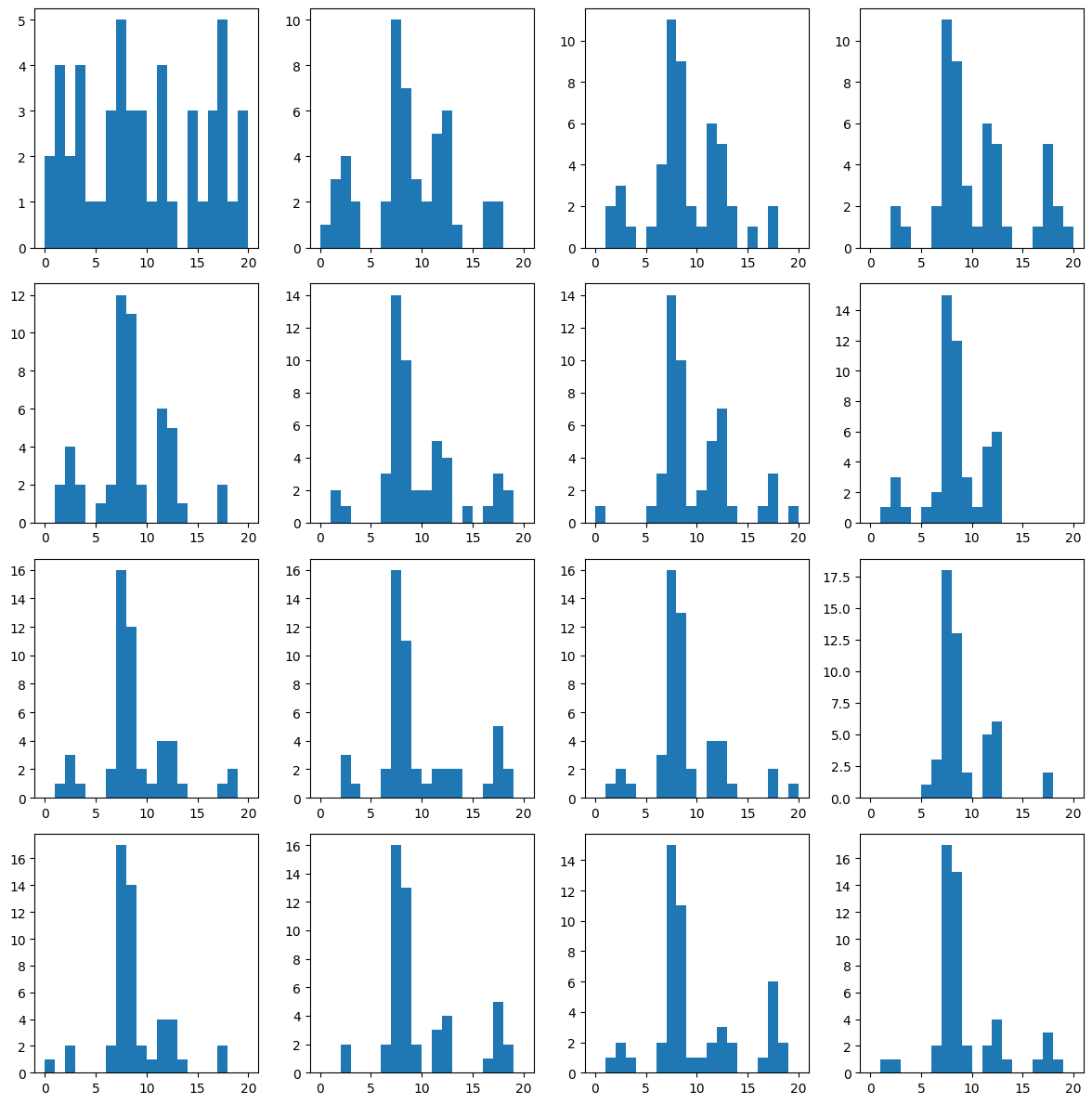}
                 \caption{$\| \hat \rho_t - \mathbbm 1_{\{7\}} \|_{\mathbb L^2(\ell)} $ (1st generator)} 
                 \label{fig: old 16}
               \end{minipage}\hfill
               \begin{minipage}{0.48\textwidth}
                 \centering
                 \includegraphics[width=\linewidth]{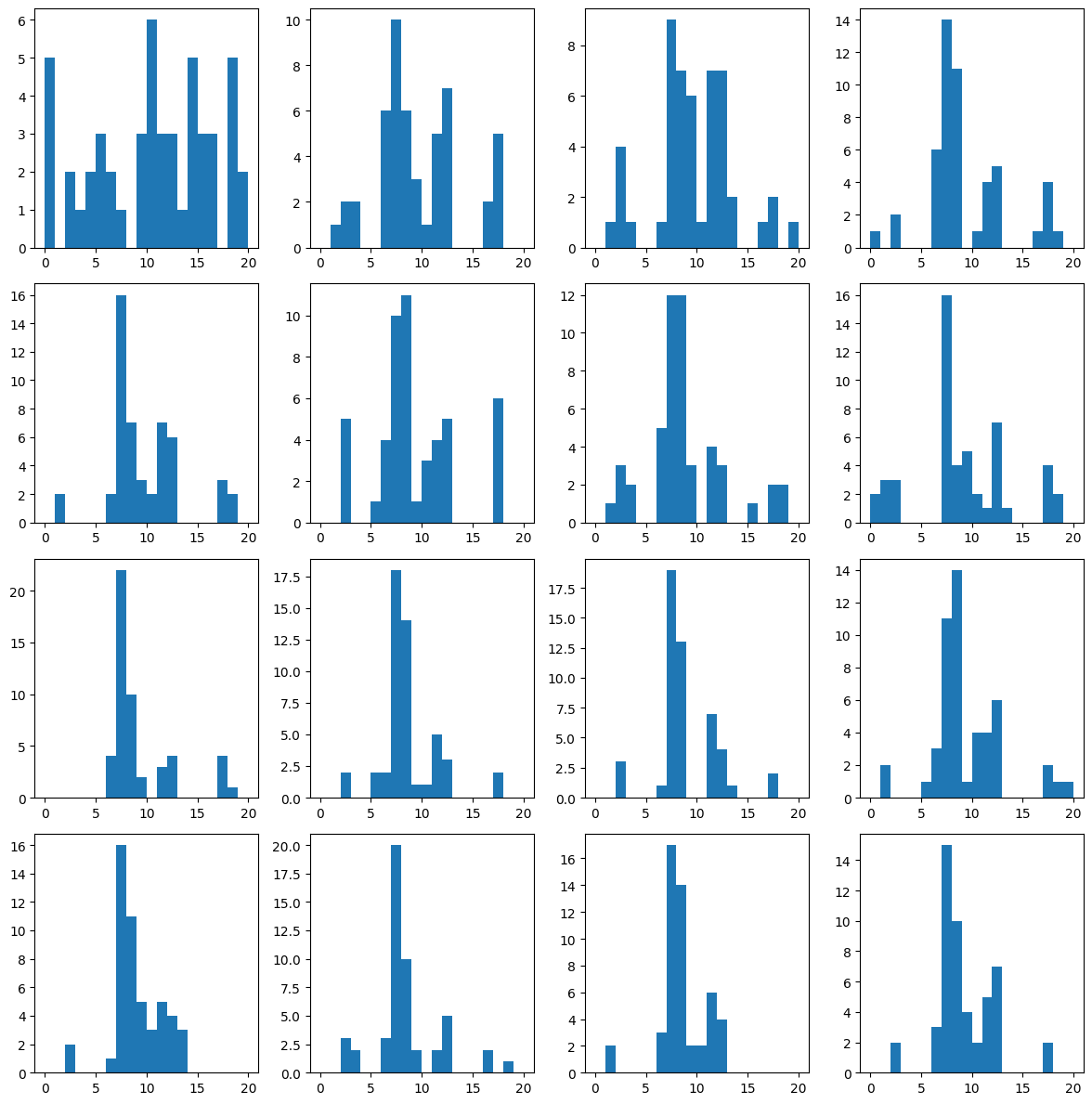}
                 \caption{$\| \hat \rho_t - \mathbbm 1_{\{7\}} \|_{\mathbb L^2(\ell)} $ (2nd generator)}
                 \label{Fig: new 16}
               \end{minipage}
            \end{figure}
    
        From figures presented above, it appears that the algorithm employing the first generator outperforms the one utilizing the second generator over a fixed time period, particularly concerning variance and convergence characteristics. This is notable despite the fact that the particle system governed by the first generator undergoes fewer transitions. However, it is imperative to note that in practical terms, the number of computations often outweighs the consideration of time alone.

\section{Conclusion}
  To globally minimize a function $U$ given on a finite set $V$, we extended it into a penalized functional ${\mathcal U}_\beta$ on the set ${\mathcal P}(V)$ of probability measures on $V$, where $\beta$ is a non-negative parameter. The larger $\beta$ is, the more concentrated on the global minimizers of $U$  is the unique global minimizer $\eta_\beta$ of ${\mathcal U}_\beta$.
Another ingredient entering in the functional ${\mathcal U}_\beta$ is a strictly convex function $\varphi\st(0,+\iy)\ri\RR_+$. Following Erbar and Maas \cite{Erbar}, we endow ${\mathcal P}(V)$ with a Riemannian structure (except for the regularity), strongly related to $\varphi$. It enables us  to consider the gradient descent associated to ${\mathcal U}_\beta$, and
as it can be expected, the probability measure-valued dynamical system obtained in this way is converging to $\eta_\beta$ as time goes to infinity.
This result leads us to consider a time-inhomogeneous version $(\mu_t)_{t\geq 0}$ of this dynamical system, where the parameter $\beta$ evolves with time, with $\beta_{t}$ growing to infinity as the time $t$ becomes larger and larger.  Our main result gives conditions on the evolution $(\beta_t)_{t\geq 0}$ insuring that for large time $t\geq 0$, $\mu_t$ concentrates on the global minimizers of $U$.
The proof is based on a new functional inequality. 
So while the above considerations are an adaption to the finite setting of the general method described in \cite{Bolte} for the global minimization of Morse functions $U$ on compact Riemannian manifolds $M$, we are able to go much further by relaxing the disappointing geometric restriction imposed in \cite{Bolte} that $M$ should be a circle.
Another interesting feature of this approach is that the dynamical system $(\mu_t)_{t\geq 0}$ can be interpreted as the time-marginal distributions of a non-linear and time-inhomogeneous Markov process on $V$, which can thus be approximated by interacting particle systems.
The paper ends with an example of such a numerical implementation. We hope to investigate quantitatively the quality of this particle approximation in future works.

\appendix

\section{On Markov-Riemann structures}

Our purpose here is to see why the Maas framework \cite{Maas1} based on the introduction of a function $\theta$  is too restrictive to recover the traditional 
Metropolis algorithm as a gradient descent flow.
\par\me
Let us extend the Maas framework following \cite{miclo:hal-04094968}. On the finite set $V$, consider $\cP_+(V)$ and $\cLi$, respectively the set of positive probability measures on $V$
and the set of irreducible Markov generators on $V$.
Assume we are given a locally Lipschitz (with respect to the total variation) mapping 
\bqn{Lm}
K\st \cP_+(V)\ni \mu&\mapsto& K_\mu\in\cLi\eqn
such that for any $\mu\in \cP_+(V)$, $K_\mu$ is 
reversible with respect to $\mu$. \par
A form (here we implement  Remark 5 of  \cite{miclo:hal-04094968}, replacing in the terminology ``vector fields'' by ``forms'') on $V$ is an anti-symmetric mapping $F\st V\times V\ni(x,y)\mapsto F(x,y)\in\RR$, i.e.\ satisfying
\bq
\fo (x,y)\in V\times V,\qquad F(x,y)&=&-F(y,x)\eq
\par
Denote $\cV(V)$ the set of forms on $V$. 
We endow it with the following scalar products, one for each given $\mu\in\cPp$:
\bq
\fo F_1,F_2\in\cV(V),\qquad \lan F_1,F_2\ran_{\mu\ltimes K_\mu}&\df& {\f12\sum_{(x,y)\in V\times V} F_1(x,y)F_2(x,y)\,\mu(x)K_\mu(x,y)}\eq
(the corresponding Euclidean norm will be denoted $\lVe\,\cdot\,\rVe_{\mu\ltimes L_\mu}$).
\par
Let $\cF(V)$ be the set of real functions defined on $V$. 
We equally endow it with the family of scalar products    corresponding to the $\LL^2(\mu)$ space, namely for any $\mu\in\cPp$:
\bq
\fo f_1,f_2\in\cF(V),\qquad 
\lan f_1,f_2\ran_\mu&\df&\sum_{x\in V} f_1(x)f_2(x)\,\mu(x)
\eq
(the corresponding Euclidean norm will be denoted $\lVe\cdot\rVe_\mu$).\par
Consider  the mapping $\dd$ defined by
\bq
\cF(V)\ni f&\mapsto \dd[f]\df (f(y)-f(x))_{(x,y)\in V^2}\in \cV(V)\eq
(it will play the role of the exterior derivative in differential geometry).
\par
The image $\dd(\cF(V))$ is denoted $\cE(V)$ (it corresponds to the set of exact forms in differential geometry and was called the set of gradient fields in \cite{miclo:hal-04094968}).\par
For any fixed $\mu\in\cPp$, let $\dd^*_\mu$ (called the $\mu$-divergence in \cite{miclo:hal-04094968}) be the dual operator to $-\dd$ with respect to the Euclidean structures associated to the 
scalar products $\lan\cdot,\cdot\ran_\mu$ and $\lan\cdot,\cdot\ran_{\mu\ltimes K_\mu}$.
More explicitly, we compute that
\bq
\fo F\in \cV(E),\,\fo x\in V,\qquad \dd_\mu^*[F](x)&=&\sum_{y\in V} K_\mu(x,y)F(x,y)\eq
and it appears that
\bq \fo \mu\in\cPp,\qquad K_\mu&=&\dd^*_\mu\circ \dd\eq
\par
To any $\mu\in\cPp$ and $F\in \cV(V)$, we associate the Markov generator $K_{\mu,F}$ given by
\bq
\fo x\neq y\in V,\qquad K_{\mu,F}(x,y)&\df& K_\mu(x,y) F_+(x,y)\eq
where $F_+(x,y)$ stands for the positive part of $F(x,y)$.\par
According to (15) in \cite{miclo:hal-04094968},  $K_{\mu, F}$ and $\dd^*_\mu$ are related through
\bqn{fond}
\nonumber\fo f\in\cF(V),\qquad \mu[K_{\mu,F}[f]]&=&-\mu[f \dd_\mu^*[F]]\\
&=&\lan \dd f, F\ran_{\mu\ltimes K_\mu}
\eqn\par
To any $F\in \cV(V)$ and $\mu\in\cPp$, we associate the semi-flow  $
(\cS_F(\mu,t))_{t\in[0,\tau(F,\mu))}$ solution of
\bqn{sf}
\fo t\in[0,\tau(F,\mu)),\qquad \dot\mu_t&=&\mu_t K_{\mu_t,F}\eqn
starting with $\mu_0=\mu$. The time $\tau(F,\mu)>0$ is assumed to be the explosion time of the above evolution equation, in the sense that $\lim_{t\ri\tau(F,\mu)_-}\mu_t$ does not exists in $\cPp$.
\par
\begin{rem}
Note that if the mapping defined in \eqref{Lm} is globally Lipschitz (and in particular bounded), then we get $\tau(F,\mu)=+\iy$ for any $\mu\in\cPp$ and $F\in \cV(V)$. It follows that $\cS_F$ can be seen as a semi-group acting on $\cPp$, in the sense that for any $F\in \cV(V)$ and $\mu\in\cPp$,
\bq
\fo t,s\geq 0,\qquad
\cS_F(\cS_F(\mu,t),s)&=&\cS_F(\mu,t+s)
\eq\par
Let us explain the interest in our finite setting of $(\cS_F(\delta_x,t))_{t\geq 0}$, where $\delta_x$ is the Dirac mass at $x\in V$.
Consider $M$ a compact Riemannian manifold and let $\omega$ be a differential form on $M$. The Riemannian structure enables us to transform it into a vector field $v$.
For any $x\in M$, we can  consider the flow $(x(t))_{t \in \RR}$ generated by $v$, i.e.\ the solution of the ordinary differential equation 
\bq
\lt\{\begin{array}{rcl}
x(0)&=&x\\[2mm]
\dot{x}(t)&=& v(x(t)),\quad \fo t\in\RR
\end{array}\rt.
\eq
\par
Then $(\cS_F(\delta_x,t))_{t\geq 0}$ for $x\in V$ is an analogue of $(\delta_{x(t)})_{t\geq 0}$ for $x\in M$, when $\omega$ is replaced by $F$.
There are two important differences between these continuous and finite  settings. First the flow has to be replaced by a semi-flow, only defined for non-negative times.
Secondly, our semi-flow $(\cS_F(\delta_x,t))_{t\geq 0}$ does not stay in the set of Dirac masses but has to spread, taking values in $\cPp$ (the only exception being the case of the zero form).
\end{rem}
\par
In fact we are more interested in the time-inhomogeneous version of \eqref{sf}. Let $\mu,\nu\in\cPp$ be given. 
We denote $\cD(\mu,\nu)$ (respectively $\cD_{\cE(V)}(\mu,\nu)$) the set of continuous paths $F\df(F(t))_{t\in[0,1]}$ from $[0,1]$ to $\cV(V)$ (resp.\ $\cE(V)$) such that the 
solution of
\bqn{inmu}
 \dot\mu_t&=&\mu_t K_{\mu_t,F(t)}\eqn
starting with $\mu_0=\mu$  is defined for all $t\in[0,1]$ and satisfies $\mu(1)=\nu$.\par
\begin{rem}\label{Helm}
It was shown in \cite{miclo:hal-04094968} that for any $F\df(F(t))_{t\in[0,1]}$ as above, there exists a unique continuous path $G\df(G(t))_{t\in[0,1]}$ from $[0,1]$ to $\cE(V)$ such that \eqref{inmu} is equivalent to
\bq
 \dot\mu_t&=&\mu_t K_{\mu_t,G(t)}\eq
\end{rem}
Define
\bq
D(\mu,\nu)&\df&\inf_{F\in \cD(\mu,\nu)}\int_0^1 \lVe F(t)\rVe_{\mu_t\ltimes K_{\mu_t}}\, dt\eq
where $(\mu_t)_{t\in[0,1]}$ is the solution of \eqref{inmu} starting from $\mu$. 
By convention, the above infimum should be $+\iy$ when $\cD(\mu,\nu)=\emptyset$. But 
this does not happen, as
it was shown in 
\cite{miclo:hal-04094968}. Thus, up to a regularity assumption on the mapping $K$ of \eqref{Lm}, $D$ is a Riemannian metric on $\cPp$.
Furthermore from Remark \ref{Helm}, we have
\bq
D(\mu,\nu)&=&\inf_{F\in \cD_{\cE(V)}(\mu,\nu)}\int_0^1 \lVe F(t)\rVe_{\mu_t\ltimes K_{\mu_t}}\, dt\eq
\par
This Riemannian structure enables to consider gradient of regular functionals $\cH_{\varphi}$ defined on $\cPp$. Indeed, according to the usual procedure,
$\na_K\cH_{\varphi}(\mu)$ is defined at any $\mu\in\cPp$ as the unique element from $\cE(V)$
such that for any $F\in\cV(V)$,
\bqn{grad}
\lt.\f{d}{dt} \cH_{\varphi}(\mu_t)\rt\vert_{t=0}&=&\lan \na_K\cH_{\varphi}(\mu), F\ran_{\mu\ltimes K_\mu}
\eqn
where $(\mu_t)_{t\geq 0}$ starts with $\mu_0=\mu$ and satisfies \eqref{sf}.
In fact it is sufficient that \eqref{grad} is satisfied for all $F\in \cE(V)$, 
see  \cite{miclo:hal-04094968}, also  for the existence and uniqueness of $\na_K\cH_{\varphi}(\mu)\in\cE(V)$.\par
Once this gradient $\na_K \cH_{\varphi}$ has been defined from $\cPp$ to $\cE(V)$, for any $\mu\in\cPp$, we can consider the gradient descent dynamical system $(\mu_t)_{t\in [0,\tau)}$ starting with $\mu_0=\mu$ and satisfying
\bqn{desc}
\fo t\in[0,\tau),\qquad \dot{\mu}(t)&=&\mu_tK_{\mu_t,-\na_K \cH_{\varphi}(\mu_t)}
\eqn
where $\tau$ is the explosion time of this $\cPp$-valued flow. The interest of this evolution is that it corresponds to the time-marginal distributions of a non-linear Markov process and thus in principe it can be approximated by interacting particle systems, see e.g.\ the book of Del Moral \cite{MR3060209}. It justifies the consideration of Riemannian structures on $\cPp$ derived from mappings of the form \eqref{Lm}, called Markov-Riemann structures in \cite{miclo:hal-04094968}. It was checked there that not all Riemannian structures on $\cPp$ are of this form.
\par\me
Let us give a family of examples that leads to the same traditional Metropolis algorithm.
We begin by recalling the latter. The two ingredients are a generator $L\in \cLi$ reversible with respect to a probability $\ell$, as well as a probability $\pi\in\cPp$. The associated Metropolis generator $L_\pi$ is defined by
\bq
\fo x\neq y \in V,\qquad L_\pi(x,y)&\df& L(x,y)\lt(\f{\pi(y)\ell(x)}{\pi(x)\ell(y)}\wedge 1\rt)\eq
\par
Given an initial probability $\mu\in\cPp$, the Metropolis flow $(\mu_t)_{t\in \RR_+}$ starting with $\mu_0=\mu$ is the solution of the linear evolution equation
\bqn{Metro}
\fo t\geq 0,\qquad 
\dot{\mu}(t)&=&\mu_tL_\pi\eqn
which is defined for all times and satisfies $\lim_{t\ri+\iy}\mu_t=\pi$.\par
In addition, let us be given a smooth and strictly function $\varphi\st \RR_+\ri \RR_+$ satisfying $\varphi(1)=0$ and consider the functional $\cH_{\varphi}$ defined on $\cPp$ via
\bq
\fo \mu\in\cPp,\qquad \cH_{\varphi}(\mu)&\df& \sum_{x\in V}\varphi\lt(\f{\mu(x)}{\pi(x)}\rt)\, \pi(x)\eq
\par
Due to Jensen's inequality and its case of equality, the unique global minimizer of $\cH_{\varphi}$ is $\pi$.
\par
Let us associate to $(L,\pi,\varphi)$ a Markov-Riemann structure on $\cPp$. Consider the function $\theta\st\RR_+^2\ri\RR_+$ defined by
\bqn{theta}
\fo t,s\in\RR_+,\qquad
\theta(t,s)&\df& \lt\{\begin{array}{ll}
\di\f{t-s}{\varphi'(t)-\varphi'(s)}&\hbox{, if $t\neq s$}\\[4mm]
\di\f{1}{\varphi''(t)}&\hbox{, if $t=s$}
\end{array}\rt.
\eqn
and the mapping \eqref{Lm} given by
\bqn{Lm2}
\fo \mu\in\cPp,\,\fo x\neq y\in V,\qquad K_\mu(x,y)&\df&\f{\ell(x)}{\mu(x)} L(x,y) \lt(\f{\pi}{\ell}(x)\wedge \f{\pi}{\ell}(y)\rt)\theta\lt(\f{\mu(x)}{\pi(x)},\f{\mu(y)}{\pi(y)}\rt)
\eqn
\par
Its interest is:
\begin{proposition}
Whatever the choice of the convex function $\varphi$, the gradient descent associated to $\cH_{\varphi}$ in the Markov-Riemann structure coming from \eqref{Lm2}
is the Metropolis flow \eqref{Metro}.
\end{proposition}
Thus the global minimization of the functional $\cH_{\varphi}$ via a gradient descent in the Riemannian structure coming from \eqref{Lm2} does not enable us to deduce a new stochastic algorithm.
On the other side, it appears that all the above $\cH_\varphi$  serve as Liapounov functions for the Metropolis algorithm and their investigations can be performed as part of the general theory of gradient descent and \lojasiewicz' inequalities, see e.g.\  Blanchet and Bolte \cite{zbMATH06912912}. It would be interesting to study more thoroughly the role of the Riemannian metric, for instance what can be said when in \eqref{theta} when one chooses another convex function than $\varphi$? Are there Riemannian structures insuring a faster convergence?
\par
\proof
We begin by computing $\na_K\cH_{\varphi}(\mu)$ for any given $\mu\in\cPp$. Let $F$ be a form and consider $(\mu_t)_{t\in[0,\tau(F,\mu))}$ the solution of
\eqref{sf}  starting from $\mu$.
We compute
\bq
\lt.\f{d}{dt} \cH_{\varphi}(\mu_t)\rt\vert_{t=0}&=& \sum_{x\in V}\varphi'\lt(\f{\mu_t(x)}{\pi(x)}\rt)\,\lt. \f{d\mu_t}{dt}(x)\rt\vert_{t=0}\\
&=&\mu_0\lt[K_{\mu_0,F}\lt[\varphi'\lt(\f{\mu_0}{\pi}\rt)\rt]    \rt]\\
&=&\lan \dd\lt[\varphi'\lt(\f{\mu}{\pi}\rt)\rt] ,F\ran_{\mu\ltimes K_\mu}
\eq
where we used \eqref{fond}, showing that 
\bq
\na_K\cH_{\varphi}(\mu)&=&\dd\lt[\varphi'\lt(\f{\mu}{\pi}\rt)\rt]\eq
\par
We deduce that for any $\mu\in\cPp$, and any $x,y\in V$,
\bq
K_{\mu, -\na_K\cH_{\varphi}(\mu)}(x,y)&=&K_\mu(x,y)\lt(-\varphi'\lt(\f{\mu}{\pi}(y)\rt)+\varphi'\lt(\f{\mu}{\pi}(x)\rt)\rt)_+\\
&=&\f{\ell(x)}{\mu(x)} L(x,y) \lt(\f{\pi}{\ell}(x)\wedge \f{\pi}{\ell}(y)\rt)\theta\lt(\f{\mu(x)}{\pi(x)},\f{\mu(y)}{\pi(y)}\rt)\lt(\varphi'\lt(\f{\mu}{\pi}(y)\rt)-\varphi'\lt(\f{\mu}{\pi}(x)\rt)\rt)_-\\
&=&\f{\ell(x)}{\mu(x)} L(x,y) \lt(\f{\pi}{\ell}(x)\wedge \f{\pi}{\ell}(y)\rt)\lt(\f{\mu}{\pi}(y)-\f{\mu}{\pi}(x)\rt)_-\\
&=&\f{\pi(x)}{\mu(x)} L_\pi(x,y)\lt(\f{\mu}{\pi}(y)-\f{\mu}{\pi}(x)\rt)_-\\
&=&\f{\pi(x)}{\mu(x)} L_{\pi,-\dd[\mu/\pi]}(x,y)
\eq
\par
It follows that for any test function $f\in\cF(V)$,
\bq
\mu[K_{\mu, -\na_K\cH_{\varphi}(\mu)}[f]]&=&
\mu\lt[\f{\pi}{\mu} L_{\pi,-\dd[\mu/\pi]}[f]\rt]\\&=&
\pi
\lt[ L_{\pi,-\dd[\mu/\pi]}[f]\rt]\\
&=&-\lan \dd f, \dd[\mu/\pi]\ran_{\pi\ltimes L_\pi}\\
&=&-\f12\sum_{x,y\in V} (f(y)-f(x))\lt(\f{\mu}{\pi}(y)-\f{\mu}{\pi}(x) \rt) \pi(x)L_\pi(x,y)\\
&=&\sum_{x,y\in V} (f(y)-f(x))\f{\mu}{\pi}(x) \pi(x)L_\pi(x,y)\\
&=&\mu[L_\pi[f]]
\eq
where in the second equality we used \eqref{fond}, but with the irreducible generator $L_\pi$ reversible with respect to $\pi$ instead of $K_\mu$ and $\mu$.
\par
These computations show that \eqref{desc} reduces to \eqref{Metro} as desired.\wwtbp
\par
Note that when $\varphi=\varphi_1$ and  $\pi$ is the Gibbs distribution given by
\bq
\fo x\in V,\qquad \pi(x)&=&\f{\exp(-\beta U(x))}{Z_\beta}\ell(x)\eq
where $Z_\beta$ is the normalizing constant, then $\cH_{\varphi} =\cU_\beta$, the functional considered in \eqref{Ub}. Nevertheless in this case we don't recover the non-linear flow investigated in the main text, because the Riemannian structure is different: 
there the mapping \eqref{Lm} is rather given by
\bq
\fo \mu\in\cPp,\,\fo x\neq y\in V,\qquad K_\mu(x,y)&\df&\f{1}{\mu(x)} L(x,y) \theta\lt({\mu(x)},{\mu(y)}\rt)
\eq
\par

\section{Sampling finite Markov processes}

Our purpose here is to recall how to sample time-homogeneous and time-inhomogeneous Markov processes, as well as sample in an approximate manner  time-homogenous and time-inhomogeneous non-linear Markov processes.

\subsection{Time-homogeneous cases}\label{thc}

Let $L\df(L(x,y))_{x,y\in S}$ be a Markov generator on the finite set $S$ and $m_0\df(m_0(x))_{x\in S}$ be a probability distribution on $S$.
A Markov process $X\df(X(t))_{t\geq 0}$ with initial law $m_0$ and whose generator is $L$ can be sampled in the following way.\par
\begin{itemize}
\item
Sample $X(0)$ according to $m_0$.
\item Sample an exponential random variable $E_1$ of parameter 1 and define $\tau_1\df E_1/\vert L(X(0),X(0))\vert$. For $t\in (0, \tau_1)$, take $X(t)\df X(0)$.
If $L(X(0),X(0))=0$, we have $\tau_1=+\iy$ and the construction stops here. Otherwise we proceed to the next step.
\item Sample $X(\tau_1)$ according to the probability $L(X(0),\cdot)/\vert L(X(0),X(0))\vert$.
\item Sample an exponential random variable $E_2$ of parameter 1 and define $\tau_2\df \tau_1 +E_2/\vert L(X(\tau_1),X(\tau_1))\vert$. For $t\in (\tau_1, \tau_2)$, take $X(t)\df X(\tau_1)$.
If $L(X(\tau_1),X(\tau_2))=0$, we have $\tau_2=+\iy$ and the construction stop here. Otherwise we proceed to the next step.
\item Sample $X(\tau_2)$ according to the probability $L(X(\tau_1),\cdot)/\vert L(X(\tau_1),X(\tau_1))\vert$.
\end{itemize}
In these constructions the samplings are  implicitly independent from the previous steps (this will also be so in the following constructions).
\par
The construction proceeds iteratively, to get $\tau_3$, $X(\tau_3)$, $\tau_4$, $X(\tau_4)$, ... If it happens that for some $n\in \NN$, $L(X(\tau_n),X(\tau_n))=0$, then we get $\tau_{n+1}=+\iy$ and the construction stops there. Otherwise, we obtain an infinite sequence of jump times $(\tau_n)_{n\in\NN}$ with
\bq
\lim_{n\ri\iy} \tau_n&=&+\iy\eq
\par
For $t> 0$, let $m_t$ be  the law of $X(t)$. It is the solution of the evolution equation (starting from $m_0$)
\bq
\fo t\geq 0,\qquad \pa_tm_t&=& m_tL
\eq
(where $m_t$ is seen as a row vector).

\subsection{Time-inhomogeneous cases}\label{tic}

The Markov generator $L$ is replaced by a (measurable and locally integrable) family  $(L_t)_{t\geq 0}$.  A time-inhomogenous Markov process $X\df(X(t))_{t\geq 0}$ with initial law $m_0$ and whose generators are given by $(L_t)_{t\geq 0}$ can be sampled in the following way.\par
\begin{itemize}
\item
Sample $X(0)$ according to $m_0$.
\item Sample an exponential random variable $E_1$ of parameter 1 and define $\tau_1$ as 
\bq
\tau_1&\df&\inf\lt\{ t>0\st \int_0^t \vert L_s(X(0),X(0))\vert \,ds =E_1\rt\}\eq
For $t\in (0, \tau_1)$, take $X(t)\df X(0)$.
If  $\tau_1=+\iy$ the construction stops here. Otherwise we proceed to the next step.
\item Sample $X(\tau_1)$ according to the probability $L_{\tau_1}(X(0),\cdot)/\vert L_{\tau_1}((X(0),X(0))\vert$.
\item Sample an exponential random variable $E_2$ of parameter 1 
and define $\tau_2$ as
\bq
\tau_2&\df&\inf\lt\{ t>0\st \int_{\tau_1}^{\tau_1+t} \vert L_s(X(0),X(0))\vert \,ds =E_2\rt\}\eq
 For $t\in (\tau_1, \tau_2)$, take $X(t)\df X(\tau_1)$.
If  $\tau_2=+\iy$  the construction stop here. Otherwise we proceed to the next step.
\item Sample $X(\tau_2)$ according to the probability $L_{\tau_2}(X(\tau_1),\cdot)/\vert L_{\tau_2}(X(\tau_1),X(\tau_1))\vert$.
\end{itemize}
\par
The construction proceeds iteratively, to get $\tau_3$, $X(\tau_3)$, $\tau_4$, $X(\tau_4)$, ... 
This procedure may end in a finite number of steps if it happens that for some $n\in\NN$ we get $\tau_n=+\iy$.
On the contrary when the whole sequence $(\tau_n)_{n\in\NN}$ of (finite) jump times is defined, 
consider 
\bq
\tau_\iy&\df& \lim_{n\ri\iy} \tau_n\eq
definition which is extended to the case where there are only a finite number of jumps by taking $\tau_\iy=+\iy$.
\par
It can be shown that under our local integrability assumption, namely
\bqn{li}
\fo t\geq 0,\qquad \int_0^t \max( \vert L_s(x,x)\vert\st x\in S)\, ds&<&+\iy\eqn
we have $\tau_\iy=+\iy$ (a.s.). In particular this is satisfied if the mapping $\RR_+\ni t\mapsto L_t$ is bounded.  \par
If Condition \eqref{li} is removed (but keeping the mesurability assumption), 
it may happen
that $\tau_\iy<+\iy$, in which case $\tau_\iy$ is called an explosion time.
Then $X\df (X(t))_{t\in [0,\tau_\iy)}$ is  only defined on the (random) interval $[0,\tau_\iy)$. 
\par
For $t> 0$, let $m_t$ be  the law of $X(t)$. It is the solution of the time-inhomogeneous evolution equation (starting from $m_0$)
\bq
\fo t\geq 0,\qquad \pa_tm_t&=& m_tL_t
\eq
\par
Note that the above construction coincides with that of Section \ref{thc}, in the time-homogeneous cases where $L_t$ does not depend on $t\in\RR_+$.
In truly time-inhomogeneous cases, one should be able to compute the inverse of the mapping
\bq
(0,+\iy)\ni t&\mapsto &\int_s^{s+t} \vert L_u(x,x)\vert \, du \eq
(where $s\geq 0$ and $x\in S$ are given),
which suggests to rather consider simple mappings $\RR_+\ni t\mapsto L_t$.

\subsection{Non-linear cases}

Let $\cP(S)$ be the set of probability measures on $S$ and $\cG(S)$ be the set of Markov generators on $S$.
Consider a Lipschitzian mapping
\bq
\cP(S)\ni m&\mapsto& L_m\in\cG(S)\eq
\par
Given an initial  probability distribution $m_0\in\cP(S)$,
we are interested in the solution $(m_t)_{t\geq 0}$ of the non-linear evolution
\bq
\fo t\geq 0,\qquad \pa_t m_t&=&m_t L_{m_t}\eq
\par
It is not easy in general to sample directly a Markov process $X\df(X(t))_{t\geq 0}$ such that at any time $t\geq 0$, $m_t$ is the law of $X(t)$ and the instantaneous generator is $L_{m_t}$.
A probabilistic approximation goes through systems of interacting particles.
\par
Let $N\in\NN$ be a number of evolving particles, denoted $X_N\df(X_{N,l})_{l\in\lin N\rin}\df (X_{N,l}(t))_{l\in\lin N\rin, t\geq 0}$.
The process $X_N$ is Markovian on $S^N$ and its generator $L_N$ is such that
\bq
L_N&\df&\sum_{l\in\lin N\rin} L_{N,l}\eq
where for any $l\in \lin N\rin$, $L_{N,l}$ is the Markov generator on $S^N$ given by
\bq
\lefteqn{\hskip-30mm \fo x\df(x_k)_{k\in\lin N\rin}\neq y\df(y_k)_{k\in\lin N\rin}\in S^N,}\\
L_{N,l}(x,y)&\df&\lt\{ \begin{array}{ll}L_{\eta(x)}(x_l,y_l)&\hbox{, if $x_k=y_k$ for all $k\in\lin N\rin\setminus\{l\}$}\\
0&\hbox{, otherwise}
\end{array}\rt.
\eq
where for any $ x\df(x_k)_{k\in\lin N\rin}\in S^N$, $\eta(x)$ stands for the empirical measure
\bqn{eta}
\eta(x)&\df& \f1N\sum_{l\in\lin N\rin} \delta_{x_l}\ \in\ \cP(S)\eqn
\par
Assume furthermore that the law of $X_N(0)$ is $m_0^{\otimes N}$.\par
For large $N$, the process $X_{N,1}$ (or any $X_{N,l}$ with $l\in \lin N\rin$) is an approximation of $X$ and 
$(\eta(X_N(t)))_{t\geq 0}$ is a random approximation of $(m_t)_{t\geq 0}$.
\par
The Markov process $X_N$ can be sampled as described in Section \ref{thc}. Taking into account that if $\cE_1, \cE_2, ..., \cE_N$ are $N$ independent exponential random variables of respective parameters $\lambda_1, \lambda_2, ..., \lambda_N\geq 0$, then $\min(\cE_l, l\in\lin N\rin)$ is an exponential random variable of parameter $\lambda_1+\lambda_2+\cdots +\lambda_N$, we get the following alternative description of the procedure:
\begin{itemize}
\item
Sample $X_N(0)$ according to $m_0^{\otimes N}$.
\item Sample $N$ independent  exponential random variables $E_{1,1}, E_{1,2}, ..., E_{1,N}$ of  parameter 1, define $\tau_1$ as 
\bq
\tau_1&\df&\min\lt(\f{E_{1,l}}{\vert L_{\eta(X_N(0))}(X_{N,l}(0), X_{N,l}(0))\vert}\st l\in\lin N\rin\rt)\eq
and call $I_1$ the index where the minimum is attained (which is a.s.\ unique if $\tau_1<+\iy$).
For $t\in (0, \tau_1)$, take $X_N(t)\df X_N(0)$.
If  $\tau_1=+\iy$ the construction stops here. Otherwise we proceed to the next step.
\item Sample $X_{N,I_1}(\tau_1)$ according to the probability $L_{\eta(X_N(0))}(X_{N,I_1}(0),\cdot)/\vert L_{\eta(X_N(0))}(X_{N,I_1}(0),X_{N,I_1}(0))\vert$.
\item Keep the other coordinates: for $l\neq I_1$, take $X_{N,l}(\tau_1)\df X_{N,l}(0)$, this ends the construction of $X_N(\tau_1)$.
\item Sample $N$ independent  exponential random variables $E_{2,1}, E_{2,2}, ..., E_{2,N}$ of  parameter 1, define $\tau_2$ as 
\bq
\tau_2&\df&\tau_1+\min\lt(\f{E_{1,l}}{\vert L_{\eta(X_N(\tau_1))}(X_{N,l}(\tau_1), X_{N,l}(\tau_1))\vert}\st l\in\lin N\rin\rt)\eq
and call $I_2$ the index where the minimum is attained (which is a.s.\  unique if $\tau_2<+\iy$).
For $t\in (\tau_1,\tau_2)$, take $X_N(t)\df X_N(\tau_1)$.
If  $\tau_2=+\iy$ the construction stops here. Otherwise we proceed to the next step.
\item Sample $X_{N,I_2}(\tau_2)$ according to the probability $L_{\eta(X_N(\tau_1))}(X_{N,I_1}(\tau_1),\cdot)/\vert L_{\eta(X_N(\tau_1))}(X_{N,I_1}(\tau_1),$\linebreak $X_{N,I_1}(\tau_1))\vert$.
\item Keep the other coordinates: for $l\neq I_2$, take $X_{N,l}(\tau_2)\df X_{N,l}(\tau_1)$, this ends the construction of $X_N(\tau_2)$.
\end{itemize}
\par
The construction proceeds iteratively, to get $\tau_3$, $X_N(\tau_3)$, $\tau_4$, $X_N(\tau_4)$, ... The construction may stop in a finite number of iteration(s), if it happens that $\tau_n=+\iy$ for some $n\in\NN$. Otherwise, we obtain an infinite sequence of jump times $(\tau_n)_{n\in\NN}$ with
\bq
\lim_{n\ri\iy} \tau_n&=&+\iy\eq
\par

\subsection{Non-linear and time-inhomogeneous cases}

Consider a mapping
\bq
\RR_+\times\cP(S)\ni (t,m)&\mapsto& L_{t,m}\in\cG(S)\eq
which is locally integrable in the first variable (uniformly with respect to the second) and  Lipschitzian in the second variable (locally uniformly with respect to the first variable).
\par
Given an initial  probability distribution $m_0\in\cP(S)$,
we are interested in the solution $(m_t)_{t\geq 0}$ of the non-linear evolution
\bq
\fo t\geq 0,\qquad \pa_t m_t&=&m_t L_{t,m_t}\eq
\par
It is not easy in general to sample directly a Markov process $X\df(X(t))_{t\geq 0}$ such that at any time $t\geq 0$, $m_t$ is the law of $X(t)$ and the instantaneous generator is $L_{t,m_t}$.
A probabilistic approximation goes through systems of interacting particles.
\par
Let $N\in\NN$ be a number of evolving particles, denoted $X_N\df(X_{N,l})_{l\in\lin N\rin}\df (X_{N,l}(t))_{l\in\lin N\rin, t\geq 0}$.
The process $X_N$ is Markovian on $S^N$, but time-inhomogeneous, and its instantaneous generator $L_{t,N}$ at time $t\geq 0$ is such that
\bq
L_{t,N}&\df&\sum_{l\in\lin N\rin} L_{t,N,l}\eq
where for any $l\in \lin N\rin$, $L_{t,N,l}$ is the Markov generator on $S^N$ given by
\bq
\fo x\df(x_k)_{k\in\lin N\rin}\neq y\df(y_k)_{k\in\lin N\rin}\in S^N,\qquad L_{t,N,l}(x,y)&\df&\lt\{ \begin{array}{ll}L_{t,\eta(x)}(x_l,y_l)&\hbox{, if $x_k=y_k$ for all $k\in\lin N\rin\setminus\{l\}$}\\
0&\hbox{, otherwise}
\end{array}\rt.
\eq
and where for any $ x\df(x_k)_{k\in\lin N\rin}\in S^N$, $\eta(x)$ is still given by \eqref{eta}.
\par
Assume furthermore that the law of $X_N(0)$ is $m_0^{\otimes N}$.\par
For large $N$, the process $X_{N,1}$ (or any $X_{N,l}$ with $l\in \lin N\rin$) is an approximation of $X$ and 
$(\eta(X_N(t)))_{t\geq 0}$ is a random approximation of $(m_t)_{t\geq 0}$.
\par
The Markov process $X_N$ can be sampled as described in Section \ref{tic}. Here is an alternative description of the procedure:
\begin{itemize}
\item
Sample $X_N(0)$ according to $m_0^{\otimes N}$.
\item Sample $N$ independent  exponential random variables $E_{1,1}, E_{1,2}, ..., E_{1,N}$ of  parameter 1, define $\tau_1$ as 
\bqn{tt}
\tau_1&\df& \min(\tau_{1,l}\st l\in\lin N\rin)\eqn
with
\bq
\fo l\in\lin N\rin,\qquad \tau_{1,l}&\df&\inf\lt\{ t>0\st \int_{0}^{t} \vert L_{s,\eta(X_N(0))}(X_{N,l}(0),X_{N,l}(0))\vert \,ds =E_{1,l}\rt\}
\eq
and call $I_1$ the index where the minimum is attained in \eqref{tt} (which is a.s.\ unique if $\tau_1<+\iy$).
For $t\in (0, \tau_1)$, take $X_N(t)\df X_N(0)$.
If  $\tau_1=+\iy$ the construction stops here. Otherwise we proceed to the next step.
\item Sample $X_{N,I_1}(\tau_1)$ according to the probability $L_{\tau_1, \eta(X_N(0))}(X_{N,I_1}(0),\cdot)/\vert L_{\tau_1, \eta(X_N(0))}(X_{N,I_1}(0),$
\linebreak $X_{N,I_1}(0))\vert$.
\item Keep the other coordinates: for $l\neq I_1$, take $X_{N,l}(\tau_1)\df X_{N,l}(0)$, this ends the construction of $X_N(\tau_1)$.
\item Sample $N$ independent  exponential random variables $E_{2,1}, E_{2,2}, ..., E_{2,N}$ of  parameter 1, define $\tau_2$ as 
\bqn{ttt}
\tau_2&\df& \min(\tau_{2,l}\st l\in\lin N\rin)\eqn
with
\bq
\fo l\in\lin N\rin,\qquad \tau_{2,l}&\df&\inf\lt\{ t>0\st \int_{\tau_1}^{\tau_1+t} \vert L_s(X_{N,l}(\tau_1),X_{N,l}(\tau_1))\vert \,ds =E_{2,l}\rt\}
\eq
and call $I_2$ the index where the minimum is attained in \eqref{ttt} (which is a.s.\ unique if $\tau_2<+\iy$).
For $t\in ( \tau_1,\tau_2)$, take $X_N(t)\df X_N(\tau_1)$.
If  $\tau_2=+\iy$ the construction stops here. Otherwise we proceed to the next step.
\item Sample $X_{N,I_2}(\tau_2)$ according to the probability $L_{\tau_2, \eta(X_N(\tau_1))}(X_{N,I_2}(\tau_1),\cdot)/\vert L_{\tau_2, \eta(X_N(\tau_1))}(X_{N,I_2}(\tau_1),$
\linebreak $X_{N,I_2}(\tau_1))\vert$.
\item Keep the other coordinates: for $l\neq I_2$, take $X_{N,l}(\tau_2)\df X_{N,l}(\tau_1)$, this ends the construction of $X_N(\tau_2)$.
\end{itemize}
\par
The construction proceeds iteratively, to get $\tau_3$, $X_N(\tau_3)$, $\tau_4$, $X_N(\tau_4)$, ... The construction may stop in a finite number of iteration(s), if it happens that $\tau_n=+\iy$ for some $n\in\NN$. Otherwise, we obtain an infinite sequence of jump times $(\tau_n)_{n\in\NN}$ with
\bq
\lim_{n\ri\iy} \tau_n&=&+\iy\eq
\par

    \vskip2cm
\hskip70mm
\vbox{
\copy5
}

    \newpage
    \let\clearpage\relax

\end{document}